\documentclass[11pt]{amsart} 
\usepackage{amsthm,amsbsy,amsmath,amssymb,amscd,amsfonts,array,mathrsfs,verbatim,enumerate,xypic,enumitem,color}
\usepackage[all]{xy}
\xyoption{arc} 

\newtheorem{thm}{Theorem}[subsection]

\newtheorem{prop}[thm]{Proposition}     
\newtheorem{lem}[thm]{Lemma}
\newtheorem{cor}[thm]{Corollary}

\theoremstyle{definition}

\newtheorem{defn}[thm]{Definition}

\newtheorem{example}[thm]{Example} 
\newtheorem{rem}[thm]{Remark}
\newtheorem{notation}[thm]{Notation}
\newtheorem{setup}[thm]{Setup}
\newtheorem{const}[thm]{Construction}

\DeclareFontFamily{OT1}{rsfs}{}
\DeclareFontShape{OT1}{rsfs}{n}{it}{<-> rsfs10}{}
\DeclareMathAlphabet{\curly}{OT1}{rsfs}{n}{it}

\makeatletter \@addtoreset{equation}{subsection} \makeatother  
\makeatletter \@addtoreset{thm}{subsection} \makeatother  

\makeatletter \@addtoreset{equation}{subsection} \makeatother  
\makeatletter \@addtoreset{thm}{subsection} \makeatother  

\newcommand{\Esp}{{\bf Esp}} 
\newcommand{\DAb}{{\bf D(Ab)}} 
\newcommand{\Alg}{{\bf Alg}} 
\newcommand{\Ab}{{\bf Ab}} 
\newcommand{\Sets}{{\bf Sets}} 
\newcommand{\Mon}{{\bf Mon}} 
\newcommand{\DS}{{\bf DS}}   
\newcommand{\LRS}{{\bf LRS}}  
\newcommand{\RS}{{\bf RS}}  
\newcommand{\Sites}{{\bf Sites}} 
\newcommand{\Topoi}{{\bf Topoi}} 
\newcommand{\LogSch}{{\bf LogSch}} 
\newcommand{\Sh}{{\bf Sh}} 
\newcommand{\St}{{\bf St}} 
\newcommand{\Sch}{{\bf Sch}} 
\newcommand{\LMS}{{\bf LMS}} 
\newcommand{\Man}{{\bf Man}} 
\newcommand{\MS}{{\bf MS}}   
\newcommand{\Top}{{\bf Top}}   
\newcommand{\Fans}{{\bf Fans}} 
\newcommand{\LogDS}{{\bf LogDS}} 
\newcommand{\KMFans}{{\bf KMFans}} 
\newcommand{\GSFans}{{\bf GSFans}} 

\renewcommand{\AA}{\mathbb{A}} 
\newcommand{\BB}{\mathbb{B}} 
\newcommand{\CC}{\mathbb{C}} 
\newcommand{\DD}{\mathbb{D}} 
\newcommand{\EE}{\mathbb{E}} 
\newcommand{\FF}{\mathbb{F}} 
\newcommand{\GG}{\mathbb{G}} 
\newcommand{\NN}{\mathbb{N}} 
\newcommand{\PP}{\mathbb{P}} 
\newcommand{\QQ}{\mathbb{Q}} 
\newcommand{\RR}{\mathbb{R}} 
\newcommand{\UU}{\mathbb{U}} 
\newcommand{\ZZ}{\mathbb{Z}} 

\newcommand{\latticedatalimit}{\Lambda} 

\newcommand{\m}{\mathfrak{m}}

\newcommand{\C}{\mathcal{C}}  
\newcommand{\D}{\mathcal{D}}

\newcommand{\M}{\mathcal{M}} 
\newcommand{\N}{\mathcal{N}}
\renewcommand{\O}{\mathcal{O}} 

\renewcommand{\u}{\underline}

\newcommand{\ov}{\overline}
\newcommand{\into}{\hookrightarrow}
\newcommand{\be}{\begin{eqnarray*}}
\newcommand{\ee}{\end{eqnarray*}}

\newcommand{\bne}[1]{\begin{eqnarray} \label{#1}}
\newcommand{\ene}{\end{eqnarray}}
\newcommand{\xym}{\xymatrix}
\newcommand{\bp}{\begin{pmatrix}}
\newcommand{\ep}{\end{pmatrix}}
\newcommand{\slot}{ \hspace{0.05in} {\rm \_} \hspace{0.05in} } 
\newcommand{\dirlim}{\displaystyle \lim_{ \longrightarrow } \,} 

\newcommand{\Hom}{\operatorname{Hom}}   
\newcommand{\RHom}{\operatorname{RHom}}   
   
\newcommand{\Ext}{\operatorname{Ext}}
\renewcommand{\H}{\operatorname{H}}

\newcommand{\Aut}{\operatorname{Aut}}

\newcommand{\Supp}{\operatorname{Supp}}  

\newcommand{\Ker}{\operatorname{Ker}}
\newcommand{\Cok}{\operatorname{Cok}}

\newcommand{\Id}{\operatorname{Id}}
\newcommand{\Spec}{\operatorname{Spec}}

\newcommand{\Span}{\operatorname{Span}}

\newcommand{\Star}{\operatorname{Star}}

\def\arraystretch{1.2} 

\numberwithin{equation}{subsection}

\begin{document}

\author{W.~D.~Gillam}
\address{Department of Mathematics, Bogazici University, Istanbul, Turkey 34342}
\email{wdgillam@gmail.com}

\author{Sam Molcho}
\address{Department of Mathematics, University of Colorado Boulder, Boulder, Colorado 80309}
\email{Samouil.Molcho@colorado.edu}

\date{\today}
\title[A theory of stacky fans]{A theory of stacky fans}

\begin{abstract}  We study the category of \emph{KM fans}---a ``stacky" generalization of the category of fans considered in toric geometry---and its various realization functors to ``geometric" categories.  The ``purest" such realization takes the form of a functor from KM fans to the $2$-category of stacks over the category of fine fans, in the ``characteristic-zero-\'etale" topology.  In the algebraic setting, over a field of characteristic zero, we have a realization functor from KM fans to (log) Deligne-Mumford stacks.  We prove that this realization functor gives rise to an equivalence of categories between (lattice) KM fans and an appropriate category of toric DM stacks.  Finally, we have a differential realization functor to the category of (positive) log differentiable spaces.  Unlike the other realizations, the differential realization of a stacky fan is an ``actual" log differentiable space, not a stack.  Our main results are generalizations of ``classical" toric geometry, as well as a characterization of ``when a map of KM fans is a torsor".  The latter is used to explain the relationship between our theory and the ``stacky fans" of Geraschenko and Satriano.  \end{abstract}

\maketitle

\arraycolsep=2pt 
\def\arraystretch{1.2}

\setcounter{tocdepth}{2}
\tableofcontents

\newpage

\section*{Introduction}

The purpose of this paper is to introduce and study some ``new" combinatorial objects, called \emph{KM fans} (\S\ref{section:definitions}), generalizing the ``classical" \emph{fans} of toric geometry (see \S\ref{section:conesandfans}, \cite{F}, \cite{O}, \cite{CLS}).  A KM fan consists of a ``classical" fan $F$ (except we allow the ``lattice" $N$ to be an arbitrary finitely generated abelian group), together with a lattice $F_\sigma$ of finite index in $N \cap \Span \sigma$ for each cone $\sigma$ (we call these $F_\sigma$ ``lattice data"), satisfying an obvious compatibility condition.  Several similar notions of ``stacky fans" have been previously considered in the literature (see below), but we believe that our theory is both the most fully-developed and the easiest to work with.

Although one \emph{could} study our theory of KM fans from a purely combinatorial point of view, the fullest picture can be obtained only by understanding the (sometimes subtle) relationship between combinatorial properties of KM fans (or maps of KM fans) and geometric properties of their ``realizations."  In classical toric geometry the ``realization" of a fan typically means the associated complex variety, often together with the corresponding torus action.  We find it useful to consider various different ``realization" functors, related to the following \emph{categories of spaces} (see \S\ref{section:categoriesofspaces}): \begin{enumerate} \item the category $\Fans$ of fine fans \item the category $\LogSch$ of fine log schemes over field $k$ of characteristic zero \item the category $\LogDS$ of fine (positive) log differentiable spaces \end{enumerate} Corresponding to these three categories, we have three functors $F \mapsto F$, $F \mapsto X(F)$, and $F \mapsto Y(F)$ out of the category of KM fans.  These ``realization" functors are called the \emph{fan}, \emph{algebraic}, and \emph{differential realization}, respectively.  They take values in the ``category" \begin{enumerate} \item of stacks over $\Fans$ in the \emph{characteristic-zero-\'etale (CZE) topology}. \item of Deligne-Mumford (DM) stacks over $\LogSch$ in the \emph{strict \'etale} topology. \item $\LogDS$. \end{enumerate}  The algebraic realization $X(F)$ of a KM fan $F$ is Zariski-locally a global quotient of an affine toric variety (with its usual log structure) by a finite, diagonalizable group scheme acting through the torus of the toric variety.  We emphasize that the ``differential realization" of a KM fan is an actual log differentiable space, \emph{not a stack}.

The beauty and simplicity of our theory can perhaps be most easily appreciated by starting with the example constructions in \S\ref{section:examples}.  It is an interesting exercise to think about the geometric meanings of the realizations of these constructions, most of which are eventually explained later in the paper.

One recurrent theme is that KM fans are an extremely natural generalization of classical fans in which it should be possible to make sense of every result or construction from the classical theory of toric varieties.  Although we have not made an effort to work out generalizations of \emph{every} construction from, say, \cite{CLS}, we suspect that it \emph{would} be possible to do so!  To give the general idea, we have worked out analogues of the following classical results: \begin{enumerate} \item the geometric interpretation of the \emph{support} of a fan in terms of one-parameter ``subgroups" of the corresponding torus (\S\ref{section:supportandproperness}) \item the combinatorial characterization of \emph{proper} maps of toric varieties (\S\ref{section:supportandproperness}) \item \label{starfan} the construction of the \emph{star fan} associated to a cone in a fan, and the interpretation of its realization (\S\ref{section:starfansandstratification}) \item \label{stratification} the stratification of a toric variety by torus orbits (\S\ref{section:starfansandstratification}) \item the description of the fundamental group of a toric variety (\S\ref{section:fundamentalgroup}) \end{enumerate}  

At this point we should comment on the importance of allowing the ``lattice" $N$ for our KM fans to have torsion.  The KM fans where $N$ has torsion are precisely those whose algebraic realization has generically non-trivial isotropy.  Obviously we need KM fans of this type to make sure the strata in \eqref{stratification} are themselves realizations of KM fans.  Indeed, applying the star fan construction to a cone in a \emph{lattice} KM fan (one where $N$ actually \emph{is} a lattice) will typically yield a KM fan where $N$ has torsion.  The point is that ``non-lattice" KM fans will arise from many natural constructions with lattice KM fans (also including those in \cite{GM2}, for example), so a fully satisfactory theory must also take these into account.

We also establish results which have no ``classical analog," such as our characterization of maps of KM fans with representable scheme-theoretic realization (Theorem~\ref{thm:representability}).  Our Theorem~\ref{thm:quotients2}, roughly speaking, characterizes the maps of KM fans whose realizations are torsors.  This result is apparently new even when restricted to the case of classical fans, though here the KM fan generalization actually takes a great deal of work.  We use Theorem~\ref{thm:quotients2} to explain the relationship between our theory and the theory of ``stacky fans" introduced in \cite{GS1} (see below).

As in the classical theory of toric varieties, one is bound to ask questions like: \begin{enumerate} \item What is the essential image of the various realization functors? \item Are these realization functors ``faithful / full" in some sense (for example, when we restrict to some kind of ``torus-equivariant" morphisms)? \end{enumerate}

In \S\ref{section:toricstacks} we prove (Theorem~\ref{thm:main}) that the algebraic realization $X$ yields an equivalence of categories between the category of \emph{lattice} KM fans and the category of ``toric DM stacks" with ``torus equivariant maps" as the morphisms (Definition~\ref{defn:toricstack}).  The proof of this is completely parallel to the classical case and relies on a stacky version of Sumihiro's Theorem explained to us by A.~Geraschenko and M.~Satriano.

In the differential setting (where realization is certainly not faithful), these questions seem to be very subtle.  We are also rather unsure how to ``intrinsically" characterize the stacks obtained by realizing \emph{non-lattice} KM fans, even in the algebraic setting.

\subsection*{Previous work}  There have been many previous attempts in the literature to work out a theory of ``stacky fans" and ``toric stacks" in analogy with the classical theory of fans and toric varieties.  See \cite{GS1} and the long list of references in the first line of its abstract for some of these attempts.  We do not claim any incredible originality in our definition of \emph{KM fan}.  We arrived at this definition in the spring of 2013 at ETH Z\"{u}rich by carefully reading \cite{KM}.  Reading between the lines of \cite{KM}, one can see that Kottke and Melrose (for whom our KM fans are named) are implicitly aware of the differential realization of a KM fan---we came to the idea of a KM fan by noticing that one would have to work with \emph{stacks} to do the sort of ``gluing" which they could so easily perform in the differential setting (because extracting roots of positive smooth functions is such a benign operation in differential geometry, whereas one often needs to form \emph{root stacks} to carry out analogous constructions algebraically).  

When we were first thinking about \cite{GM2}, we also noticed that the ``lattice data" for a KM fan were lurking in the ``Chow quotient" construction of \cite{KSZ} and we knew this extra data should be exploited to form a ``refined Chow quotient" with an interesting modular interpretation.  The possibility of considering both the algebraic and differential realizations of a KM fan has been a major consideration for us from the outset, as we explain in our paper \cite{GM2}.  Indeed, the present paper was largely written in order to be used in \cite{GM2} and \cite{AM}.

After the appearance of \cite{AM}, where the basic notions of this paper first appeared, A.~Geraschenko kindly informed us that I.~Tyomkin also considered something like a ``KM fan" in \cite[Definition~4.1]{Tyo}.  Indeed, Tyomkin's ``toric stacky data" is the same thing as what we call a \emph{lattice KM fan}---he is of course also aware of the algebraic realization of such a thing, which he discusses in \cite[\S4.1]{Tyo}, but says little else about the notion.  We do not deny that a KM fan could justifiably be called a \emph{Tyomkin fan}---we're just too set in our ways to change terminology at this point.

Of the various papers on ``toric stacks," the only one we can claim any great familiarity with is \cite{GS1}, where Geraschenko and Satriano introduce a kind of ``stacky fan," which we call a \emph{GS fan} (Definition~\ref{defn:GSfan}) to avoid confusion.  A GS fan is a fan $F$ in the usual (``classical") sense, together with combinatorial data corresponding to a subgroup $S$ of the torus $T$ for the corresponding toric variety $X=X(F)$.  The (algebraic) realization of such a GS fan is the stack-theoretic quotient $[X / S]$.  We show (Theorem~\ref{thm:folding}) that every GS fan whose (algebraic) realization is a separated DM stack gives rise by ``folding" to a lattice KM fan with the same realization---thus every such GS fan ``is" a lattice KM fan.  

It turns out, however, that not every lattice KM fan arises by ``folding" a GS fan.  This is because the separated DM stack $X=X(F)$ obtained by realizing a lattice KM fan $F$ need not admit a torsor $X' \to X$ whose total space $X'$ is a toric variety (with the torsorial group acting through the torus of $X'$), even though every such $X$ does admit such a torsor Zariski locally.  We prove this by introducing an ``unfolding" construction (\S\ref{section:unfolding}) which yields a torsor over (the realization of) any KM fan which is appropriately ``universal" (Theorem~\ref{thm:unfolding}).  Among other things, we show that if there were an $X' \to X$ as before, then the (rigidified) unfolding of the KM fan $F$ would have to be a classical fan, which, in general, it is not.  Ultimately we obtain a fairly simple combinatorial characterization (Theorem~\ref{thm:KMandGSfans}) of the lattice KM fans that can be obtained by folding a GS fan.

The unfolding construction is really rather simple.  One can think of a KM fan as an abstract set of cones $F$ and lattice data $\{ F_\sigma \}$ (or, better, as an abstract set of monoids $\{ P_\sigma := F_\sigma \cap \sigma \}$) together with various (``face") inclusions between these cones  (or monoids) \emph{together with an embedding of this data into an ambient finitely generated abelian group} $N$.  To unfold a KM fan, we just replace the italicized data with the ``universal" such $N$---i.e.\ the direct limit of the lattice data $\{ F_\sigma \}$.  An interesting feature of this construction is that it can yield a non-classical KM fan (even a non-lattice KM fan) when applied to a classical fan.

It should be mentioned that in an early version of \cite{GS2}, Geraschenko and Satriano purported to prove a general result characterizing certain (not necessarily Deligne-Mumford) ``toric stacks" which would imply, in particular, that every ``toric DM stack" in our sense (Definition~\ref{defn:toricstack}) is the realization of a GS fan.  This is not true; we expect that this mistake will be corrected in a revision of \cite{GS2}.  The ``local analysis" of toric stacks performed in that version of \cite{GS2} is entirely correct---we make use of it to prove our ``stacky Sumihiro Theorem" (Lemma~\ref{lem:Sumihiro})---but they make a small mistake in passing from their local description of toric stacks to a global one.

\subsection*{Acknowledgements} We would like to thank A.~Geraschenko and M.~Satriano for explaining to us how their results from \cite{GS2} yield a proof of``Sumihiro's theorem" for toric DM stacks (Lemma~\ref{lem:Sumihiro} here). W.D.G.\ was supported by a Marie Curie/T\"{U}BITAK Co-Funded Brain Circulation Scheme fellowship.

\subsection*{Conventions and notation} Throughout this paper, \emph{monoid} means \emph{commutative monoid with zero} and \emph{FGA group} means \emph{finitely generated abelian group}.  For an abelian group $N$, we set $N^\lor := \Hom(N,\ZZ)$, $\EE(N) := \Ext^1(N,\ZZ)$.

\newpage

\section{Categories of spaces} \label{section:categoriesofspaces} The purpose of this section is to give a brief review of the categories $\Fans$, $\LogSch$, and $\LogDS$ mentioned in the introduction.  We also review the \emph{CZE topology} on $\Fans$ and the \emph{strict-\'etale topology} on $\LogSch$.  It is necessary to say something about these topologies because we will want to consider sheaves on (and stacks over) them---these notions depend on a choice of topology.  The resulting $2$-categories of stacks will be the targets of various realization functors from the category of KM fans.  Although this section is very brief, we have tried to provide ample references for further reading.  Readers familiar with the classical theory of toric varieties and who are interested primarily in the basic theory of KM fans and their algebraic realizations can probably skip to \S\ref{section:KMfans}.

\subsection{Fans and the CZE topology}  \label{section:fans}  In this section we review the category of (abstract) fans, which is to the category of monoids what the category of schemes is to the category of rings.  A thorough study of this category can be found in \cite{G2}.

\begin{defn} \label{defn:fine} A monoid is called \emph{fine} iff it is isomorphic to a finitely generated submonoid of a group. \end{defn}

\begin{defn} \label{defn:face} A subset $I$ (possibly empty!) of a monoid $P$ is called an \emph{ideal} iff $P+I = I$.  A submonoid $F$ of a monoid $P$ is called a \emph{face} (notation: $F \leq P$) iff $P \setminus F$ is an ideal of $P$.  An ideal of $P$ whose complement is a face is called a \emph{prime ideal}, so that taking complements defines a bijection between the set of prime ideals of $P$ and the set of faces of $P$. \end{defn}

The set of faces of a monoid $P$ is denoted $\Spec P$.  Every monoid $P$ has a smallest face $P^*$ (the units of $P$) and a largest face $P$.  The monoid $\ov{P} := P / P^*$ is called the \emph{sharpening} of $P$.  The set $\Spec P$ is finite whenever $P$ is finitely generated because, in general, if $S \subseteq P$ generates $P$ then $S \cap F$ generates $F$ for any face $F \leq P$ (in particular, any face of a finitely generated monoid is also finitely generated), hence $F \mapsto S \cap F$ is an injection from $\Spec P$ to the powerset of $S$.  

We topologize $\Spec P$ and equip it with a sheaf of monoids $\M_P$ in a manner completely analogous to the construction of the locally ringed space $\Spec A$ associated to a (commutative) ring $A$.  The stalk $\M_{P,F}$ of $\M_P$ at a face $F \in \Spec P$ is the localization $F^{-1}P$ of $P$ at $F$.  As with rings, the $\Spec$ construction yields a fully faithful embedding from $\Mon^{\rm op}$ to the category of locally monoidal spaces $\LMS$, whose objects are pairs $X=(X,\M_X)$ consisting of a topological space $X$ and a sheaf $\M_X$ of monoids on $X$, and whose morphisms $(X,\M_X) \to (Y,\M_Y)$ are pairs consisting of a map of topological spaces $f : X \to Y$ and a map $f^\sharp : f^{-1} \M_Y \to \M_X$ of sheaves of monoids on $X$ which is \emph{local} in the sense that each stalk $f^\sharp_x : \M_{Y,f(x)} \to \M_{X,x}$ satisfies $(f^\sharp_x)^{-1}(\M_{X,x}^*) = \M_{Y,f(x)}^*$.  (The category $\MS$ of \emph{monoidal spaces} has the same objects as $\LMS$, but the locality condition is dropped in defining a morphism in $\MS$, so that $\LMS \into \MS$ is not full.)  As in the analogous ring-theoretic situation, the locally monoidal space $\Spec P$ represents the functor \be \LMS^{\rm op} & \to & \Sets \\ X & \mapsto & \Hom_{\Mon}(P,\M_X(X)). \ee

\begin{defn} \label{defn:abstractfan} An \emph{affine fan} (resp.\ a \emph{fine affine fan}) is a locally monoidal space isomorphic in $\LMS$ to $\Spec P$ for some monoid (resp.\ fine monoid) $P$.  A \emph{fan} (resp.\ \emph{fine fan}) is an object $(X,\M_X)$ of $\LMS$ with an open cover $\{ U_i \}$ such that each $(U_i,\O_X|U_i)$ is an affine fan (resp.\ fine affine fan).   We define a morphism of fans so that fans form a full subcategory of $\LMS$.  We write $\Fans$ for the category of \emph{fine} fans. \end{defn}

The category $\LMS$ is, in many ways, much simpler than the category $\LRS$ of locally ringed spaces.  For example, unlike the inclusion functor $\LRS \into \RS$ and the underlying space functor $\LRS \to \Top$, the inclusion functor $\LMS \into \MS$ and the underlying space functor $\LMS \to \Top$ preserve inverse limits \cite[Theorem~5.4.1]{GM1}.  In other words, the inverse limit $X$ of a functor $i \mapsto X_i$ to $\LMS$ is constructed in the ``obvious" way:  On the level of spaces, $X$ is the inverse limit of the $X_i$ (denote the projections $\pi_i : X \to X_i$) and $\M_X$ is the direct limit of $i \mapsto \pi_i^{-1} \M_{X_i}$ in the category of sheaves of monoids on $X$.  (It is true, but not obvious, that the structure maps $\pi_i^{-1} \M_{X_i} \to \M_X$ are local.)  Like the inclusion functor $\Sch \into \LRS$, the inclusion functor from the category of (not necessarily fine) fans into $\LMS$ preserves finite inverse limits (all of which exist in these categories).  This isn't true of fine fans (though the category of fine fans does have finite inverse limits) because, for example, the direct limit of a finite diagram of fine monoids (taken in the category of all monoids) is not necessarily fine, though it is always finitely generated.

We now recall some basic facts about the CZE topology on the category $\Fans$.  This topology is closely related with the \'etale topology on the category of schemes over a field of characteristic zero.  For further details, see \cite[\S4.12]{G2}.

\begin{defn} \label{defn:CZE} A map of fine monoids $Q \to P$ is called \emph{CZE (characteristic-zero-\'etale)} iff $Q^* \to P^*$ is injective with finite cokernel and $$ \xym{ Q^* \ar[r] \ar[d] & P^* \ar[d] \\ Q \ar[r] & P } $$ is a pushout diagram of monoids (equivalently: of fine monoids).  (See \cite[4.12.1]{G2} for several equivalent characterizations of CZE maps.)  A map of fine fans $f : X \to Y$ is called \emph{CZE} iff $f_x : \mathcal{M}_{Y,f(x)} \to \mathcal{M}_{X,x}$ is a CZE map of monoids for every $x \in X$.  A \emph{CZE cover} is a surjective CZE map of fine fans.  \end{defn}

\begin{lem} \label{lem:CZEmaps} If $h : Q \to P$ is a CZE map of fine monoids then $\Spec h$ is a CZE map of fine monoids and a homeomorphism on topological spaces, hence also a CZE cover. \end{lem} 

\begin{proof} See \cite[4.12.5]{G2}. \end{proof}

If $h :Q \to P$ is a map of fine monoids for which $\Spec h$ is a CZE map of fans, then $h$ may not be a CZE map of monoids (because $\Spec h$ may not be surjective), but $h$ will factor through the localization $Q \to G^{-1}Q$ of $Q$ at the face $G := h^{-1}(P^*)$ via a CZE map of fine monoids $G^{-1}Q \to P$ (see \cite[4.12.6]{G2}).

\begin{defn} \label{defn:topology} In this paper, a \emph{topology} on a category $\C$ is a class $\tau$ of $\C$-morphisms, called \emph{covers}, closed under composition and base change (in particular these base changes are assumed to exist) and containing all isomorphisms.\footnote{To avoid set-theoretic difficulties in various constructions (sheafification, in particular), one should also assume that for any $Y \in \C$ there is a \emph{set} $S_Y$ of covers with codomain $Y$ (``covers of $Y$") cofinal amongst all covers of $Y$, meaning: For any other cover $f : X \to Y$ of $Y$, there is some cover in $S_Y$ factoring through $f$.  This should be clear in every example we consider.}  A \emph{site} is a category equipped with a topology.  \end{defn}

A \emph{topology} in our sense is basically a ``pretopology" in the sense of \cite[IV.4.2.5]{SGA3} (it is even a rather special case of a ``pretopology"); one can attach to it a topology, in the sense of \cite[IV.4.2.1]{SGA3}, in the manner described in \cite[IV.4.2.4]{SGA3}.  Although our definition of a ``topology" is slightly restrictive, it is general enough to include every example of interest to us.

\begin{prop} CZE covers are closed under composition and base change, hence they are the covers for a topology on the category $\Fans$ of fine fans, called the \emph{CZE topology}. \end{prop}

\begin{proof} See \cite[4.12.7]{G2}. \end{proof}

\begin{rem} \label{rem:Ux} Since the topological space underlying any fine fan $X$ is locally finite, any $x \in X$ has a smallest open neighborhood $U_x$.  Evidently \be U_x & = & \{ y \in X : x \in \ov{ \{ y \} } \}. \ee  In fact, by working in an affine open subspace of $X$ containing $x$, one sees that $U_x = \Spec \M_{X,x}$.  This identification is natural in $(X,x)$ in the following sense: For a morphism of fine fans $f : X \to Y$ and a point $x \in X$, one has a commutative diagram of fine fans \bne{Uxdiagram} & \xym{ U_x \ar[r] \ar[d]_{f|U_x} & X \ar[d]^-f \\ U_{f(x)} \ar[r] & Y } \ene where the horizontal arrows are the inclusions and the left vertical arrow is, under the aforementioned natural isomorphisms, identified with $\Spec$ of the map $f^\sharp_x : \M_{Y,f(x)} \to \M_{X,x}$. \end{rem} 

The upshot of the above remark is that conditions on the stalks of a map of fine fans translate directly into (``Zariski") local conditions on the map itself.  For example, the following result follows immediately from Remark~\ref{rem:Ux} (cf.\ \cite[4.12.8]{G2}):

\begin{prop} \label{prop:CZEimpliesZariski}  Suppose $f : X \to Y$ is a CZE map of fine fans.  Then: \begin{enumerate} \item Zariski locally on $f$, $f$ is $\Spec$ of a CZE map of monoids.  \item The map of topological spaces underlying $f$ is a local homeomorphism. \end{enumerate} In particular, any CZE cover of fine fans is a Zariski cover on the underlying spaces. \end{prop}

Any abelian group $A$ has a natural abelian group object structure in $\Mon^{\rm op}$ (the opposite of the category of monoids).  An action of $A$ on a monoid $P$ in $\Mon^{\rm op}$ is the same thing as a group homomorphism $a : P^{\rm gp} \to A$.  (See \cite[\S2.8]{G2}.)  Since $\Spec : \Mon^{\rm op} \to \Fans$ preserves finite inverse limits, it takes group objects to group objects (and actions to actions).

\begin{prop} \label{prop:CZEtorsors} Suppose \bne{FGAGSeq} & 0 \to A \to B \to C \to 0, \ene is an exact sequence of FGA groups.  Then there exists an injective map of FGA groups $A \into A'$ with finite cokernel such that the exact sequence \bne{FGAGSeq2} & 0 \to A' \to B' \to C \to 0 \ene obtained by pushing out \eqref{FGAGSeq} along $A \to A'$ splits.  It follows that the sequence \bne{GGFGAGSeq} & 0 \to \GG(C) \to \GG(B) \to \GG(A) \to 0 \ene obtained by applying $\GG( \slot ) = \Spec( \slot )$ to \eqref{FGAGSeq} is a short exact sequence of (representable) sheaves of abelian groups on the category of fine fans in the CZE topology which can be split after pulling back along a CZE cover $\GG(A') \to \GG(A)$.  Equivalently, the induced action of $\GG(C)$ on $\GG(B)$ makes $\GG(B) \to \GG(A)$ a $\GG(C)$ torsor, locally trivial in the CZE topology.  \end{prop}

\begin{proof} See \cite[4.12.10]{G2}. \end{proof}

\subsection{Log schemes}  \label{section:logschemes} We work throughout over a fixed base field $k$, of characteristic zero.  Let $\Sch$ denote the category of $k$-schemes, which we will refer to throughout as ``schemes".  In most of this paper, we will not be terribly concerned with log structures, but it is sometimes helpful to 

\begin{defn} \label{defn:logstructure} A \emph{prelog structure} (resp.\ \emph{log structure}) on $X \in \Sch$ is a map $\alpha : \M_X \to \O_X = (\O_X,\cdot)$ of sheaves of monoids on the \'etale site of $X$ (resp.\ such that $\alpha|\alpha^{-1}(\O_X^*) : \alpha^{-1}(\O_X^*) \to \O_X^*$ is an isomorphism).  A map of log structures or prelog structures is a homomorphism of sheaves of monoids over $\O_X$. \end{defn}

Note that, in the above definition---and throughout this section---$\O_X$ denotes the \emph{\'etale} structure sheaf of $X$, defined by taking an \'etale map $U \to X$ to $\O_U(U)$.

\begin{example} \label{example:triviallogstructure} Take $\M_X := \O_X^*$, $\alpha : \M_X \to \O_X$ the inclusion.  This defines a log structure on $X$ called the \emph{trivial log structure}; it is the initial object in the category of log structures on $X$. \end{example}

The inclusion of log structures into prelog structures has a left adjoint given by taking $\alpha : \M_X \to \O_X$ to the natural map $\alpha^a : \M_X \oplus_{\alpha^{-1} \O_X^*} \O_X^* \to \O_X$ defined via the universal property of the pushout.  The log structure $\alpha^a$ is called the \emph{log structure associated to the prelog structure} $\alpha$.  If there is no chance of confusion (writing $\alpha_X$ instead of $\alpha$ will often alleviate this confusion) one writes $\M_X$ for a log or prelog structure and $\M_X^a$ for the associated log structure.

\begin{example} \label{example:trivialassociatedlogstructure} If $\alpha : \M_X \to \O_X$ is a prelog structure where $\alpha$ factors through $\O_X^* \into \O_X$, then $\alpha^a$ is the trivial log structure (Example~\ref{example:triviallogstructure}).  The easiest way to see this is to just directly check that every map of prelog structures from $\M_X$ to a log structure $\N_X$ factors uniquely through $\M_X \to \O_X^*$.  \end{example}

\begin{defn} \label{defn:pullbacklogstructure} If $f : X \to Y$ is a map of schemes and $\alpha_Y : \M_Y \to \O_Y$ is a prelog structure on $Y$, then the \emph{pullback log structure} $f^* \M_Y$ is the log structure on $X$ associated to the prelog structure given by the composition $$ \xym{ f^{-1} \M_Y \ar[r]^-{ f^{-1} \alpha_Y } & f^{-1} \O_Y \ar[r]^-{f^\sharp} & \O_X. } $$  \end{defn}

Formation of the pullback log structure $f^* \M_Y$ commutes with forming the associated log structure $\M_Y^a$.

\begin{defn} \label{defn:logscheme} A \emph{log scheme} is a scheme $X$ equipped with a log structure $\M_X$.  A map of log schemes $f : (X,\M_X) \to (Y,\M_Y)$ is a map of schemes, also abusively denoted $f : X \to Y$, together with a map $f^\dagger : f^* \M_Y \to \M_X$ of log structures on $X$.  Such a map is called \emph{strict} iff $f^\dagger$ is an isomorphism. \end{defn}

\begin{defn} \label{defn:strictetale} A map of log schemes is called \emph{strict \'etale} (resp.\ a \emph{strict \'etale cover}) iff it is strict and the underlying map of schemes is \'etale (resp.\ and surjective).  \end{defn}

\begin{defn} \label{defn:finelogstructure} Given a monoid $P$ and a monoid homomorphism $P \to \O_X(X)$, we can consider the corresponding map from the associated constant sheaf ``$P$" on the \'etale site of $X$ to the (\'etale) structure sheaf $\O_X$ of $X$.  This defines a prelog structure on $X$ whose associated log structure we denote $(P \to \O_X(X))^a$, or just $P^a$ if there is no chance of confusion.  A log structure on $X$ \'etale locally isomorphic to one of the form $P^a$, with $P$ fine, is called a \emph{fine log structure} on $X$.  A \emph{fine log scheme} is a scheme equipped with a fine log structure.  The category of fine log schemes is denoted $\LogSch$.  Strict \'etale covers define a topology (Definition~\ref{defn:topology}) on $\LogSch$ called the \emph{strict \'etale topology}. \end{defn}

\begin{example} \label{example:XP} Suppose $P$ is a fine monoid, $k[P]$ is the associated monoid algebra over our base field $k$, and $X(P) := \Spec k[P]$ is the associated affine scheme.  For any $U \in \Sch$ we have natural bijections \be \Hom_{\Sch}(U,X(P)) & = & \Hom_{\Alg(k)}(k[P],\O_U(U)) \\ & = & \Hom_{\Mon}(P,\O_U(U)), \ee so $X(P)$ represents the presheaf $U \mapsto \Hom_{\Mon}(P,\O_U(U))$ on $\Sch$.  There is an obvious monoid homomorphism $\alpha_P : P \to k[P] = \O_X(X)$ taking $p \in P$ to its image $[p]$ in the monoid algebra.  We regard $X(P)$ as a fine log scheme by equipping it with the (manifestly fine) log structure $\M_{X(P)} := P^a$ associated to $\alpha_P$.  A common alternative notation for $X(P)$ is $\Spec(P \to k[P])$.  Combining the modular interpretation of $X(P) \in \Sch$ above with the universal property of $P^a$, one sees that $X(P) \in \LogSch$ represents the presheaf \be \LogSch^{\rm op} & \to & \Sets \\ U & \mapsto & \Hom_{\Mon}(P,\M_X(X)). \ee  The construction of $X(P)$ is contravariantly functorial in $P$; we view $X( \slot )$ as a functor \bne{X} X( \slot ) : \Mon^{\rm op} & \to & \LogSch \ene from fine monoids to fine log schemes.  \end{example}

\begin{example} \label{example:XG} Suppose $G$ is a FGA group.  Then the log structure on $X(G)$ is trivial (see Example~\ref{example:trivialassociatedlogstructure}). \end{example}

The category of fine log schemes has finite inverse limits, though they do not generally commute with the forgetful functor $\LogSch \to \Sch$.  The category of fine monoids $\Mon$ has all finite direct limits (though they aren't generally the same is those calculated in the category of all monoids).  The functor \eqref{X} preserves finite inverse limits. 

\subsection{Log differentiable spaces}  \label{section:logdifferentiablespaces}  To carry out our ``differential realization" constructions, we will need some convenient category (with good formal properties) where we can ``do our differential geometry."  We will work in the category $\DS$ of \emph{differentiable spaces} discussed in \cite{G} (see the references there for history and further reading).  For our purposes, we will need only (some of) the following formal properties of $\DS$:  \begin{enumerate} \item \label{fullsubcategory} The category $\DS$ is a full subcategory of the category $\LRS/\RR$ of locally ringed spaces over $\RR$.  \item \label{Rpoint} Every point $x$ of a differentiable space $(X,\O_X)$ is an $\RR$-point, meaning that the composition of the $\LRS / \RR$ structure map $\RR \to \O_{X,x}$ and the projection $\O_{X,x} \to k(x)$ from the local ring $\O_{X,x}$ to its residue field $k(x)$ is an isomorphism.  \item The topological space underlying any differentiable space is locally homeomorphic to a closed subspace of $\RR^n$ (for varying $n$).  \item \label{local} ``Being in $\DS$ is a local property of objects of $\LRS/\RR$."  That is, for $(X,\O_X) \in \LRS / \RR$ and any open cover $\{ U_i \}$ of $X$, we have $(X,\O_X) \in \DS$ iff $(U_i,\O_X|U_i) \in \DS$ for all $i$.  \item \label{containsMan} The category $\DS$ contains the category $\Man$ of smooth manifolds as a full subcategory; the inclusion functor $\Man \into \DS$ preserves finite products.  \item \label{inverselimitsinDS} The category $\DS$ has all finite inverse limits.  Such limits commute with the ``underlying topological space" functor $\DS \to \Top$. \item \label{modularinterpretationofR} The real line $\RR$, with its usual smooth manifold structure is an object of $\DS$ (by \eqref{containsMan}) representing the ``global sections" functor \be \DS^{\rm op} & \to & \Sets \\ X & \mapsto & \O_X(X). \ee \item \label{modularinterpretationofRplus} For $X \in \DS$, $f \in \O_X(X)$, $x \in X$, let $f(x) \in \RR$ be the image of $f$ in the residue field $k(x)$ under the identification $k(x)=\RR$ of \eqref{Rpoint}.  The presheaf \be \DS^{\rm op} & \to & \Sets \\ X & \mapsto & \{ f \in \O_X(X) : f(x) \in \RR_{\geq 0} \}. \ee is representable by a differentiable space, denoted $\RR_{\geq 0}$, and the $\DS$-morphism $\RR_{\geq 0} \to \RR$ deduced from the ``modular interpretation" of $\RR$ in \eqref{modularinterpretationofR} is a closed embedding in $\LRS / \RR$ where the underlying closed embedding of topological spaces is the obvious embedding $\RR_{\geq 0} \into \RR$ suggested by the notation.   \end{enumerate}

The ``modular interpretation" of $\RR_{\geq 0}$ in \eqref{modularinterpretationofRplus} endows $\RR_{\geq 0}$ with the structure of a monoid object in $\DS$.  For any fine monoid $P$, consider the presheaf $Y(P)$ on $\DS$ defined by \be Y(P)(U) & := & \Hom_{\Mon}(P,\RR_{\geq 0}(U)). \ee  There is an obvious natural isomorphism $Y(\NN) = \RR_{\geq 0}$.  Combining this observation with property \eqref{inverselimitsinDS} of $\DS$ and the fact that every fine monoid is finitely presented, we see that $Y(P)$ is representable (by a differentiable space we also abusively denote by $Y(P)$) for every $P \in \Mon$.  This defines a functor \be Y( \slot ) : \Mon^{\rm op} & \to & \DS. \ee  (Compare the analogous construction of $X(P)$ in \S\ref{section:logschemes}.)

\begin{defn} \label{defn:DSlogstructure}  Let $X$ be a differentiable space.  We define a \emph{prelog structure} (resp.\ \emph{log structure}, \emph{fine log structure}) on $X$ exactly as we defined a prelog structure (resp.\ log structure, fine log structure) on a scheme $X$ in Definitions~\ref{defn:logstructure} and \ref{defn:finelogstructure}, except we replace the sheaf of monoids $(\O_X,\cdot)$ used in that case with the sheaf of monoids $\O_{X,\geq 0}$ on our differentiable space $X$ defined by $U \mapsto \RR_{\geq 0}(U)$ (and we work with the Zariski topology on $X$---there is no meaningful notion of the ``\'etale topology" on $\DS$).  We then define a \emph{log differentiable space}, \emph{morphism of log differentiable spaces}, and a \emph{fine log differentiable space} in analogy with Definitions~\ref{defn:logscheme}.  The category of fine log differentiable spaces is denoted $\LogDS$. \end{defn}

\begin{rem} What we call a ``log differentiable space" in the above definition is called a ``positive log differentiable space" in \cite{GM1}. \end{rem}

\begin{example} \label{example:YP}  Let $P$ be a fine monoid.  From the ``modular description" of $Y(P)$ above, we obtain a tautological monoid homomorphism $\alpha_P : P \to \O_{Y(P),\geq 0}(Y(P))$.  We view $Y(P)$ as a fine log differentiable space by endowing it with the fine log structure $\M_{Y(P)} := P^a$ associated to $\alpha_P$.  Using the modular interpretation of the underlying differentiable space $Y(P)$ and the universal property of the associated log structure $P^a$, we see that $Y(P)$ represents the presheaf \be \LogDS^{\rm op} & \to & \Sets \\ U & \mapsto & \Hom_{\Mon}(P,\M_U(U)). \ee  This construction yields a finite-inverse-limit-preserving functor \be Y( \slot ) : \Mon^{\rm op} & \to & \LogDS . \ee   \end{example}

\begin{example} \label{example:YG} Suppose $G$ is a FGA group.  Then we see that the log structure on $Y(G)$ is trivial by using Example~\ref{example:trivialassociatedlogstructure} exactly as in Example~\ref{example:XG}.  Furthermore, if $G$ is finite, then $Y(G)$ is the terminal object of $\LogDS$ (the point $\Spec \RR$ with trivial log structure).  This is a special case of \cite[5.9.3]{GM1}, but let us give an alternative direct argument.  The key claim is that for any differentiable space $X$, the group $\O_{X,>0}(X)$ of positive functions on $X$ is torsion free.  To prove the claim, suppose $f \in \O_{X,>0}(X)$ satisfies $f^m = 1$ for some $m \in \ZZ_{>0}$.  We want to show that $f=1$.  Since any differentiable space $X$ is, essentially by definition, ``Taylor reduced", it suffices to check that $f_x \in \O_{X,x}$ is equal to $1$ in the $\m_x$-adic completion of the local ring $\O_{X,x}$ for each $x \in X$.  Since $f^m=1$, $f(x) \in \RR_{>0}$, and $1$ is the only root of unity in $\RR_{>0}$ we know $f(x)=1$, so $f_x=1+\epsilon$ for some $\epsilon \in \m_x$.  Now $f_x^m=1$ implies that $m \epsilon=0$ modulo $\m_x^2$, so $\epsilon \in \m_x^2$ since $m \in \RR^* \subseteq \O_{X,x}^*$.  Repeatedly applying this same argument shows that $\epsilon \in \m_x^n$ for every $n \in \ZZ_{>0}$, so $f=1$ in the $\m_x$-adic completion of $\O_{X,x}$, as desired.

Now we use the ``modular interpretation" of $Y(G)$ above to see that, for any $U \in \LogDS$, the set \be \Hom_{\LogDS}(U,Y(G)) & = & \Hom_{\Mon}(G,\M_U(U)) \\ & = & \Hom_{\Ab}(G,\M_U^*(U)) \\ & = & \Hom_{\Ab}(G,\O_{X,>0}(X)) \ee is punctual because the group $\O_{X,>0}(X)$ is torsion-free and the finite group $G$ is a torsion group.  \end{example}

\subsection{Realization functors} \label{section:realizationofmonoids}  The (contravariant) association of various geometric objects (topological spaces, schemes, \dots) to monoids lies at the heart of toric geometry.  In \cite[\S4]{GM1}, we formulated an axiomatic setup for constructing functors \be \AA : \Mon^{\rm op} & \to & \Esp \\ \AA : \Fans & \to & \Esp \ee out of the category of monoids $\Mon$ (or fine monoids, finitely generated monoids, etc.), or its ``generalization," the category $\Fans$ of fans (\S\ref{section:fans}).  These functors take values in various categories ``$\Esp$" of ``geometric objects" (``spaces").  In \cite[Definition~4.0.7]{GM1} we defined a \emph{category of spaces} to be a pair $(\Esp,\AA^1)$, where $\Esp$ is a category equipped with an ``underlying space functor" $\Esp \to \Top$, often written $X \mapsto \u{X}$, and a monoid object $\AA^1$ satisfying some reasonable properties that will almost certainly hold in every example the reader can imagine.  In particular, the categories $\Fans$, $\LogSch$, and $\LogDS$ of \S\ref{section:fans}, \S\ref{section:logschemes}, and \S\ref{section:logdifferentiablespaces}, equipped with their usual ``underlying space" functors and the ``obvious" monoid objects $\Spec \NN$, $\Spec( \NN \to k[\NN])$, and $Y(\NN)$ are categories of spaces.  

For any ``category of spaces" $(\Esp,\AA^1)$, the general nonsense of \cite[\S4]{GM1} yields a finite-inverse-limit-preserving functor \bne{monoidrealization} \AA : \Mon^{\rm op} & \to & \Esp \ene taking $\NN$ to $\AA^1$.  It follows from the fact that \eqref{monoidrealization} preserves finite inverse limits (and the fact that any fine monoid is finitely presented) that, for any finitely generated monoid $P$, the space $\AA(P)$ represents the presheaf \be \Esp^{\rm op} & \to & \Sets \\ U & \mapsto & \Hom_{\Mon}(P,\AA^1(U)). \ee  One thus sees that there is an \emph{essentially unique} such functor.  By axiomatizing the idea of ``gluing together spaces along open subspaces," we show that \eqref{monoidrealization} extends to a finite-inverse-limit-preserving functor \bne{AA} \AA : \Fans & \to & \Esp, \ene also essentially unique in an appropriate sense (as a ``morphism of categories of spaces").  See \cite[\S4]{GM1} for details.

Since \eqref{monoidrealization} preserves finite inverse limits, it takes group objects to group objects.  Every abelian group $A$ can be regarded as an abelian group object in $\Mon^{\rm op}$ (functorially in $A$), thus we obtain a functor \bne{GG} \GG( \slot ) : \Ab^{\rm op} & \to & \Ab(\Esp) \ene from FGA groups to abelian group objects in spaces.  For example, when $\Esp = \Sch$ is the category of (log) $k$-schemes, $\GG(A) = \Spec k[A]$ is the usual diagonalizable group scheme attached to $A$ (with the trivial log structure when $\Esp=\LogSch$).  The group spaces $\GG(A)$ will play an important role in this paper since we often consider stacks obtained as quotients by actions of such groups.

One of our axioms for the underlying space functor $\Esp \to \Top$ is that ``open subspaces are representable."  Here is what this means:  Given $X \in \Esp$ and an open subset $\u{U} \subseteq \u{X}$ of its underlying topological space, we can consider the subpresheaf $U$ of $X$ defined by letting $U(T)$ be the set of $\Esp$-morphisms $f : T \to Y$ such that the underlying map of topological space $\u{f} : \u{T} \to \u{Y}$ factors through $\u{U}$.  We assume that this $U$ is representable and that the map of topological spaces underlying the $\Esp$ morphism $U \to X$ ``is" the inclusion $\u{U} \into \u{X}$, as the notation would suggest.  We can an $\Esp$-morphism $V \to Y$ isomorphic to a map $U \to X$ constructed above an \emph{open embedding}.  One checks readily that open embeddings are stable under base change (we assume $\Esp$ has all finite inverse limits).  Using the notion of ``open embedding" one can then define a ``Zariski cover," and a ``Zariski topology," in the usual way.

In order to form ``realizations" of KM fans, we will generally need to use not only our category $\Esp$ of ``spaces" but also the ``categories" of sheaves and stacks over $\Esp$.  To make sense of this, we need to assume that $\Esp$ is equipped with a topology $\tau$, which is \emph{admissible} in the following sense:

\begin{defn} \label{defn:admissible} Let $(\Esp,\AA^1)$ be a category of spaces.  A topology $\tau$ on $\Esp$ is called \emph{admissible} iff it is\begin{enumerate} \item \label{admissible1} \emph{subcanonical}, \item \label{admissible2} \emph{compatible with the CZE topology}, and \item \label{admissible3} \emph{compatible with the structure of} $\Esp$ \emph{as a category of spaces}. \end{enumerate} \end{defn}

These conditions mean, respectively, that \begin{enumerate} \item every $\tau$-cover is a universal effective descent morphism (canonical cover), \item the realization \eqref{AA} takes CZE covers to $\tau$-covers, and \item every ``Zariski cover" is a $\tau$-cover. \end{enumerate}  

\begin{prop} \label{prop:tautopology}  The CZE topology on the category $\Fans$ of fine fans, the strict \'etale topology on the category $\LogSch$ of fine log schemes, and the Zariski topology on the category $\LogDS$ of fine log differentiable spaces are admissible.  \end{prop}

\begin{proof}  The CZE topology is subcanonical by \cite[4.12.4]{G2}.  The fact that the strict \'etale topology is subcanonical follows from the fact that the \'etale topology on $\Sch$ is subcanonical (since even the fppf topology is subcanonical by \cite[VIII.5.1]{SGA1}) and the fact that log structures are defined in terms of sheaves in the \'etale topology, hence one certainly has \'etale descent for such sheaves.  In checking through the details, one uses the fact that if $f : X \to Y$ is strict \'etale, then $f^{-1}\M_Y \to \M_X$ is an isomorphism of \'etale sheaves of monoids on $X$ over the \'etale structure sheaf $\O_X$.  This is because $f^{-1} \O_Y \to \O_X$ is an isomorphism of \'etale sheaves (these are the \emph{\'etale} structure sheaves and $f$ is \'etale) so $f^* \M_Y = f^{-1} \M_Y$ and $f^* \M_Y \to \M_X$ is an isomorphism since $f$ is strict (Definition~\ref{defn:logscheme}).   The Zariski topology (=``strict Zariski topology") is subcanonical by the same reasoning since the Zariksi topology is subcanonical on the category $\DS$ of differentiable spaces.  (It is even subcanonical on the category $\LRS / \RR$ of locally ringed spaces over $\RR$, which contains $\DS$ as a full subcategory.)

Algebraic realization takes CZE covers to \'etale covers by \cite[4.12.9]{G2}.  That reference doesn't mention log structures, so we also need to check that the algebraic realization of a CZE map of fine fans is strict.  This is a local question, so by Proposition~\ref{prop:CZEimpliesZariski} we reduce to showing that the algebraic realization of a CZE map of fine monoids is strict.  Since algebraic realization preserves inverse limits, strict maps are stable under base change, and a CZE map of fine monoids is, by definition, a pushout of an injection $A \to B$ of FGA groups with finite cokernel, we just need to show that the algebraic realization of $A \to B$ is strict.  But this is automatic because the log structures on the algebraic realizations of the FGA groups $A$ and $B$ are trivial (Example~\ref{example:XG}).

To see that differential realization takes CZE covers to Zariski covers of log differentiable spaces, we reduce as in the previous paragraph to showing that the differential realization of an injection $A \to B$ of FGA groups with finite cokernel is an isomorphism of log differentiable spaces.  As in the previous paragraph, the log structures on the realizations of $A$ and $B$ are trivial (Example~\ref{example:YG}), so we just need to check that $Y(B) \to Y(A)$ is a diffeomorphism, which is \cite[5.9.3]{GM1}.

It is clear that a Zariski cover is a cover for any of the topologies in question.  \end{proof}

The following proposition will be used frequently throughout the paper without further comment:

\begin{prop} \label{prop:tautorsors} Suppose \bne{SequenceA} & 0 \to A \to B \to C \to 0, \ene is an exact sequence of FGA groups and $(\Esp,\AA^1)$ is a category of spaces with an admissible topology $\tau$.  Then there exists an injective map of FGA groups $A \into A'$ with finite cokernel such that the exact sequence \bne{SequenceB} & 0 \to A' \to B' \to C \to 0 \ene obtained by pushing out \eqref{SequenceA} along $A \to A'$ splits.  It follows that the sequence \bne{SequenceC} & 0 \to \GG(C) \to \GG(B) \to \GG(A) \to 0 \ene obtained by applying \eqref{GG} to \eqref{SequenceA} is a short exact sequence of (representable) sheaves of abelian groups on $\Esp$ which can be split after pulling back along a $\tau$-cover $\GG(A') \to \GG(A)$.  Equivalently, the induced action of $\GG(C)$ on $\GG(B)$ makes $\GG(B) \to \GG(A)$ a $\GG(C)$ torsor, locally trivial in the $\tau$ topology.  \end{prop}

\begin{proof} This follows immediately from Proposition~\ref{prop:CZEtorsors} and the definition of an admissible topology (Definition~\ref{defn:admissible}). \end{proof}

Assumption \eqref{admissible2} in Definition~\ref{defn:admissible} ensures that for an admissible topology $\tau$, the functor \eqref{AA} defines a morphism of sites\footnote{Our convention concerning the ``variance" of this map follows [SGA3].  The various conventions are set up so that the functors \be \Top & \to & \Sites \\ X & \mapsto & \{ {\rm open \; subsets \; of \;} X \} \ee and \be \Sites & \to & \Topoi \\ X & \mapsto & \Sh(X) \ee are covariant.} $\Esp \to \Fans$ and hence a morphism of topoi \bne{generaltopoimap} \AA = (\AA^{-1},\AA_*) : \Sh( \Esp ) & \to & \Sh( \Fans ). \ene  For now, we will be concerned mainly with the ``inverse image" part $\AA^{-1}$ of \eqref{generaltopoimap}, which we will simply denote \bne{topoimap} \AA : \Sh( \Fans ) & \to & \Sh ( \Esp ). \ene  By ``categorifying" the construction of this inverse image map, one similarly obtains a map of $2$-categories \bne{stackmap} \AA : \St( \Esp ) & \to & \St( \Fans ) \ene where $\St( \C )$ denotes the $2$-category of stacks over a site $\C$.  The various realization functors \eqref{AA}, \eqref{topoimap}, \eqref{stackmap} fit into a $2$-commutative diagram \bne{realizationfunctors} & \xym{ \Fans  \ar[r]^-{\AA} \ar[d] & \Esp \ar[d] \\ \Sh( \Fans ) \ar[r]^-{\AA} \ar[d] & \Sh( \Esp ) \ar[d] \\ \St( \Fans ) \ar[r]^-{\AA} & \St( \Esp ) } \ene where the top vertical arrows are the (fully faithful, inverse-limit-preserving) Yoneda embeddings and the bottom vertical arrows are similar embeddings obtained by viewing a presheaf on a category $\C$ as a category fibered in \emph{sets} (hence also in groupoids) over $\C$.  The abuse of notation in denoting all three horizontal arrows by ``$\AA$" is hence relatively harmless.

\section{KM Fans} \label{section:KMfans}

\subsection{Cones and fans} \label{section:conesandfans} Here we recall, for the convenience of the reader, some basic notions of toric geometry.

\begin{defn} \label{defn:lattice} A \emph{lattice} is a group isomorphic to $\ZZ^n$ for some $n \in \NN$ (equivalently, a torsion-free FGA group). \end{defn}

\begin{notation} \label{notation:MN} Throughout the paper $N$, $N'$, $N''$ will be FGA groups and $L$, $L'$, $L''$ will be lattices.  We set $N_{\RR} := N \otimes \RR$.  For any FGA group $N$, the vector spaces $N^\lor_{\RR}$ and $N_{\RR}$ are canonically dual; we denote the canonical pairing between them by $\langle \slot , \slot  \rangle : N^\lor_\RR \times N_\RR  \to  \RR$. \end{notation}

\begin{defn} \label{defn:cone} A \emph{cone} in $N_{\RR}$ \emph{(with respect to $N$)} is a non-empty subset $\sigma \subseteq N_{\RR}$ such that \bne{generators} \sigma & = & \{ \lambda_1 n_1 + \cdots + \lambda_k n_k : \lambda_1, \dots, \lambda_k \in \RR_{\geq 0} \} \ene for some $n_1, \dots, n_k \in N$.\footnote{A cone in this sense is a \emph{rational convex polyhedral cone} in the sense of Fulton \cite[Pages 9, 12]{F}.}  Here, and elsewhere, we suppress notion for the map $N \to N_{\RR}$.  If $\sigma$ is a cone and $n_1,\dots,n_k \in N$ satisfy \eqref{generators}, we say that $n_1,\dots,n_k$ are \emph{generators} for $\sigma$.  Evidently, a cone is a monoid under addition.  A cone $\sigma$ is called \emph{sharp} iff $\sigma \cap - \sigma = \{ 0 \}$ (equivalently, $\sigma$ contains no linear subspace of $N_{\RR}$ of positive dimension).\footnote{Our \emph{sharp cones} are Fulton's \emph{rational strictly convex polyhedral cones} \cite[Page 14]{F}.}   We refer to \be \dim \sigma & := & \dim \Span \sigma. \ee as the \emph{dimension} of the cone $\sigma$  

If $\sigma,\tau$ are cones in $N_{\RR}$ with $\tau \subseteq \sigma$, then we say that $\tau$ is a \emph{subcone} of $\sigma$.  We say that a subcone $\tau \subseteq \sigma$ is a \emph{face} of $\sigma$ (notation: $\tau \leq \sigma$) iff $\tau$ is a face of the monoid $\sigma$ in the sense of Definition~\ref{defn:face}.\footnote{This notion of ``face of a cone" is the same as Fulton's \cite[Page 9]{F}.  Every ``face" of the monoid $\sigma$, in the sense of Definition~\ref{defn:face}, is a subcone of $\sigma$.} 

For a subset $\sigma \subseteq N_{\RR}$, we call \be \sigma^\lor & := & \{ m \in M_{\RR} : \langle m, n \rangle \geq 0 {\rm \; for \; all \;} n \in \sigma \} \\ \sigma^\perp & := & \{ m \in M_{\RR} : \langle m, n \rangle = 0 {\rm \; for \; all \;} n \in \Sigma \} \ee the \emph{dual} and \emph{orthogonal complement} of $\sigma$, respectively.  \end{defn} 

\begin{rem} \label{rem:cones} Suppose $N \to N'$ is a map of FGA groups such that the induced map $N \otimes \QQ \to N' \otimes \QQ$ is an isomorphism (equivalently $N \otimes \RR \to N' \otimes \RR$ is an isomorphism).  Then a subset $\sigma \subseteq N_{\RR}=N'_{\RR}$ is a cone with respect to $N$ iff it is a cone with respect to $N'$.  In particular, it is harmless to assume that $N$ is torsion-free when discussing cones---on the other hand, there is no harm in allowing $N$ to have torsion; this general setup will be convenient later. \end{rem}

\begin{lem} \label{lem:keyfact} For any cone $\sigma \subseteq N_{\RR}$ (with respect to $N$), the dual $\sigma^\lor \subseteq N^\lor_{\RR}$ is a cone (with respect to $N^\lor$) and $\sigma = \sigma^{\lor \lor}$. \end{lem}

\begin{proof}  See \cite[Page 12]{F} and the references listed in \cite[Page 132, (5)]{F}. \end{proof}

\begin{defn} \label{defn:Ssigma} For a cone $\sigma \subseteq N_{\RR}$ we write $S_{\sigma}(N) := \sigma^\lor \cap N^\lor$ for the \emph{monoid of integral points in the dual cone of} $\sigma$.  By Lemma~\ref{lem:keyfact}, the monoid $S_{\sigma}(N)$ is finitely generated.  If there is no chance of confusion (there often will be) we will write $S_{\sigma}$ instead of $S_\sigma(N)$.  \end{defn}

\begin{rem} \label{rem:Ssigma} It is a basic fact of toric geometry that, when $\sigma$ is sharp, $S_\sigma$ generates $N^\lor$ as a group---i.e.\ $S_{\sigma}^{\rm gp} = N^\lor$.  It is clear from the definitions that the group of units in the monoid $S_\sigma$ is $S_\sigma^* = \sigma^\perp \cap N^\lor$. \end{rem}

\begin{notation} \label{notation:abusive}  For subsets $A \subseteq N$, $\sigma \subseteq N_{\RR}$, we often write $\sigma \cap A$ as abusive notation for the set of $a \in A$ mapping into $\sigma$ under the natural map $A \to N_{\RR}$.  (Usually $A \subseteq N$ will be a subgroup and $\sigma$ will be a cone in $N_{\RR}$.) \end{notation}

\begin{defn} \label{defn:Nsigma}  For a cone $\sigma \subseteq N_{\RR}$, we set $N_{\sigma} := \Span \sigma \cap N$ (using Notation~\ref{notation:abusive}). \end{defn}

\begin{rem} \label{rem:Nsigma} Because of the way $N_{\sigma}$ is defined, the inclusion $N_{\sigma} \subseteq N$ is \emph{saturated}, meaning the quotient $N/N_{\sigma}$ is free (i.e.\ is a lattice). \end{rem}

\begin{defn} \label{defn:fan} A \emph{fan} $F$ \emph{in} $N$ is a finite, non-empty set of sharp cones in $N_{\RR}$ such that: \begin{enumerate}[label=(\roman*), ref=\roman*] \item $\tau \in F$ for any $\sigma \in F$ and any $\tau \leq \sigma$ and \item $\sigma \cap \tau \leq \sigma$ for any $\sigma, \tau \in F$. \end{enumerate}  \end{defn}  

\begin{defn} \label{defn:nondegenerate} A fan $F$ in $N$ is \emph{atoroidal} (or \emph{has no torus factors}) iff the natural map of $\RR$ vector spaces \bne{nondegeneratevectorspacemap} \bigoplus_{\sigma \in F} \Span \sigma & \to & N_{\RR} \ene is surjective, or, equivalently, the natural map of groups \bne{nondegeneratemap} \bigoplus_{\sigma \in F} N_\sigma & \to & N \ene has finite cokernel; otherwise we say that $F$ \emph{has torus factors}.  A fan $F$ is called \emph{nondegenerate} iff $\dim \sigma = \dim N_{\RR}$ for every maximal cone $\sigma \in F$.  \end{defn}

\subsection{Lattice data and KM fans} \label{section:definitions}  Finally we come to the main objects of study.  Throughout this section, $N$ denotes an arbitrary FGA group.

\begin{defn} \label{defn:latticedatum} Let $\sigma$ a cone in $N_{\RR}$ with respect to $N$ (Definition~\ref{defn:cone}).  A \emph{lattice datum} for $\sigma$ (with respect to $N$) is a lattice $F_{\sigma} \subseteq N_{\sigma}$ (Definition~\ref{defn:Nsigma}) such that $F(\sigma) := N_{\sigma} / F_{\sigma}$ is finite.  A \emph{lifting} of a such a lattice datum is a lattice $L \subseteq N$ such that $N/L$ is finite and $F_\sigma = \Span \sigma \cap L$. \end{defn}

\begin{lem} \label{lem:liftingsexist} For $N$, $\sigma$ as in the above definition, every lattice datum $F_\sigma$ for $\sigma$ has at least one lifting $L$.  Furthermore, one can choose $L$ so that the map $N/F_\sigma \to N/L$ is an isomorphism on torsion subgroups. \end{lem}

\begin{proof} Since $N / N_{\sigma}$ is free (Remark~\ref{rem:Nsigma}) we can always \emph{choose} a splitting $N \cong N_{\sigma} \oplus N / N_{\sigma}$.  Then $L := F_{\sigma} \oplus N / N_{\sigma}$ is clearly a lifting of $F_\sigma$.  For this choice of $L$, the map $N / F_\sigma \to N/L$ is just the obvious projection $N_\sigma / F_\sigma \oplus N / N_\sigma \to N_\sigma / F_\sigma$, which is an isomorphism on torsion subgroups since $N / N_\sigma$ is torsion free. \end{proof}

\begin{defn} \label{defn:KMfan} A \emph{KM fan} $(N,F,\{ F_{\sigma} \})$ in $N$ is a fan $F$ in $N_{\RR}$ (Definition~\ref{defn:fan}), together with a lattice datum $F_\sigma \subseteq N_\sigma$ for each cone $\sigma \in \Sigma$.  These lattice data $F_\sigma$ are required to satisfy the \emph{compatibility condition} $F_\tau = \Span \tau \cap F_{\sigma}$ for all $\sigma \in F$, $\tau \leq \sigma$.  We usually abusively write ``$F$" for the triple $(N,F,\{ F_\sigma \})$ and say that ``$F$ is a KM-fan in $N$ with \emph{lattice data} $\{ F_{\sigma} \}$."  When $N$ is torsion-free (i.e.\ a lattice), such a triple is called a \emph{lattice KM fan}.  We sometimes refer to the group $N$ as the ``lattice for the KM fan $F$" even though it need not be a lattice in general. \end{defn}

\begin{rem} \label{rem:liftings} If $F = (N,F,\{ F_{\sigma} \})$ is a KM fan, $\sigma \in F$, $\tau \leq \sigma$, and $L \subseteq N$ is a lifting of $F_\sigma$, then the compatibility condition $F_\tau = F_\sigma \cap \Span \tau$ implies that $L$ is also a lifting of $F_\tau$. \end{rem}

\begin{defn} \label{defn:KMfanmorphism} A \emph{morphism} of KM-fans \be f : (N,F, \{ F_\sigma \} ) & \to & (N',F', \{ F'_\tau \}) \ee is a group homomorphism $f : N \to N'$ such that for each $\sigma \in F$ there is some $\sigma' \in F'$ such that (i) $f_{\RR}(\sigma) \subseteq \sigma'$ and (ii) $f(F_{\sigma}) \subseteq F'_{\sigma'}$.  With this notion of morphisms, KM fans form a category denoted $\KMFans$. \end{defn}

\begin{rem} \label{rem:KMfan}  We again emphasize that we do \emph{not} require $N$ to be torsion-free (Notation~\ref{notation:MN}) as in the usual theory of toric varieties.  The lattice data $F_{\sigma}$ \emph{are} torsion-free (Definition~\ref{defn:latticedatum}).  Even if we started off by considering only lattice KM fans, various basic constructions (for example, the star fan construction of \S\ref{section:starfansandstratification}) would immediately force us to consider the more general setup.  In the definition of \emph{morphism}, the coherence condition for $F'$ ensures that (ii) holds for \emph{some} $\sigma' \in F'$ satisfying (i) iff (ii) holds for \emph{all} $\sigma' \in F'$ satisfying (i). \end{rem}

\subsection{Drawing KM fans} \label{section:drawingKMfans}  One could alternatively (but equivalently) define a KM fan as follows:

\begin{defn} \label{defn:KMfan2} A \emph{KM fan} $(N,\{ P_\sigma \})$ is a FGA group $N$ together with a finite set $\{ P_\sigma \}$ of submonoids of $N$ satisfying the following properties: \begin{enumerate} \item Each $P_\sigma$ is a finitely generated, sharp, saturated monoid.  Here ``saturated" means that $P$ is ``intrinsically saturated" in the sense that \be P & = & \{ p \in P^{\rm gp} : \exists n \in \ZZ_{>0} {\rm \; such \; that \;} np \in P \} . \ee  It does \emph{not} mean that $P$ is ``saturated in $N$." \item We have $F \in \{ P_\sigma \}$ whenever $F$ is a face of some $P_\sigma \in \{ P_\sigma \}$. \item $P_\sigma \cap P_\tau$ is a face of both $P_\sigma$ and $P_\tau$ whenever $P_\sigma,P_\tau \in \{ P_\sigma \}$. \end{enumerate} \end{defn}

To go from a KM fan $(N,F,\{ F_\sigma \})$ in the sense of Definition~\ref{defn:KMfan} to a KM fan $(N,\{ P_\sigma : \sigma \in F \})$ in the sense of Definition~\ref{defn:KMfan2}, we set $P_\sigma := \sigma \cap F_\sigma$ (Notation~\ref{notation:abusive}) for $\sigma \in F$.  To go from a KM fan $(N,\{ P_\sigma \})$ in the sense of Definition~\ref{defn:KMfan2} to a KM fan $(N,F,\{ F_\sigma \})$ in the sense of Definition~\ref{defn:KMfan} we first get the fan $F$ by setting \be F & := & \{ C(P_\sigma) : P_\sigma \in \{ P_\sigma \} \} \ee where $C(P) \subseteq N_{\RR}$ denotes the cone associated to a finitely generated submonoid $P \subseteq N$ (the cone with generators given by any set of generators for $P$).  We then take $F_\sigma := P_\sigma^{\rm gp}$ as the lattice datum for the cone $C(P_\sigma) \in F$.  We leave it as an exercise for the reader to check that these two constructions are inverse.

Usually we work with Definition~\ref{defn:KMfan} because we think it will be more palatable to those familiar with ``classical" toric geometry.  The setup of Definition~\ref{defn:KMfan2} suggests a practical way to \emph{draw} a KM fan $(N,F,\{ F_\sigma \}) = (N, \{ P_\sigma \})$:  First we ``draw" the group $N$ and the cones in the fan $F$ in ``the usual way."  This can be a little tricky if $N$ has torsion.  A good convention is to pick an isomorphism \be N & \cong & \ZZ^r \oplus \ZZ / n_1 \ZZ \oplus \cdots \oplus \ZZ / n_k \ZZ \ee and then draw (or at least ``think of") each cyclic summand of $N$ as a different ``dimension."  One might as well then just draw the cones of $F$ in the ``free" dimensions (i.e.\ in $\RR^r = \ZZ^r \otimes \RR$) as usual, keeping in mind that in some sense these latter cones should be thought of as extending out ``infinitely" in each of the ``torsion dimensions."  (It is good to keep this in mind because the monoid $\sigma \cap N$ of lattice points in a cone always contains the torsion subgroup $N_{\rm tor}$.)  In this picture of $N$ and $N_{\RR}$, one should plot an open circle for each ``lattice point" (element of $N$).  Next, one fills in (darkens/shades) each open circle corresponding to lattice point belonging to some $P_\sigma$ (what we will call the ``fine support" of the KM fan).  The resulting picture is not particularly different from the way one draws a classical fan---we just draw a fan ``as usual," then carefully indicate which integral points of each cone are actually in $P_\sigma$.  A good policy is to always draw enough lattice points so that the filled in lattice points ``inside" each cone $\sigma$ (again, one has to keep in mind that $\sigma$ is viewed as extended infinitely into all torsion dimensions in defining ``inside" here) generate the monoid $P_\sigma$.

The setup of Definition~\ref{defn:KMfan2} also makes it easy to make the following

\begin{defn} \label{defn:smooth} A KM fan $(N,\{ P_\sigma \})$ is called \emph{smooth} iff each $P_{\sigma}$ is isomorphic to $\NN^n$ for some $n = n(\sigma)$. \end{defn}

\subsection{Examples} \label{section:examples} Here are some examples of KM fans and KM fan morphisms to keep in mind:

\begin{example} \label{example:classicalfan} {\bf (Classical fans)} Every ``classical" fan $F = (N,F)$ can be regarded as a KM fan $F = (N,F,\{ N_\sigma \})$ by taking the groups $N_{\sigma}$ (Definition~\ref{defn:Nsigma}) as the lattice data---note that $N$ is always understood to be a \emph{lattice}, hence the subgroups $N_\sigma \subseteq N$ are also lattices.  It is clear from this construction that a morphism of classical fans also induces a morphism between the associated KM fans and that, in fact, this construction is a fully faithful functor from classical fans to KM fans.  The essential image of this functor consists of those lattice KM fans $F = (N,F,\{ N_\sigma \})$ for which $F_\sigma = N_\sigma$ (equivalently, $F_\sigma$ is saturated in $N$) for every $\sigma \in F$.  By slight abuse of terminology, we often refer to such a KM fan as a ``classical fan." \end{example}

\begin{example} \label{example:GmA1} As a particular instance of Example~\ref{example:classicalfan}, we can view the classical fan in $\ZZ$ whose only cone is $\{ 0 \} \subseteq \ZZ_{\RR} = \RR$ as a KM fan, which we shall denote $\GG_m$.  (There should be no confusion with the \emph{scheme} $\GG_m = \Spec \ZZ[\ZZ]$ because we will always make it clear whether we are referring to the KM fan $\GG_m$ or the scheme $\GG_m$.)  Making standard abuses of notation, one might denote this KM fan $\GG_m = (\ZZ,0,0)$, though it would be more precise to write something like $\GG_m = (\ZZ,\{ \{ 0 \} \}, \{ \{ 0 \} \})$.  Similarly, we can view the classical fan in $\ZZ$ whose cones are $\{ 0 \}$ and $\RR_{\geq 0}$ as a KM fan, denoted $\AA^1$.  There is an evident map of KM fans $\GG_m \to \AA^1$.  Whenever we make reference to a map of KM fans $\GG_m \to \AA^1$, we of course mean to refer to this map unless we say otherwise. \end{example}

\begin{example} \label{example:grouptofan} {\bf (Zero fan)} For any FGA group $N$, we can consider the KM fan whose only cone is the zero cone; this is necessarily equipped with the only possible lattice datum---the zero lattice.  We shall call this KM fan the \emph{zero fan} associated to $N$ and denote it $(N,0,0)$, though more pedantic notation might have a couple curly brackets around the zeros.  This construction gives a functor $N \mapsto (N,0,0)$ from FGA groups to KM fans which is left adjoint to the functor $(N,F,\{ N_\sigma \}) \mapsto N$. \end{example} 

\begin{example} \label{example:coarsefan}  {\bf (Coarse fan)} Given any KM fan $F = (N,F,\{ F_\sigma \})$ we can construct a classical fan $\ov{F}$ in the lattice $\ov{N} := N / N_{\rm tor}$, called the \emph{underlying fan} (or \emph{coarse fan}) of $F$.  To do this, just note that the projection $N \to \ov{N}$ has finite kernel, hence induces an isomorphism $N_{\RR} = \ov{N}_{\RR}$ of vector spaces.  Since $F$ is, by definition, a fan in this vector space (with respect to $N$), so we can also view $F$ as a fan with respect to $\ov{N}$ (we denote the latter fan $\ov{F}$ for clarity) as in Remark~\ref{rem:cones}.  If we view $\ov{F}$ as a KM fan as in Example~\ref{example:classicalfan}, then the quotient projection $\pi : N \to \ov{N}$ defines a morphism of KM fans $\pi : F \to \ov{F}$.  To see this, first note that for each $\sigma \in F$, the map $\pi_{\RR}$ takes $\sigma$ isomorphically onto the corresponding cone $\ov{\sigma} \in \ov{F}$.  Next note that $F_\sigma \subseteq N_\sigma$ (by definition of a lattice datum) and $\pi$ takes $N_{\sigma}$ surjectively onto $\ov{N}_{\ov{\sigma}}$, so $\pi(F_\sigma) \subseteq \ov{N}_{\ov{\sigma}}$.

The construction $F \mapsto \ov{F}$ is functorial in the KM fan $F$.  The reader can easily check that this functor is left adjoint to the functor from classical fans to KM fans discussed in Example~\ref{example:classicalfan}. \end{example}

\begin{example} \label{example:rigidification} {\bf (Rigidification)} There is a slightly more refined version of the construction of Example~\ref{example:coarsefan} which is often useful.  Suppose $F = (N,F,\{ F_\sigma \})$ is a KM fan.  Let $\ov{N} := N / N_{\rm tor}$, $q : N \to \ov{N}$ the quotient map.  As in Example~\ref{example:coarsefan}, we can regard $F$ as a fan $\ov{F}$ in $\ov{N}_{\RR} = N_{\RR}$.  If $F_\sigma$ is the lattice datum for a cone $\sigma \in F$, then $q|F_\sigma : F_\sigma \to q(F_\sigma)$ is an isomorphism because $F_\sigma$ is a lattice, hence its intersection with $\Ker q = N_{\rm tor}$ is $\{ 0 \}$.  Then $F^{\rm rig} := (\ov{N},\ov{F}, \{ q(F_\sigma) \} )$ is a KM fan called the \emph{rigidification} of $F$.  Evidently $q$ defines a map of KM fans $q : F \to F^{\rm rig}$ which is readily seen to be initial among all KM fan maps from $F$ to a lattice KM fan.  The construction of $F^{\rm rig}$ is functorial in $F$ and defines a functor which is left adjoint to the inclusion of lattice KM fans into KM fans.  Note that $\ov{F} = \ov{F^{\rm rig}}$. \end{example}

\begin{example} \label{example:rootstacks} {\bf (Roots)} Suppose $F =(N,\{ P_\sigma \})=(N,F,\{ F_\sigma \})$ is a smooth KM fan (Definition~\ref{defn:smooth}).  Let $\rho_1,\dots,\rho_n \in F$ be the rays of $F$ and let $e_i \in P_{\rho_i} \cong \NN$ be the generators of the corresponding monoids.  Fix $a=(a_1,\dots,a_n) \in \ZZ_{>0}^n$.  Since $F$ is smooth, for each $\sigma \in F$, we have \be P_{\sigma} & = & \NN \langle e_{j_1}, \dots, e_{j_k} \rangle,\ee where $\rho_{j_1},\dots,\rho_{j_k}$ are the rays of $F$ contained in $\sigma$.  By replacing each $P_{\sigma}$ with \be aP_{\sigma} & := & \NN \langle a_{j_1}e_{j_1}, \dots, a_{j_k}e_{j_k} \rangle, \ee we obtain a new KM fan denoted $aF$.  The identity map $N \to N$ defines a map of KM fans $aF \to F$.  \end{example}

\begin{example} \label{example:dilation} {\bf (Dilation)} There is an interesting variant of the construction of Example~\ref{example:rootstacks} that makes sense for an \emph{arbitrary} KM fan $F=(N,F,\{ F_\sigma \})$.  Fix $a \in \ZZ_{>0}$.  Then $aF := (N,F, \{ aF_\sigma \})$ is a KM fan called the $a$-\emph{dilation} of $F$. \end{example}

The dilation construction of Example~\ref{example:dilation} can actually be built out of the two more general constructions discussed in the next two examples.

\begin{example} \label{example:inflation} {\bf (Inflation)} Suppose $F=(N,F,\{ F_\sigma \})$ is a KM fan and $N \subseteq N'$ is a finite index inclusion of groups.  As in Remark~\ref{rem:cones}, we can regard $F$ as a fan in $N'$, thus $F':=(N',F, \{ F_\sigma \})$ is also a KM fan, and the inclusion $N \into N'$ induces a morphism of KM fans $F \to F'$ called the \emph{inflation} of $F$ with respect to $N \subseteq N'$. \end{example}

\begin{example} \label{example:contraction} {\bf (Contraction)} Suppose $F=(N,F,\{ F_\sigma \})$ is a KM fan and $N' \subseteq N$ is a finite index inclusion.  Then $F' := (N',F,\{ F_\sigma \cap N' \})$ is also a KM fan and the inclusion $N' \into N$ defines a morphism $F' \to F$ called the \emph{contraction} of $F$ with respect to $N' \subseteq N$. \end{example}

\begin{defn} \label{defn:simplicial} Consider a KM fan $F=(N,F,\{ F_\sigma \})$ and a cone $\sigma \in F$.  The cone $\sigma$ is called \emph{simplicial} iff the dimension of $\Span \sigma$ is equal to the number of rays of $\sigma$.  The KM fan $F$ is called \emph{simplicial} iff each of its cones is simplicial in the previous sense.  Note that these are notions of cones and fans only and have no dependence on the lattice data for $F$. \end{defn} 

\begin{example} \label{example:canonicalresolution} {\bf (Canonical resolution)}  Suppose $F=(N,F,\{ F_\sigma \})$ is a simplicial KM fan.  Consider a cone $\sigma \in F$ and let $\rho_1, \dots, \rho_n$ denote the rays of $\sigma$.  Let $e_i \in F_{\sigma} \cap \rho_i = F_{\rho_i} \cap \rho_i$ be the primitive integral point along $\rho_i$ (the generator of the monoid $F_{\rho_i} \cap \rho_i \cong \NN$).  Since $F$ is simplicial, the $e_i$ freely generate a sublattice $G_\sigma \subseteq F_\sigma$ of finite index.  Furthermore, the monoid $G_\sigma \cap \sigma$ is freely generated by the $e_i$ and we have \be G_\sigma \cap \Span \tau & = & G_\tau \ee for $\tau \leq \sigma \in F$.  Thus $F' := (N,F,\{ G_\sigma \})$ is a smooth fan and the identity map $N \to N$ defines a morphism of KM fans $F' \to F$, called the \emph{canonical resolution} of $F$. \end{example}

The fun really begins when we consider KM fans where $N$ has torsion.

\begin{example} \label{example:P22} Let $N := \ZZ \oplus \ZZ/2 \ZZ$.  Let $F$ be the fan in $N_{\RR}$ whose cones are $\{ 0 \}$, $\sigma_+ := \RR_{\geq 0}$, and $\sigma_- := \RR_{\leq 0}$ in $N_{\RR} = \RR$.  Set $F_0 := \{ 0 \}$, $F_+ := \ZZ \langle (1,1) \rangle$, $F_- := \ZZ \langle (-1,0) \rangle$.  Notice that $F_0$, $F_+$, $F_-$ are lattice data for the cones $\{ 0 \}$, $\sigma_+$, $\sigma_-$, respectively, satisfying the compatibility condition $$F_{\pm} \cap \Span \{ 0 \} = \{ 0 \} = F_0,$$ so $F := (N,F,\{ F_0, F_+, F_- \})$ is a KM fan.  Regard the classical fan $\tilde{F}$ in $\tilde{N} = \ZZ^2$ whose cones are $\{ 0 \}$, $\{ 0 \} \times \RR_{\geq 0}$, and $\RR_{\geq 0} \times \{ 0 \}$ as a KM fan as in Example~\ref{example:classicalfan}.  Define $f : \tilde{N} \to N$ by the matrix \be f & := & \bp 1 & -1 \\ 1 & 0 \ep . \ee  Then $f$ defines a map of KM fans $\tilde{F} \to F$.  Without having any idea how we will define the realization of a map of KM fans, the reader may wish to speculate about the realization of $F$ and the map $f$.  See Example~\ref{example:P22revisited} for the answer.  \end{example}

\subsection{The fundamental lemmas} \label{section:fundamentallemma}  Here we will establish the key technical results that will enable us to define the gluings necessary to construct the various realizations of a KM fan in \S\ref{section:KMfanrealization}.

\begin{lem} \label{lem:torsor} Suppose $f : L \to L'$ is a map of lattices with finite cokernel, $\sigma$ (resp.\ $\sigma'$) is a cone in $L_{\RR}$ (resp.\ $L'_{\RR}$) and $f$ induces bijections \be f_{\RR} : \sigma & \to & \sigma' \\ f : \Span \sigma \cap L & \to & \Span \sigma' \cap L'.\ee  Let $B := \Cok (f^\lor)$, so that $S_\sigma(L) \subseteq L^\lor$ maps to $B$ via the composition $S_{\sigma}(L) \into L^\lor \to B$, hence we have an action of $\GG(B)$ on $\AA(S_\sigma(L))$ for any realization functor $\AA$ (\S\ref{section:realizationofmonoids}).  Then the map $\AA(S_\sigma(L)) \to \AA(S_{\sigma'}(L'))$ induced by $f$ is a $\GG(B)$ torsor under this action.  If $f$ is surjective, this torsor is (non-canonically) trivial. \end{lem}

\begin{proof} First suppose $f$ is surjective.  Set $K := \Ker f$.  We can \emph{choose} a splitting $L \cong L' \oplus K$ identifying $f$ with the projection $L' \oplus K \to L'$.  Here we have $B = K^\lor$.  The hypotheses ensure that, under this splitting, we have $\sigma \cong \sigma' \times \{ 0 \}$ in $L_{\RR} = L'_{\RR} \times K_{\RR}$.  It is then clear from the definitions that $S_\sigma(L) \cong S_{\sigma'}(L') \oplus K^\lor$ (isomorphism of monoids over $K^\lor$) in $L^\lor \cong (L')^\lor \oplus K^\lor$, hence $\AA(S_\sigma(L)) \to \AA(S_{\sigma'}(L'))$ is a trivial $\GG(B) = \GG(K^\lor)$ torsor.

Next suppose $f$ is a finite index inclusion $L \into L'$.  Here we will suppress notation for $f$ and view $L$ as a subgroup of $L'$.  By hypothesis we have $$ \Span \sigma \cap L = \Span \sigma' \cap L' =: F. $$  We can view $\sigma$ (or, equivalently, $\sigma'$) as a cone in $F_{\RR}$; let us call this cone $\ov{\sigma}$ for clarity.  The definition of $F$ ensures that the inclusions $F \into L$, $F \into L'$ are saturated (cf.\ Remark~\ref{rem:Nsigma}), so the quotients $Q := L/F$, $Q' := L'/F$ are free.  We have a map of short exact sequences of lattices as below.  \bne{SESmapAA} & \xym{ 0 \ar[r] & F \ar@{=}[d] \ar[r] & L \ar[r] \ar[d] & Q \ar[d] \ar[r] & 0 \\ 0 \ar[r] & F \ar[r] & L' \ar[r] & Q' \ar[r] & 0 } \ene  Since $L \into L'$ is injective, $Q \into Q'$ is injective and $L'/L = Q'/Q =: A$ by the Snake Lemma.  Since $Q'$ is free, we can \emph{choose} a section $s : Q' \to L'$ of the surjection $L' \to Q'$.  By diagram chasing, we see that $s|Q :Q \to L'$ takes values in $L \subseteq L'$ and provides a section of $L \to Q$.  Thus the sequences in \eqref{SESmapAA} can be ``compatibly split" so that \eqref{SESmapAA} is isomorphic to the diagram of lattices \bne{SESmapAAA} & \xym{ 0 \ar[r] & F \ar@{=}[d] \ar[r] & F \oplus Q \ar[r] \ar[d] & Q \ar[d] \ar[r] & 0 \\ 0 \ar[r] & F \ar[r] & F \oplus Q' \ar[r] & Q' \ar[r] & 0 } \ene where all the maps are the obvious ones.  Dualizing \eqref{SESmapAAA}, we obtain a map of short exact sequences of lattices \bne{SESmapAAdual} & \xym{ 0 \ar[r] & Q^\lor \ar[r] & F^\lor \oplus Q^\lor \ar[r] & F^\lor \ar[r] & 0 \\ 0 \ar[r] & (Q')^\lor \ar[u] \ar[r] & F^\lor \oplus (Q')^\lor \ar[u] \ar[r] & F^\lor \ar@{=}[u] \ar[r] & 0 } \ene where all the arrows are the obvious ones, all the vertical arrows are injective, and we have $$ Q^\lor / (Q')^\lor = L^\lor / (L')^\lor =: B = \EE(A). $$  The hypotheses and definitions ensure that, with respect to our compatible splittings $L \cong F \oplus Q$, $L' \cong F \oplus Q'$, we have $S_\sigma(L) = S_{\ov{\sigma}}(F) \oplus Q^\lor$ and $S_{\sigma'}(L') = S_{\ov{\sigma}}(F) \oplus (Q')^\lor$ inside $L^\lor \cong F^\lor \oplus Q^\lor$ and $(L')^\lor \cong F^\lor \oplus (Q')^\lor$.  Now we have a diagram of monoids \bne{push} & \xym{ 0 \ar[r] & (Q')^\lor \ar[r] \ar[d] & Q^\lor \ar[r] \ar[d] & B \ar[r] & 0 \\ & S_{\sigma'}(L') \ar[r] & S_\sigma(L) \ar@{.>}[ru] } \ene where the square is a pushout and the top row is an exact sequence of groups with $B$ finite.  

Any realization functor $\AA$ preserves finite inverse limits and makes $\AA(Q^\lor) \to \AA((Q')^\lor)$ a $\GG(B)$ torsor (Proposition~\ref{prop:tautorsors}).  Since $\GG(B)$ torsors are stable under base change, we conclude that $\AA(S_\sigma(L)) \to \AA(S_{\sigma'}(L'))$ is a $\GG(B)$ torsor, as desired.  (The action of $\GG(B)$ on $\AA(S_\sigma(L))$ defined via base change is the same as the action defined in the statement of the theorem---this is reflected by the fact that the map $S_\sigma(L) \to B$ defined in the statement of the lemma completes \eqref{push} as indicated to a commutative diagram.  It is also worth noting that, although we used our \emph{choice} of splitting $s$ as an expedient means to prove that the square in \eqref{push} is a pushout, the diagram \eqref{push} can be constructed without making such a choice of splitting---the groups $Q^\lor$ and $(Q')^\lor$ are canonically interpreted as the groups of units in the monoids $S_\sigma(L)$ and $S_{\sigma'}(L')$, respectively. (See Remark~\ref{rem:Ssigma}.)

Finally, for the general case, we factor $f : L \to L'$ as a surjection $L \to L''$ followed by a finite index inclusion $L'' \into L'$.  Let $\sigma'' \subseteq L''_{\RR}$ be the cone corresponding to $\sigma' \subseteq L'_{\RR}$ under the isomorphism $L''_{\RR} = L'_{\RR}$ induced by $L'' \into L'$.  The hypotheses on $f$ ensure that the maps $L \to L''$ (using the cones $\sigma$, $\sigma''$) and $L'' \into L'$ (using the cones $\sigma''$, $\sigma'$) also satisfy the same hypotheses.  We have inclusions $(L')^\lor \into (L'')^\lor \into L^\lor$ and hence a short exact sequence \bne{SESBB} & 0 \to & (L'')^\lor / (L')^\lor \to L^\lor / (L')^\lor \to L^\lor/(L'')^\lor \to 0. \ene  The realization of \eqref{SESBB} is a short exact sequence of group objects \bne{SESBB2} & 0 \to & \GG(L^\lor / (L'')^\lor) \to \GG(L^\lor / (L')^\lor) \to \GG((L'')^\lor/(L')^\lor) \to 0 \ene by Proposition~\ref{prop:tautorsors}.  The map $\AA(S_\sigma(L)) \to \AA(S_{\sigma'}(L'))$ factors as \bne{factorization} & \AA(S_\sigma(L)) \to \AA(S_{\sigma''}(L'')) \to \AA(S_{\sigma'}(L')).\ene  The results proved above show that the natural action of $\GG(L^\lor / (L'')^\lor)$ on $\AA(S_\sigma(L))$ makes the left map in \eqref{factorization} a trivial $\GG(L^\lor / (L'')^\lor)$ torsor and the natural action of $\GG((L'')^\lor/(L')^\lor)$ on $\AA(S_{\sigma''}(L''))$ makes the right map in \eqref{factorization} a $\GG((L'')^\lor/(L')^\lor)$ torsor.  Combining these results with the exact sequence \eqref{SESBB2}, one sees that the natural action of $\GG(L^\lor / (L')^\lor)$ on $\AA(S_\sigma(L))$ makes $\AA(S_\sigma(L)) \to \AA(S_{\sigma'}(L'))$ a $\GG(L^\lor / (L')^\lor)$ torsor.  \end{proof}

\begin{defn} \label{defn:AAsigma}  Let $\sigma$ be a cone in $N_{\RR}$, $F_\sigma$ a lattice datum for $\sigma$, $L \subseteq N$ a lifting of $F_\sigma$ (Definition~\ref{defn:latticedatum}).  Since the inclusion $L \subseteq N$ has finite index, we have $L_{\RR} = N_{\RR}$ and we can regard $\sigma$ as a cone in $L_{\RR}$ as in Remark~\ref{rem:cones}.  We can then consider the submonoid $S_\sigma(L) \subseteq L^\lor$ as in Definition~\ref{defn:Ssigma}.  Dualizing the short exact sequence $$0 \to L \to N \to N/L \to 0,$$ we obtain an exact sequence \bne{fundamentalSES} & 0 \to N^\lor \to L^\lor \to \EE(N/L) \to \EE(N) \to 0.\ene  The composition of the inclusion $S_\sigma(L) \into L^\lor$ and the map $L^\lor \to \EE(N/L)$ can be viewed as an action of the group object $\EE(N/L)$ on $S_\sigma(L)$ in $\Mon^{\rm op}$.  All realization functors $\AA$ (\S\ref{section:realizationofmonoids}) preserve finite inverse limits, so we obtain an induced action of $\GG(\EE(N/L))$ on $\AA(S_\sigma(L))$.  Let \be \AA(\sigma,N,F_\sigma,L) & := & [ \AA(S_\sigma(L)) / \GG(\EE(N/L)) ]  \ee be the stack-theoretic quotient.  (If $N$ and/or $F_\sigma$ are/is clear from context we drop them from the notation in $\AA(\sigma,N,F_\sigma,L)$, writing, for example, $\AA(\sigma,L)$ in lieu of $\AA(\sigma,N,F,L)$.)   We will prove in Lemma~\ref{lem:fundamental} that $\AA(\sigma,N,F_\sigma,L)$ is ``independent" of the choice of lifting $L$.   Since any lattice datum admits at least one lifting (Lemma~\ref{lem:liftingsexist}), we can unambiguously write $\AA(\sigma,N,F_\sigma)$ instead of $\AA(\sigma,N,F_\sigma,L)$.  (We similarly shorten this to $\AA(\sigma,N)$ or even just $\AA(\sigma)$ if there is no risk of confusion.) \end{defn}

\begin{rem} \label{rem:torusaction1} When $N$ is a lattice we have $\EE(N)=0$, so applying $\GG( \slot )$ to \eqref{fundamentalSES} yields an exact sequence \bne{GfundamentalSES} & 0 \to \GG(\EE(N/L)) \to \GG(L^\lor) \to \GG(N^\lor) \to 0 \ene of sheaves of abelian groups on $\Esp$.  Therefore Lemma~\ref{lem:quotients} gives a natural action of the ``torus" $\GG(N^\lor)$ on $\AA(\sigma,L)$.  This naturality ensures, in particular, that when $\tau \leq \sigma$, the map $\AA(\tau,L) \to \AA(\sigma,L)$ is $\GG(N^\lor)$ equivariant. \end{rem}

\begin{lem} \label{lem:fundamental}  Suppose $\sigma \subseteq N_{\RR}$ is a cone with lattice datum $F_\sigma$ and $L_1,L_2 \subseteq N$ are two liftings of $F_{\sigma}$ (Definition~\ref{defn:latticedatum}).  Then there is a canonical isomorphism of stacks \be \AA(\sigma,L_1) & = & \AA(\sigma,L_2).\ee  When $N$ is a lattice this is an isomorphism of stacks with $\GG(N^\lor)$ action as in Remark~\ref{rem:torusaction1}. \end{lem}

\begin{proof} One sees easily that $L_{12} := L_1 \cap L_2$ is also a lifting of $F_{\sigma}$.  By considering the inclusions $L_{12} \subseteq L_i$ ($i=1,2$) we reduce to treating the case where $L_1 \subseteq L_2$.  In this case, we have a map of short exact sequences as below. \bne{diagramA2} & \xym{ 0 \ar[r] & L_1 \ar[r] \ar[d] & N \ar@{=}[d] \ar[r] & N/L_1 \ar[r] \ar[d] & 0 \\ 0 \ar[r] & L_2 \ar[r] & N \ar[r] & N / L_2 \ar[r] & 0 } \ene  The map $N/L_1 \to N/L_2$ is surjective and we have $$ L_2/L_1 = \Ker(N/L_1 \to N/L_2) =: A$$ by the Snake Lemma.  Dualizing \eqref{diagramA2} we obtain a commutative diagram \bne{diagramB2} & \xym{ 0 \ar[r] & N^\lor \ar@{=}[d] \ar[r] & L_1^\lor \ar[r] & \EE(N/L_1) \ar[r] & \EE(N) \ar@{=}[d] \ar[r] & 0 \\ 0 \ar[r] & N^\lor \ar[r] & L_2^\lor \ar[r] \ar[u] & \EE(N/L_2) \ar[r] \ar[u] & \EE(N)  \ar[r] & 0 } \ene where the two middle vertical arrows are injective with the same cokernel, $\EE(A)$.  

The realization of the exact sequence \bne{lemfund1} & 0 \to \EE(N/L_2) \to \EE(N/L_1) \to \EE(A) \to 0 \ene yields a short exact sequence of sheaves of abelian groups \bne{GGlemfund1} & 0 \to \GG(\EE(A)) \to \GG(\EE(N/L_1)) \to \GG(\EE(N/L_2)) \to 0. \ene  The natural map $\AA(S_\sigma(L_1)) \to \AA(S_\sigma(L_2))$ is a $\GG(\EE(A))$ torsor by Lemma~\ref{lem:torsor}, so that $[\AA(S_\sigma(L_1)) / \GG(\EE(A))] = \AA(S_\sigma(L_2))$.  We compute \bne{lemfundcomputation} \AA(\sigma,L_1) & = & [ \AA(S_\sigma(L_1)) / \GG(\EE(N/L_1)) ] \\ \nonumber & = & [[ \AA(S_\sigma(L_1)) / \GG(\EE(A)) ] / \GG(\EE(N/L_2)) ] \\ \nonumber & = & [ \AA(S_\sigma(L_2)) / \GG(\EE(N/L_2)) ] \\ \nonumber & = & \AA(\sigma,L_2) \ene using Lemma~\ref{lem:quotients} (applied to the exact sequence \eqref{GGlemfund1} and the object $X = \AA(S_\sigma(L_1))$) for the second ``equality" (natural isomorphism). 

Suppose furthermore that $N$ is a lattice.  Then \eqref{lemfund1} is part of an exact diagram \bne{bigexactdiagram} & \xym{ & 0 \ar[d] & 0 \ar[d] \\ 0 \ar[r] & N^\lor \ar@{=}[r] \ar[d] & N^\lor \ar[r] \ar[d] & 0 \ar[d] \\ 0 \ar[r] & L_2^\lor \ar[r] \ar[d] & L_1^\lor \ar[r] \ar[d] & \EE(L_2/L_1) \ar@{=}[d] \ar[r] & 0 \\ 0 \ar[r] & \EE(N/L_2) \ar[r] \ar[d] & \EE(N/L_1) \ar[r] \ar[d] & \EE(A) \ar[r] \ar[d] & 0 \\ & 0 & 0 & 0 } \ene whose realization yields an exact diagram of sheaves of abelian groups.  By using this diagram and the naturality in Lemma~\ref{lem:quotients} we see that the isomorphisms in \eqref{lemfundcomputation} are $\GG(N^\lor)$ equivariant. \end{proof}

The above lemma will be needed to define the realization of a KM fan in the next section.  To define the realization of a \emph{morphism} of KM fans (Definition~\ref{defn:KMfanmorphism}), we shall also need Lemma~\ref{lem:liftings} below.

\begin{lem} \label{lem:representability} Let $F=(N,F,\{F_\sigma \})$ be a KM fan, $\sigma \in F$ a cone of $F$.  Suppose that $A$ is a subgroup of $N$ containing $F_\sigma$ with $A/F_\sigma$ torsion free.  (This hypothesis holds when $A$ is a lifting of $F_\sigma$, or if $A = F_\tau$ for some $\tau \in F$ containing $\sigma$.)  Then the surjective map of abelian groups $N/F_\sigma \to N/A$ is injective on torsion subgroups (equivalently, its kernel is torsion-free). \end{lem}

\begin{proof} Apply the Snake Lemma to $$ \xym{ 0 \ar[r] & F_\sigma \ar[r] \ar[d] & N \ar@{=}[d] \ar[r] & N/F_\sigma \ar[d] \ar[r] & 0 \\ 0 \ar[r] & A \ar[r] & N \ar[r] & N/A \ar[r] & 0 }$$ to see that the kernel of $N/F_\sigma \to N/A$ is $A/F_\sigma$, which is torsion free by assumption. \end{proof}

\begin{lem} \label{lem:liftings} Let $f : (N,F,\{ F_\sigma \}) \to (N',F',\{ F'_\tau \})$ be a morphism of KM fans, $\sigma \in F$, $\tau \in F'$ a cone containing $f_{\RR}(\sigma)$, $L' \subseteq N'$ a lifting of $F'_{\tau}$ (Definition~\ref{defn:latticedatum}).  \begin{enumerate} \item \label{liftings1} There exists a lifting $L \subseteq N$ of $F_\sigma$ so that $f(L) \subseteq L'$.  \item \label{liftings2} If $f_{\rm tor} : N_{\rm tor} \to N'_{\rm tor}$ is injective (equivalently $\Ker f$ is torsion-free, which holds, in particular, whenever $N$ is a lattice) and $f|F_\sigma : F_\sigma \to F'_\tau$ is bijective, then the map $(N/F_\sigma)_{\rm tor} \to (N'/F'_\tau)_{\rm tor}$ induced by $f$ is injective. \item \label{liftings3} If $(N/F_\sigma)_{\rm tor} \to (N'/F'_\tau)_{\rm tor}$ is injective, then $L := f^{-1}(L')$ is a lifting of $F_\sigma$.  For this choice of $L$, the map of finite groups $N/L \to N'/L'$ induced by $f$ is injective. \item \label{liftings4} If there is a lifting $L$ of $F_\sigma$ such that $f(L) \subseteq L'$ and $N/L \to N'/L'$ is injective, then $(N/F_\sigma)_{\rm tor} \to (N'/F'_\tau)_{\rm tor}$ is injective. \end{enumerate} \end{lem}

\begin{proof} \eqref{liftings1}:  By Lemma~\ref{lem:liftingsexist} there is \emph{some} lifting $L''$ of $F_\sigma$.  The map $N/f^{-1}(L') \to N'/L'$ is clearly injective and $L'$ has finite index in $N'$, so $f^{-1}(L')$ has finite index in $N$.  By definition of a morphism of KM fans (cf.\ Remark~\ref{rem:KMfan}), $f(F_\sigma) \subseteq F'_\tau \subseteq L'$, so we have $F_\sigma \subseteq \Span \sigma \cap f^{-1}(L')$.  Now if we set $L := L'' \cap f^{-1}(L')$, then since $\Span \sigma \cap L'' = F_\sigma$ we have $\Span \sigma \cap L = F_\sigma$.  Certainly $L$ is a lattice since it is contained in the lattice $L''$.  Furthermore, $L$ has finite index in $N$ since $f^{-1}(L')$ and $L''$ have finite index in $N$, so $L$ is as desired.

\eqref{liftings2}:  Suppose $n \in N$ is an element of $N$ with $f(n) \in F'_\tau$ and $mn \in F_\sigma$ for some positive integer $m$.  To see that $(N/F_\sigma)_{\rm tor} \to (N'/F'_\tau)_{\rm tor}$ is injective, we need to show that any such $n$ is in $F_\sigma$.  Since $f|F_\sigma : F_\sigma \to F'_\tau$ is surjective, there is some $\ov{n} \in N$ with $f(\ov{n})= f(n)$.  Since $f|F_\sigma : F_\sigma \to F'_\tau$ is injective and $mn, m \ov{n} \in F_\sigma$ have the same image under $f$, we have $mn=m\ov{n}$, hence $n-\ov{n} \in N_{\rm tor}$.  Since we also have $f(n-\ov{n})=0$ and $f_{\rm tor} : N_{\rm tor} \to N'_{\rm tor}$ is injective, we conclude that $n=\ov{n} \in F_\sigma$, as desired.

\eqref{liftings3}:  Assume $(N/F_\sigma)_{\rm tor} \to (N'/F'_\tau)_{\rm tor}$ is injective.  We first show that $L = f^{-1}(L') \subseteq N$ is torsion-free (hence a lattice).  Suppose $n \in L$ is a torsion element.  Then $f(n)$ is a torsion element of the lattice $L'$, hence $f(n)=0$.  If $n \notin F_\sigma$, then $n$ would be a non-trivial torsion element in the kernel of $N/F_\sigma \to N'/F'_\tau$, contradicting our injectivity assumption.  So we conclude that $n \in F_\sigma$, hence $n=0$ because $F_\sigma$ is a lattice.

Much as in the proof of \eqref{liftings1} above, we see easily that $L$ has finite index in $N$ and that $F_\sigma \subseteq L \cap \Span \sigma$, so it remains only to establish the opposite containment.  Suppose $n \in L \cap \Span \sigma$.  Then $f(n) \in L \cap \Span \tau = F'_\tau$.  Furthermore, $n \in N \cap \Span \sigma = N_\sigma$ and $F_\sigma$ has finite index in $N_\sigma$, so $mn \in F_\sigma$ for some $m \in \ZZ_{>0}$.  Putting this all together, we conclude that the image of $n$ in $N/F_\sigma$ is in the kernel of $(N/F_\sigma)_{\rm tor} \to (N'/F'_\tau)_{\rm tor}$, so our injectivity assumption implies that $n \in F_\sigma$, as desired.

\eqref{liftings4}:  For such an $L$, we have a commutative diagram $$ \xym{ (N/F_\sigma)_{\rm tor} \ar[r] \ar[d] & N/L=(N/L)_{\rm tor} \ar[d] \\ (N'/F'_\tau)_{\rm tor} \ar[r] & N'/L'=(N'/L')_{\rm tor} } $$ where the horizontal arrows are injective by Lemma~\ref{lem:representability} and the right vertical arrow is injective by assumption, hence the left vertical arrow must also be injective, as desired. \end{proof}

\subsection{Realizations of KM fans} \label{section:KMfanrealization}  Let  $F = (N,F,\{ F_\sigma \})$ be a KM fan, $\tau \leq \sigma$ cones of $F$.  Pick a lifting (Definition~\ref{defn:latticedatum}) $L$ of $F_\sigma$.  Then $L$ is also a lifting of $F_\tau$ (Remark~\ref{rem:liftings}).  As in the usual theory of toric varieties, the inclusion of monoids $S_\sigma(L) \into S_\tau(L)$ is the localization of $S_\sigma(L)$ at the face $S_\sigma(L) \cap \tau^{\perp}$, so its realization $\AA(S_\tau(L)) \into \AA(S_\sigma(L))$ is an open embedding.  It is clear from the definition of the relevant actions (Definition~\ref{defn:AAsigma}) that this open embedding is $\GG(\EE(N/L))$ equivariant, hence it gives rise to an open embedding \bne{openembeddings} \AA(\tau,L) & \into & \AA(\sigma,L)  \ene of the corresponding quotient stacks.  It is clear from the proof of Lemma~\ref{lem:fundamental} that the open embedding \eqref{openembeddings} is also ``independent of the choice of $L$", so it may be viewed as an open embedding \bne{openembedding2} \AA(\tau) & \into & \AA(\sigma).  \ene  Furthermore, it is clear---as in the usual theory of toric varieties---that the open embeddings \eqref{openembedding2} are functorial in the inclusion of cones $\tau \leq \sigma$.

\begin{defn} \label{defn:realization}  The \emph{realization} of the KM fan $F$ is defined to be stack $\AA(F)$ obtained by ``gluing the stacks $\AA(\sigma)$ along the open embeddings \eqref{openembedding2}."  More precisely, $\AA(F)$ is ``the" $2$-direct limit of the stacks $\AA(\sigma)$ and the open embeddings \eqref{openembedding2}, taken over the poset of cones $\sigma$ of $F$ ordered by inclusion. \end{defn}

\begin{example} \label{example:classicalrealization} Suppose $F = (N,F)$ is a ``classical" fan, regarded as a KM fan as in Example~\ref{example:classicalfan}.  Then the realization $\AA(F)$ of $F$ defined above is the usual ``classical" realization.  This is because, for such a KM fan, one may take $L=N$ as the lifting of any lattice datum for any cone, making it clear that $$\AA(\sigma) = \AA(\sigma,L=N) = \AA( S_\sigma(N) )$$ is the usual realization. \end{example}

As in the usual theory of toric varieties, the structure maps $\AA(\sigma) \to \AA(F)$ to the direct limit are open embeddings satisfying \bne{intersectionformula} \AA(\sigma) \cap \AA(\tau) & = & \AA(\sigma \cap \tau).\ene  In particular, as $\sigma$ runs over the \emph{maximal} cones of $F$, the open substacks $\AA(\sigma)$ cover $\AA(F)$.  

\begin{rem} \label{rem:torusaction2} When $F$ is a lattice KM fan, we have an action of the torus $\GG(N^\lor)$ on each $\AA(\sigma)$ making the inclusions $\AA(\tau) \into \AA(\sigma)$ equivariant (see Remark~\ref{rem:torusaction1} and the final part of Lemma~\ref{lem:fundamental}), hence we obtain an action of the torus $T := \GG(N^\lor)$ on $\AA(F)$ so that each open substacks $\AA(\sigma)$ is $T$ invariant.  \end{rem}

Some features of this realization construction are summarized below.

\begin{prop} \label{prop:realization}  Given a KM fan $F$ and a choice of lifting of each lattice datum, we obtain a cover of the realization $\AA(F)$ by open substacks of the form $\AA(\sigma) = [ \AA(Q_\sigma) / \GG(E_\sigma) ] $ (one for each cone $\sigma$ in the fan $F$) where each $Q_\sigma$ is a toric monoid, each $E_\sigma$ is a finite (abelian) group, and $\GG(E_\sigma)$ acts on $\AA(Q_\sigma)$ via the realization of a group homomorphism $a_\sigma : Q_\sigma^{\rm gp} \to E_\sigma$.  In particular, $\AA(F)$ has a $\tau$-cover by a finite disjoint union of realizations of toric monoids.  It follows that the algebraic realization $X(F)$ of a KM fan is a Deligne-Mumford stack (we shall see in Corollary~\ref{cor:separated} that it is separated) with an \'etale cover by a finite disjoint union of affine toric varieties, hence it enjoys any property which is \'etale local in nature and which is enjoyed by any toric variety (e.g.\ it is normal, Cohen-Macaulay, etc.).  

Suppose furthermore that $F$ is a lattice KM fan.  Then each $a_\sigma$ is the cokernel map in an exact sequence \bne{realizationSES} 0 \to N^\lor \to Q_\sigma^{\rm gp} \to E_\sigma \to 0, \ene so $\GG(E_\sigma)$ acts as a subgroup of the torus $\GG(Q_\sigma^{\rm gp})$.  In this case the open cover is $T := \GG(N^\lor)$ equivariant (for the action of Remark~\ref{rem:torusaction2}), with $T$ acting on $[ \AA(Q_\sigma) / \GG(E_\sigma) ]$ via the natural action of Lemma~\ref{lem:quotients} coming from the realization $\GG$ of \eqref{realizationSES}.  When $\sigma=0$ is the zero cone, we can take $N$ itself as a lifting of the lattice datum $F_0=0$, thus we see that $\AA(0)=T$ (with the $T$ action given by left multiplication) is a $T$-invariant open substack of $\AA(F)$ (cf.\ the ``furthermore" in Lemma~\ref{lem:quotients}).  The open substack $T \subseteq \AA(F)$ is contained in each open substack $[ \AA(Q_\sigma) / \GG(E_\sigma) ]$ as $T= [ \GG(Q_\sigma^{\rm gp}) / \GG(E_\sigma) ]$.  In particular, it follows that $T = \Spec k[N^\lor]$ is dense in the algebraic realization $X(F)$ since the torus $\GG(Q_\sigma^{\rm gp}) = \Spec k[Q_\sigma^{\rm gp}]$ is dense in the affine toric variety $X(Q_\sigma) = \Spec k[Q_\sigma]$. \end{prop}

The realization $\AA(f) : \AA(F) \to \AA(F')$ of a \emph{morphism} of KM fans $f : (N,F,\{ F_\sigma \}) \to (N',F',\{ F'_\tau \})$ is defined as follows:  Let $\sigma$, $\tau$, $L'$ and $L$ be as in Lemma~\ref{lem:liftings}.  Dualizing the map of short exact sequences \bne{SESmapf}  & \xym{ 0 \ar[r] & L \ar[d]_f \ar[r] & N \ar[d]_f \ar[r] & N/L \ar[d] \ar[r] & 0 \\ 0 \ar[r] & L' \ar[r] & N' \ar[r] & N'/L' \ar[r] & 0, } \ene we obtain a map of exact sequences \bne{LESmapf} & \xym{ 0 \ar[r] & N^\lor \ar[r] & L^\lor \ar[r] & \EE(N/L) \ar[r] & \EE(N) \ar[r] & 0 \\ 0 \ar[r] & (N')^\lor \ar[u]_{f^\lor} \ar[r] & (L')^\lor \ar[u]_{f^\lor}  \ar[r] & \EE(N'/L') \ar[r] \ar[u] & \EE(N') \ar[r] \ar[u] & 0. } \ene  Since $f_{\RR}(\sigma) \subseteq f_{\RR}(\tau)$, the map $f^\lor : (L')^\lor \to L^\lor$ induces a map of monoids \bne{monoidmapf} f^\lor : S_\tau(L') & \to & S_\sigma(L). \ene  From the commutativity of the middle square in \eqref{LESmapf}, we see that the realization \bne{monoidmapfrealization} \AA(f^\lor) : \AA(S_\sigma(L)) & \to & \AA(S_\tau(L')) \ene of \eqref{monoidmapf} is equivariant with respect to the map of group objects \be \GG( \EE(N/L)) & \to & \GG(\EE(N'/L')), \ee hence it yields a map \bne{localf} \AA(\sigma,F,L) & \to & \AA(\tau,F',L') \ene between the corresponding quotient stacks.  Now, by following this story through the proof of Lemma~\ref{lem:fundamental}, one checks that the maps \eqref{localf} do not depend on the choice of lifting $L'$ or the choice of $L$ as in Lemma~\ref{lem:liftings}, so we obtain unambiguous maps \bne{localf2} \AA(\sigma,F) & \to & \AA(\tau,F'). \ene  One also sees readily that the ``local maps" \eqref{localf2} are compatible with the open embeddings \eqref{openembedding2} associated to inclusions of cones in $F$ and $F'$, hence we obtain an induced map $\AA(f) : \AA(F) \to \AA(F')$ on the direct limits, called the \emph{realization of} $f$.  When $N$ and $N'$ are lattices, we have $\EE(N)=\EE(N')=0$, and one can see from the commutativity of \eqref{LESmapf} and the naturality in Lemma~\ref{lem:quotients} that the realization of $f$ is equivariant with respect to the map $\GG(N^\lor) \to \GG((N')^\lor)$ induced by $f : N \to N'$.

Let us emphasize the following basic feature of this realization construction:  In the notation above, we have a commutative (but not generally cartesian) diagram of stacks over spaces \bne{localfsquare} & \xym{ \AA(S_\sigma(L)) \ar[r] \ar[d] & \AA(\sigma,F) \ar[d] \ar@{^(->}[r] & \AA(F) \ar[d]^{\AA(f)} \\ \AA(S_\tau(L')) \ar[r] & \AA(\tau,F') \ar@{^(->}[r] & \AA(F') } \ene where: \begin{enumerate} \item The leftmost vertical arrow is a map of ``affine toric varieties"---it is the map of spaces obtained via realization of an evident map of \emph{classical} fans (each with a unique maximal cone) lying over the map of lattices $L \to L'$. \item The left horizontal arrows are $\tau$-covers (``\'etale covers"), since each is a $\tau$-locally trivial torsor under some group object $\GG(A)$ for a finite abelian group $A$, and each $\GG(A)$ is a $\tau$-cover of the terminal object by our assumptions about $\tau$. \item The right horizontal arrows are open embeddings. \end{enumerate}  Questions about $\AA(f)$ which are ``\'etale local (on $\AA(f)$)" in nature can generally be reduced to questions about maps of ``toric varieties" (realizations of maps of \emph{classical} fans) by the consideration of the squares \eqref{localfsquare}.  We shall see many examples of this later.

\begin{rem} \label{rem:preimages} One can check, as in the usual theory of toric varieties, that the preimage $\AA(f)^{-1}(\AA(\tau))$ of the open substack $\AA(\tau) \subseteq \AA(F')$ is the open substack $\AA(F)_\tau$ of $\AA(F)$ given by the union of the ``basic" open substacks $\AA(\sigma)$ over all cones $\sigma$ of $F$ for which $f_{\RR}(\sigma) \subseteq \tau$.  The open embedding $\AA(F)_\tau \into \AA(F)$ is nothing but the realization of the inclusion of the KM subfan $F^\tau$ of $F$ whose cones are the cones $\sigma$ of $F$ with $f_\RR(\sigma) \subseteq \tau$ (the lattice data for $F^\tau$ are inherited from $F$). \end{rem}

\begin{example} \label{example:atomicopen}  Since every fan $F$ contains the zero cone, the realization of every KM fan $F = (N,F,\{ F_\sigma \})$ contains a distinguished open subspace $\AA(0) \subseteq \AA(F)$, which we shall call the \emph{atomic open}.  As is clear from the construction of realizations, the atomic open is the realization $\AA(N,0,0)$ of the KM fan $(N,0,0)$ attached to the group $N$ as in Example~\ref{example:grouptofan} and the inclusion $\AA(0) \subseteq \AA(F)$ is the realization of the unit $(N,0,0) \to (N,F, \{ F_\sigma \})$ of the adjunction described in that example.  As remarked above, the atomic open is functorial under maps of KM fans.  

Let us describe the realization $\AA(N,0,0)$ explicitly.  When $N$ is a lattice, we can take $N$ itself as a lifting of the lattice datum $F_0=0$ for the zero cone; thus we see that the atomic open $\AA(N,0,0)$ is the torus $\GG(N^\lor)$.  For a general $N$, we can \emph{choose} a section $s : \ov{N} \to N$ of the projection $N \to \ov{N} = N / N_{\rm tor}$ (i.e.\ a splitting $N \cong \ov{N} \oplus N_{\rm tor}$) and then use $L := s(\ov{N}) \subseteq N$  as a lifting of $F_0=0$.  Choosing $L$ in this manner ensures that the natural maps $N^\lor \to L^\lor$ and $\EE(N/L) \to \EE(N)$ are isomophisms and the natural map $L^\lor \to \EE(N/L)=\EE(N)=\EE(N_{\rm tor})$ is the zero map.  In terms of the construction of the realization $\AA(N,0,0)$ given above, this means that the action of $\GG(\EE(N/L))$ on $$\AA(S_0(L)) = \GG(L^\lor) = \GG(N^\lor) = \GG(\ov{N}^\lor)$$ is the trivial action, so that \be \AA(N,0,0) & \cong & [ \AA(S_0(L)) / \GG(\EE(N/L)) ] \\ & = & \GG(\ov{N}^\lor) \times B \GG(\EE(N_{\rm tor})) \ee is a trivial $\GG(\EE(N_{\rm tor}))$ gerbe over the torus $\GG(\ov{N}^\lor) = \AA(\ov{N},0,0)$.  The structure map $$\AA(N,0,0) \to \AA(\ov{N},0,0) = \GG(\ov{N}^\lor)$$ realizing $\AA(N,0,0)$ as such a trivial gerbe is natural in $N$ (it is just the realization of the map of KM fans $(N,0,0) \to (\ov{N},0,0)$ obtained from the natural map of groups $N \to \ov{N}$), but the \emph{trivialization} of this gerbe is \emph{not} canonical---it depends on the choice of splitting made above. \end{example}

\begin{defn} \label{defn:KMstack} A stack isomorphic to the realization of a KM fan (resp.\ lattice KM fan) will be called a \emph{KM stack} (resp.\ a \emph{toric KM stack}). \end{defn}

\section{Basic Results and Constructions} \label{section:basicresults}  In this section we will establish some fundamental results about KM fans (Definition~\ref{defn:KMfan}) and their realizations (\S\ref{section:KMfanrealization}).  All proofs which do not substantially differ from their ``classical" analogues will be left to the reader.

\subsection{Equidimensionality and reduced fibers}  \label{section:equidimensional} Here is a typical example of a statement about realizations of KM fans which is easily reduced to the ``classical case" by the method discussed at the end of \S\ref{section:KMfanrealization}:

\begin{prop} \label{prop:equidimensional} Let $f : (N,F,\{ F_\sigma \}) \to (N',F', \{ F'_{\sigma} \} )$ be a map of KM fans such that $\Cok( f : N \to N')$ is finite.  Then the following are equivalent: \begin{enumerate} \item For every cone $\sigma \in F$, the cone $\sigma' := f_{\RR}(\sigma)$ is in $F'$. \item All geometric fibers of the algebraic realization $X(f)$ of $f$ are of the same dimension. \end{enumerate}  Assume these equivalent conditions hold.  Then the following are equivalent: \begin{enumerate} \item For every cone $\sigma \in F$, the map $f : F_\sigma \to F'_{\sigma'}$ is surjective. \item All geometric fibers of the algebraic realization $X(f)$ of $f$ are reduced. \end{enumerate} \end{prop}

\begin{proof} The ``equidimensional" and ``reduced fibers" statements are all \'etale local on $X(f)$, so one can easily reduce to the case of classical KM fans by considering the diagrams \eqref{localfsquare} (running over all $\sigma, \tau$).  (Exercise!)  For classical fans, these are results of Abramovich and Karu \cite[4.1, 5.2]{AK} \cite[5.2]{Kar}. \end{proof}

\subsection{Coarse moduli space} \label{section:coarsemodulispace} Here we will show, among other things, that for a KM fan $F$, the algebraic realization $X(\pi)$ of the map $\pi : F \to \ov{F}$ discussed in Example~\ref{example:coarsefan} is the ``coarse moduli space" of the stack $X(F)$.  \emph{In this section we shall consider the algebraic realization \emph{without} log structure, to avoid confusion about the meaning of ``coarse moduli space."}  For clarity, we begin by recalling the basic ``Keel Mori Theorem" (see \cite{KeM} and \cite{Con}) concerning the existence and properties of coarse moduli spaces.  Although Keel and Mori (and especially Conrad) work with fairly general algebraic stacks (over fairly general bases), we restrict ourselves here to the case of Deligne-Mumford stacks (over a fixed base field, which will always be of characteristic zero in our applications).

\begin{thm}[Keel-Mori Theorem] \label{thm:cms} Let $\mathcal{X}$ be a Deligne-Mumford stack.  Then there exists an algebraic space $\ov{\mathcal{X}}$ and a morphism $\pi : \mathcal{X} \to \ov{\mathcal{X}}$ of stacks, called the ``coarse moduli space" of $\mathcal{X}$ with the following properties: \begin{enumerate} \item The map $\pi$ is initial among maps of stacks from $\mathcal{X}$ to an algebraic space.  (This characterizes $\pi$ up to unique isomorphism.) \item For any algebraically closed field $K$ (containing our base field), $\pi$ induces a bijection from the set of isomorphism classes in the category $\mathcal{X}(K)$ to the set $\ov{\mathcal{X}}(K)$ of $K$ points of $\ov{\mathcal{X}}$. \item $\ov{\mathcal{X}}$ is separated iff $\mathcal{X}$ is separated. \item The map $\pi$ is proper, quasi-finite, and a universal homeomorphism. \item For any flat (and locally finite type) map $f : \ov{\mathcal{X}}' \to \ov{\mathcal{X}}$, the base change $\pi' : \mathcal{X}' \to X'$ of $\pi$ along $f$ also satisfies the above properties. \end{enumerate} \end{thm}

It should be noted that, strictly speaking, we never really need to appeal to this theorem anywhere in \S\ref{section:coarsemodulispace} since, for the stacks we will consider, we will directly construct $\pi$ and we \emph{could} (but don't) also directly verify all of its properties.  We shall reduce everything to the following simple setup:

\begin{setup} \label{setup:finitequotient} Let $P$ be a fine monoid, $E$ a \emph{finite} abelian group, $a : P^{\rm gp} \to E$ a group homomorphism, $Q$ the submonoid of $P$ consisting of those $p \in P$ for which $a(p)=0$.  In other words, $Q$ is defined by the equalizer diagram \bne{eqdiagram} & Q \to P \rightrightarrows P \oplus E \ene where the parallel arrows are given by the ``projection" $p \mapsto (p,0)$ and the ``action" $p \mapsto (p,a(p))$. \end{setup}

In the above setup, the monoid $Q$ is finitely generated because it is a general principle that the inverse limit of a finite diagram of finitely generated monoids is finitely generated.  Hence $Q$ is fine since, being a submonoid of the integral monoid $P$, it is integral.  For any $p \in P$, we have $|E|p \in Q$, so $Q$ is \emph{dense} in $P$ in the sense of \cite[\S1.7]{GM1}.  By Gordan's Lemma \cite[Theorem~1.7.2(2)]{GM1}, the inclusion $Q \into P$ is finite in the sense that $P$ is finitely generated as a $Q$-module.

Let $\AA$ be a realization functor from finitely generated monoids to spaces, satisfying the properties discussed in \S\ref{section:realizationofmonoids} (in particular, we assume that ``spaces" is equipped with a topology $\tau$ satisfying some reasonable properties).  Since the two compositions in \eqref{eqdiagram} agree, the same is true of the two compositions in its realization \bne{eqdiagramrealization} &  \AA(P) \times \GG(E) \rightrightarrows \AA(P) \to \AA(Q), \ene so we have a map of stacks \bne{finitequotientmap} [ \AA(P) / \GG(E) ] & \to & \AA(Q) . \ene  The results of this section reflect properties of this map of stacks for various ``realizations."  We record these properties as a series of remarks:

\begin{rem} \label{rem:algebraicCMS} Let us consider, for example, the case where $\AA=X$ is the algebraic realization (without log structure, over a field $k$ of characteristic zero, as usual).  In general, the monoid algebra functor \bne{monoidalg} k[ \slot ] : \Mon & \to & \Alg(k) \ene does not preserve equalizers.  However, it is clear from direct inspection that the diagram \bne{monoidalgeqdiagram} & k[Q] \to k[P] \rightrightarrows k[P \oplus E] \ene obtained by applying \eqref{monoidalg} to the equalizer diagram \eqref{eqdiagram} in Setup~\ref{setup:finitequotient} \emph{is} an equalizer diagram of $k$-algebras (equivalently: of sets)---all that matters here is that the parallel arrows in \eqref{eqdiagram} are both monic; the assumptions that $E$ is finite and that $k$ has characteristic zero are irrelevant.  It then follows from \cite[V.4.1]{SGA3} (see, in particular, the discussion of the affine case there) that $\Spec$ of the equalizer diagram \eqref{monoidalgeqdiagram} is a coequalizer diagram in schemes, locally ringed spaces, and ringed spaces (hence also on the level of topological spaces).  (Also in fppf sheaves on schemes, in fact.  Here we do need the finiteness of $E$ to ensure that the $k$-group scheme $\Spec k[E]$ is finite over $\Spec k$.)  In other words, $X(Q) = \Spec k[Q]$ is the scheme-theoretic (and the algebraic space theoretic, ringed space theoretic, etc.) quotient of $\Spec k[P]$ by the action of $\GG(E) = \Spec k[E]$, so \eqref{finitequotientmap} is the coarse moduli space of the global quotient stack $[X(P)/\GG(E)]$. \end{rem}

\begin{rem} \label{rem:fanCMS} It is also shown in \cite{G2}[\S2.8] that $\Spec$ of \eqref{eqdiagram} is an equalizer diagram in fans, locally monoidal spaces, and monoidal spaces (in particular, $\Spec P \to \Spec Q$ is a homeomorphism on topological spaces).  It follows that, when $\AA$ is the fan realization, the map \eqref{finitequotientmap} is the ``coarse moduli space" of $[ \AA(P) / \GG(E) ]$ in the sense that it is initial among maps from $[ \AA(P) / \GG(E) ]$ to a fan. \end{rem}

\begin{rem} \label{rem:diffrealization} When $\AA=Y$ is the \emph{differential} realization, the map \eqref{finitequotientmap} certainly does \emph{not} have any interpretation as a ``coarse moduli space".  Indeed, the differential realization $Y(E)$ of any finite abelian group $E$ is the trivial group object, so \eqref{finitequotientmap} is just the map of differentiable spaces $Y(P) \to Y(Q)$ induced by the inclusion $Q \into P$.  This is hardly ever an isomorphism.   For a simple example, take $P=\NN$, $a : \ZZ \to \ZZ / 2 \ZZ$ the projection.  Then $Q = 2 \NN$ and $Y(P) \to Y(Q)$ is the map $\RR_{\geq 0} \to \RR_{\geq 0}$ given by $x \mapsto x^2$, which is not a diffeomorphism.  On the other hand, in Setup~\ref{setup:finitequotient}, the realization $Y(P) \to Y(Q)$ of the inclusion $Q \into P$ is always a homeomorphism on the level of topological spaces.  This follows from the properties satisfied by $Q \into P$ mentioned in Setup~\ref{setup:finitequotient} using \cite[Lemma~5.8.1(5)]{GM1}. \end{rem}

\begin{prop} \label{prop:coarsemodulispace}  Let $F= (N,F,\{ F_\sigma \})$ be a KM fan, $\pi : F \to \ov{F}$ the map of KM fans defined in Example~\ref{example:coarsefan}.  Then, locally on $\AA(\ov{F})$, the map $\AA(\pi)$ is given by maps of the form \eqref{finitequotientmap} in Setup~\ref{setup:finitequotient}.  In particular: \begin{enumerate} \item If $\AA$ is the algebraic or fan realization, then $\AA(\pi)$ is the coarse moduli space of $\AA(F)$.  \item The algebraic realization $X(\pi)$ of $\pi$ is a proper map of Deligne-Mumford stacks.  \item The differential realization $Y(\pi)$ of $\pi$ is a homeomorphism on the level of topological spaces. \end{enumerate} \end{prop}

\begin{proof}  Note that $\pi_{\RR}$ is a isomorphism from the set of cones in $F$ (ordered by inclusion) to the set of cones in $\ov{F}$ (ordered by inclusion).  Let $\sigma$ be a cone of $F$, $\ov{\sigma}$ the corresponding cone of $\ov{F}$.  The aforementioned isomorphy ensures that the open substack $\AA(\sigma,F) \subseteq \AA(F)$ is the preimage of the open subspace $\AA(\ov{\sigma},\ov{F}) \subseteq \AA(\ov{F})$ under $\AA(\pi)$ (cf.\ Remark~\ref{rem:preimages}).  The statements we want to prove are local on $\AA(\ov{F})$, so we reduce to proving them for \bne{CMS} \AA(\pi) | \AA(\sigma,F) : \AA(\sigma,F) & \to & \AA(\ov{\sigma},\ov{F}). \ene  Since $\ov{F}$ ``is" a classical fan, we may take $\ov{L} = \ov{N}$ as a lifting of the lattice datum $\ov{F}_{\ov{\sigma}} = \ov{N} \cap \Span \ov{\sigma}$ of $\ov{\sigma}$.  Pick any lifting $L \subseteq N$ of the lattice datum $F_\sigma$ (Lemma~\ref{lem:liftingsexist}).  Obviously $\pi(L) \subseteq \ov{L} = \ov{N}$, so we can calculate \eqref{CMS} as \bne{CMS2} \AA(\sigma,F,L) & \to & \AA(\ov{\sigma},\ov{F},\ov{L}=\ov{N}). \ene 

To do this, we run through the construction of \eqref{CMS2} in \S\ref{section:KMfanrealization}.  In our situation, the diagram \eqref{SESmapf} in that section takes the form\bne{SESmappi}  & \xym{ 0 \ar[r] & L \ar[d]_\pi \ar[r] & N \ar[d]_\pi \ar[r] & N/L \ar[d] \ar[r] & 0 \\ 0 \ar[r] & \ov{L} \ar@{=}[r] & \ov{N} \ar[r] & 0 \ar[r] & 0 } \ene and its ``dual" \eqref{LESmapf} takes the form \bne{LESmappi} & \xym{ 0 \ar[r] & N^\lor \ar[r] & L^\lor \ar[r] & \EE(N/L) \ar[r] & \EE(N) \ar[r] & 0 \\ 0 \ar[r] & \ov{N}^\lor \ar[u]_{\pi^\lor}^{\cong} \ar@{=}[r] & \ov{L}^\lor \ar[u]_{\pi^\lor}  \ar[r] & 0 \ar[r] \ar[u] & 0 \ar[r] \ar[u] & 0. } \ene  The left vertical arrow is an isomorphism because $\pi : N \to \ov{N}$ is just the quotient of $N$ by its torsion subgroup.  This implies that \bne{SESpi} & 0 \to \ov{L}^\lor \to L^\lor \to \EE(N/L) \ene is exact.  Since $\pi_{\RR}$ takes $\sigma$ bijectively onto $\ov{\sigma}$ we see from exactness of \eqref{SESpi} that $S_{\ov{\sigma}}(\ov{L})$ is nothing but the submonoid of $S_{\sigma}(L)$ consisting of those elements mapping to zero under $S_{\sigma}(L)^{\rm gp} = L^\lor \to \EE(N/L)$.  On the other hand, the map \eqref{CMS2} is given by \bne{CMS3} [ \AA(S_{\sigma}(L)) / \GG(\EE(N/L)) ] & \to & \AA(S_{\ov{\sigma}}(\ov{L})). \ene  This proves the first assertion.

The statements about the algebraic and fan realizations follow from the remarks preceding the theorem and general properties of coarse moduli spaces.  In the case of the differential realization, \eqref{CMS3} is just the map \bne{CMS4} Y(S_{\sigma}(L)) & \to & Y(S_{\ov{\sigma}}(\ov{L}))  \ene because $Y(\EE(N/L))$ is the trivial group object since $\EE(N/L)$ is finite.  The map \eqref{CMS4} is a homeomorphism on topological spaces by Remark~\ref{rem:diffrealization}.  This proves the final assertion.
\end{proof}

\begin{cor} \label{cor:separated} For any KM fan $F$, the algebraic realization $X(F)$ of $F$ is separated and the topological space underlying the differential realization $Y(F)$ is Hausdorff. \end{cor}

\begin{proof}  Recall from Example~\ref{example:classicalrealization} that $X(\ov{F})$ is the ``usual" realization of the classical fan $\ov{F}$, so it is a toric variety, hence, in particular, separated, so $X(F)$ is separated since the coarse moduli space map $X(\pi)$ is proper, hence separated.  Similarly, Proposition~\ref{prop:coarsemodulispace} says that $Y(\pi)$ is a homeomorphism, so for the second statement we can reduce to the case where $F$ is a classical fan, in which case $Y(F)$ is Hausdorff because it is a closed subspace of the analytic topological space $X(F)^{\rm an}$ underlying $X(F)$, which is Hausdorff because $X(F)$ is separated.  \end{proof}

\begin{example} \label{example:P22coarsening} Let $F$ be the KM fan described in Example~\ref{example:P22} and let $\pi : F \to \ov{F}$ be the coarse fan map described in Example~\ref{example:coarsefan}, so that $X(F) \to X(\ov{F})$ is the coarse moduli space of $X(F)$ by Proposition~\ref{prop:coarsemodulispace}.  The classical fan $\ov{F}$ is the fan for $\PP^1$, so $X(\ov{F}) = \PP^1$.  For either of the maximal cones $\sigma_{\pm}$ of $F$, the lattice datum $F_{\pm}$ is of finite index in $N = \ZZ \oplus \ZZ / 2 \ZZ$---indeed, the quotient $N / F_{\pm}$ is isomorphic to $\ZZ / 2 \ZZ$ and the projection $N \to N / F_{\pm}$ induces an isomorphism $\EE(N/F_{\sigma}) \to \EE(N)$ (both are isomorphic to $\ZZ / 2 \ZZ$).  Furthermore, $F_{\pm}$ maps isomorphically via $\pi$ onto the corresponding lattice datum $\ZZ$ for $\sigma_{\pm}$ in the classical fan $\ov{F}$.  By construction of the realization, we see that $\AA(\sigma_{\pm},F) \to \AA(\sigma_{\pm},\ov{F})$ is isomorphic to the obvious projection \be \AA^1 \times B \GG( \ZZ / 2 \ZZ) & \to & \AA^1 . \ee  In particular, since the group object $\GG( \ZZ / 2 \ZZ)$ is trivial in the differential realization, we see that the differential realization $Y(\pi)$ of $\pi$ is an isomorphism of log differentiable spaces.  \end{example}

\subsection{Rigidification and root stacks} \label{section:rigidificationandrootstacks}  We will now give geometric interpretations of some of the examples from \S\ref{section:examples}.  The details here will be left to the interested reader.

\begin{example} \label{example:rigidificationrealization}  Let $F = (N,F,\{ F_\sigma \})$ be a KM fan.  For any lattice $L$ of finite index in $N$, the natural map $N_{\rm tor} \to N/L$ is an \emph{injection} of finite abelian groups because $L \cap N_{\rm tor} = 0$ since $L$ is torsion-free, hence we have a surjection $\EE(N/L) \to \EE(N_{\rm tor})$ of finite abelian groups, and hence a closed embedding of finite abelian group schemes from $G := \GG(\EE(N_{\rm tor}))$ into $\GG(\EE(N/L))$.  From the construction of the algebraic realization $X(F)$, we see that $G$ is a (normal) subgroup of the inertia group scheme of $X(F)$ because, locally, $X(F)$ is a global quotient by an abelian group scheme containing $G$ as a subgroup acting trivially.  Let $F^{\rm rig} := (\ov{N},\ov{F}, \{ q(F_\sigma) \} )$ be the rigidification of $F$ as in Example~\ref{example:rigidification}.  The algebraic realization of the map of KM fans $F \to F^{\rm rig}$ can be interpreted as the rigidification of $X(F)$ with respect to $G$. \end{example}

\begin{example} \label{example:rootstacksrealization}  Let $F$, $\rho_1,\dots,\rho_n \in F$, $a=(a_1,\dots,a_n) \in \ZZ_{>0}^n$, and $aF$ be as in Example~\ref{example:rootstacks}.  The algebraic realization $X(aF) \to X(F)$ of the map of (smooth) KM fans $aF \to F$ is the fibered product, over $X(F)$, of the roots stacks of $X(F)$ along the (Cartier) divisors $D_i \subseteq X(F)$ corresponding to the rays $\rho_i$: \be X(aF) & = &  \sqrt[a_1]{(X(F),D_1)} \times_{X(F)} \cdots \times_{X(F)} \sqrt[a_n]{(X(F),D_n)}. \ee  We are not aware of any simple geometric interpretation of the geometric realization of the dilation construction of Example~\ref{example:dilation}. \end{example}

\subsection{Star fans and stratification} \label{section:starfansandstratification}  In this section and the next, we explain the analogue of the ``star fan" construction familiar from the classical theory of toric varieties.  This is used to describe (certain) ``torus invariant subvarieties of a toric variety."

Let $F = (N,F,\{ F_\sigma \})$ be a KM fan.  Fix a cone $\tau \in F$ for the remainder of the section.  We define a KM fan $\Star(\tau,F)$ (or just $\Star(\tau)$ if there is no chance of confusion) called the \emph{star of} $\tau$ \emph{in} $F$ as follows:  The ``lattice" of $\Star(\tau)$ is $N/F_\tau$, the cones of $\Star(\tau)$ are the subsets $\ov{\sigma} := \sigma / \Span \tau$ of $(N/F_\tau)_\RR = N_{\RR} / \Span \tau$ given by the images of all cones $\sigma$ of $F$ containing $\tau$; the lattice datum for such a cone $\sigma / \Span \tau$ is the image $F_\sigma / F_\tau$---note that this is torsion free because the compatibility condition $F_\tau = F_\sigma \cap \Span \tau$ ensures that $F_\tau$ is saturated in $F_\sigma$.

Notice that the ``lattice" $N/F_\tau$ for $\Star(\tau)$ may not be torsion-free even when $N$ is torsion-free because the inclusion $F_\tau \subseteq N$ may not be saturated (equivalently, the finite index inclusion $F_\tau \into N_\tau$ may not be an equality).  This is one of many points in the theory of KM fans where one must allow $N$ to have torsion in order to obtain a satisfactory theory.

Suppose $\sigma$ is a cone of $F$ containing our fixed cone $\tau$ and $L$ is a lifting of $F_\sigma$.  The compatibility condition on the lattice data $F_\tau$ and $F_\sigma$ and the fact that $L$ is a lifting of $F_\sigma$ ensure that \be L \cap \Span \tau & = & F_\tau. \ee  In particular $L/F_\tau$ is a lattice contained in $N/F_\tau$, which is readily seen to be a lifting of the lattice datum $F_\sigma / F_\tau$ for the cone $\ov{\sigma}$ in the KM fan $\Star(\tau)$.  Furthermore, we have a map of exact sequences \bne{JK} & \xym{ 0 \ar[r] & L \ar[d] \ar[r] & N \ar[d] \ar[r] & N/L \ar@{=}[d] \ar[r] & 0 \\ 0 \ar[r] & L/F_\tau \ar[r] & N/F_\tau \ar[r] & N/L \ar[r] & 0.} \ene  

It is immediate from the definitions that the inclusion $(L/F_\tau)^\lor \into L^\lor$ takes the submonoid $S_{\ov{\sigma}}(L/F_\tau)$ of $(L/F_\tau)^\lor$ isomorphically onto the face $\tau^\perp \cap S_{\sigma}(L)$ of $S_\sigma(L)$ corresponding to $\tau \leq \sigma$.  We thus view $S_{\ov{\sigma}}(L/F_\tau)$ as a face of $S_\sigma(L)$.  Like the inclusion $F \into P$ of any face $F$ of any monoid $P$, the aforementioned face inclusion has a retract \bne{faceretract} S_\sigma(L)_* & \to & S_{\ov{\sigma}}(L/F_\tau)_* \ene once we pass to the associated \emph{pointed monoids} by adjoining an element $\infty$ with the obvious addition law $p+\infty = \infty$ for all $p$.  (The retract $P_* \to F_*$ of $F_* \to P_*$ takes $p$ to $p$ when $p \in F$ and to $\infty$ when $p \notin F$.)

The map of pointed monoids \eqref{faceretract} can be ``realized" in various ``categories of spaces" such as the category of ``pointed fans" (also called ``monoid schemes" by Weibel and others), the category of schemes, and the category of differentiable spaces, etc.\ (We always assume that the pointed realization $\AA_*(P_*)$ of the ``pointing" of a monoid $P$ coincides with the given realization $\AA(P)$.)  For example, the (pointed) realization of \eqref{faceretract} in schemes (over our base field $k$, let's say) is given by $\Spec$ of the map of $k$-algebras \bne{kfaceretract} k[S_\sigma(L)] & \to & k[S_{\ov{\sigma}}(L/F_\tau)] \ene defined by taking $[p] \in k[S_\sigma(L)]$ to $[p] \in k[S_{\ov{\sigma}}(L/F_\tau)]$ whenever $p \in S_{\ov{\sigma}}(L/F_\tau) \leq S_\sigma(L)$ and by taking $[p]$ to $0$ whenever $p \in S_\sigma(L) \setminus S_{\ov{\sigma}}(L/F_\tau)$.  Note that the map of $k$-algebras \eqref{kfaceretract} does in fact retract the map of $k$-algebras associated to the face inclusion $S_{\ov{\sigma}}(L/F_\tau) \into S_\sigma(L)$.  (This is because the latter map of $k$-algebras is the pointed realization of the pointing $S_{\ov{\sigma}}(L/F_\tau)_* \into S_\sigma(L)_*$ of the aforementioned face inclusion.)

\begin{rem} \label{rem:pointedrealization} There is no reasonable way to realize the face retract \eqref{faceretract} in the category of \emph{log} schemes (or \emph{log} differentiable spaces).  One could make sense of this realization in an appropriate category of \emph{pointed} log schemes (or pointed log differentiable spaces), but we shall not pursue this here. \end{rem}

Let us record this sort of ``pointed realization" of \eqref{faceretract} as a map of ``spaces" \bne{Afaceretract} \AA(S_{\ov{\sigma}}(L/F_\tau)) & \to & \AA( S_\sigma(L) ), \ene keeping in mind that it can only be defined for \emph{certain} realizations $\AA$.  Let us now concentrate for a moment on the case where $\AA$ is either the algebraic or the differential realization (\emph{without log structures}).  In this case, \eqref{Afaceretract} is a closed embedding since it is the (pointed) realization of the surjection of (pointed) monoids \eqref{faceretract}.  Denote by $\tilde{\sigma} = (L, [ \sigma ], \{ F_\tau \cap L \})$ the classical fan whose lattice is $L$, whose set of cones $[ \sigma ]$ is the set of faces of $\sigma$, and whose realization is $\AA(S_\sigma(L)) = \AA(\tilde{\sigma})$.  Then the classical fan whose lattice is $L/F_\tau$ and whose realization is $\AA(S_{\ov{\sigma}}(L/F_\tau))$ is nothing but $\Star(\tau,\tilde{\sigma})$ and \eqref{Afaceretract} is nothing but the closed embedding \bne{Afaceretract2} \AA(\Star(\tau,\tilde{\sigma})) & \to & \AA(\tilde{\sigma}) \ene well-known in the ``classical" theory of (affine) toric varieties \cite[\S3.2]{CLS}.  In particular, the ``torus" (that is, the ``atomic open") $$ \AA(0,\Star(\tau,\tilde{\sigma})) = \AA(L/F_\tau,0,0) = \GG((L/F_\tau)^\lor) $$ (cf.\ Example~\ref{example:atomicopen}) is a locally closed subspace of $\AA(\tilde{\sigma})$.  It is well-known from the classical theory that if we fix $\sigma$ and let $\tau$ run over all subcones of $\sigma$, then the locally closed subspaces $\AA(L/F_\tau,0,0)$ form a stratification \bne{classicalstratification} \AA(\tilde{\sigma}) & = & \coprod_{\tau \leq \sigma} \AA(L/F_\tau,0,0) \ene of $\AA(\tilde{\sigma})$, corresponding to the fact that the submonoids $S_{\ov{\sigma}}(L/F_\tau)$ are precisely the faces of $S_\sigma(L)$.

In any situation where we can form the pointed realization, the commutativity of \eqref{JK} ensures that \eqref{Afaceretract} (like the realization of the corresponding face \emph{inclusion} it is retracting) is $\GG(\EE(N/L))$ equivariant, so it induces a map \bne{starmaplocal} [ \AA(S_{\ov{\sigma}}(L/F_\tau)) / \GG(\EE(N/L)) ] & \to & [ \AA( S_\sigma(L) ) / \GG(\EE(N/L)) ]. \ene  Using \eqref{JK}, we can interpret \eqref{starmaplocal} as a map \bne{starmaplocal2} \AA( \ov{\sigma}, \Star(\tau)) & \to & \AA(\sigma,F). \ene  In the algebraic or differential setting, the map \eqref{starmaplocal2} is a closed embedding since it is the stack quotient of the $\GG(\EE(N/L))$ equivariant closed embedding \eqref{Afaceretract}.  Similarly, if we fix $\sigma$ and let $\tau$ run over the subcones of $\sigma$, then the atomic opens $$\AA(0,\Star(\tau,[\sigma])) = \AA(N/F_\tau,0,0)$$ (here we write $[\sigma]$ for the KM fan whose cones are the subcones of $\sigma$ and whose lattice data are inherited from $F$) form a stratification \bne{localstratification} \AA(\sigma,F) & = & \coprod_{\tau \leq \sigma} \AA(N/F_\tau,0,0) \ene of $\AA(\sigma,F)$ which is nothing but the stack-theoretic quotient of \eqref{classicalstratification} by $\GG(\EE(N/L))$.

By running over all cones $\sigma$ of $F$ containing $\tau$ and noting that the maps \eqref{starmaplocal2} are compatible with inclusions (of cones of $F$ containing $\tau$), we obtain a map of stacks \bne{starmap} \AA(\Star(\tau,F)) & \to & \AA(F). \ene  Since the question is local on $\AA(F)$, we see that (in the differential or algebraic situation), \eqref{starmap} is a closed embedding satisfying \bne{starmapintersection} \AA(\Star(\tau,F)) \cap \AA(\sigma,F) & = & \AA(\Star(\tau,[\sigma])). \ene  Combining the discussion thus far with the formula \eqref{starmapintersection} and the formula \eqref{intersectionformula} for the intersections of the $\AA(\sigma,F)$, we obtain:

\begin{prop} \label{prop:stratification} Let $F$ be a KM fan and let $\AA$ denote either the algebraic or differential realization (\emph{without log structures}).  The atomic opens in the various $\AA(\Star \tau)$ determine a stratification \bne{KMfanstratification} \AA(F) & = & \coprod_{\tau \in F} \AA(N/F_\tau,0,0) \ene of $\AA(F)$ by locally closed substacks.  The stratum $\AA(N/F_\tau,0,0) = \AA(0,\Star(\tau))$ is a (non-canonically) trivial gerbe over the ``torus" $\GG((N/F_\tau)^\lor)$ banded by $\GG(\EE(N/F_\tau))$ (Example~\ref{example:atomicopen}).  The closed substack $\AA(\Star(\tau))$ of $\AA(F)$ is the smallest closed substack of $\AA(F)$ containing the stratum $\AA(0,\Star \tau)$ corresponding to $\tau$. \end{prop}

\begin{rem} \label{rem:stratification1} As mentioned in Remark~\ref{rem:pointedrealization}, one cannot make sense of the stratification of Proposition~\ref{prop:stratification} for the \emph{log} algebraic or \emph{log} differential realizations.  Although it is possible, in some weak sense, to make sense of these ``stratifications" using \emph{pointed} log structures, one should be aware that the log structure on $\AA(N/F_\tau,0,0)$ coming from the fact that it is the log realization of a KM fan is trivial, whereas the restriction of the log structure on $\AA(F)$ to such a stratum is generally non-trivial.  This boils down to the fact that the ``pointed log realization" of a face retract \eqref{faceretract} is generally \emph{not} a ``strict closed embedding" even though it is a closed embedding on underlying spaces. \end{rem}

\begin{rem} \label{rem:stratification2} As in the classical case of toric varieties, one can see (by reducing to that classical case, for example) that the stratification of Proposition~\ref{prop:stratification} is functorial with respect to maps of KM fans $f : F \to F'$.  The (algebraic or differentiable) realization $\AA(f)$ of $f$ takes the stratum $\AA(N/F_\sigma,0,0)$ of $\AA(F)$ corresponding to a cone $\sigma \in F$ into the stratum $\AA(N'/F'_\tau,0,0)$ of $\AA(F')$ (via the map $\AA(N/F_\sigma,0,0) \to \AA(N'/F'_\tau,0,0)$ induced by the map $N/F_\sigma \to N'/F'_\tau$ induced by $f$) corresponding to the smallest cone $\tau \in F'$ containing $f_{\RR}(\sigma)$. \end{rem}

\subsection{Support and properness} \label{section:supportandproperness}  Here we define the analog of the ``support" of a classical fan and explain how it is related to properness of the realization of a map of KM fans.

\begin{defn} \label{defn:support} Let $F = (N,F,\{ F_\sigma \})$ be a KM fan.  The \emph{fine support} (or \emph{stacky support}, or just \emph{support} if there is no chance of confusion) of $F$ is the subset $\Supp F \subseteq N$ defined by $$ \Supp F  :=  \bigcup_{\sigma \in F} F_\sigma \cap \sigma = \bigcup_{\sigma \in F} P_{\sigma} $$ (cf.\ Definition~\ref{defn:KMfan2}).  The \emph{coarse support} of $F$, denoted $\Supp \ov{F}$, is, as the notation suggests, the subset of $\ov{N} = N/N_{\rm tor}$ given by the support of the classical fan $\ov{F}$ underlying $F$ (Example~\ref{example:classicalfan})---i.e.\ \be \Supp \ov{F} & = & \bigcup_{\sigma \in F} \ov{N}_{\sigma} . \ee   \end{defn}

Recall the KM fans $\GG_m$ and $\AA^1$ from Example~\ref{example:GmA1}.  If $F$ is a KM fan, any $n \in N$ determines a map of KM fans $n : \GG_m \to F$ defined by mapping $1 \in \ZZ$ to $n \in N$.

\begin{prop} \label{prop:support} For a KM fan $F = (N,F,\{ F_\sigma \})$ and an $n \in N$, the following are equivalent: \begin{enumerate} \item \label{support1} $n \in \Supp F$ \item \label{support2} The map of KM fans $n : \GG_m \to F$ extends to a map of KM fans $\tilde{n} : \AA^1 \to F$. \item \label{support3} The map of log stacks $X(n) : \GG_m=X(\GG_m) \to X(F)$ obtained by algebraically realizing the map of KM fans $n$ extends to a map of log stacks $\AA^1 \to X(F)$. \item \label{support4} The map of stacks $X(n) : \GG_m \to X(F)$ underlying the map of log stacks in \eqref{support3} extends to a map of stacks $\AA^1 \to X(F)$. \item \label{support5} There is an algebraically closed field $K$ containing the base field $k$ such that the map of stacks in \eqref{support4} (base changed to $K$) extends to a map of stacks $\tilde{n}_{K} : \AA^1_K \to X(F)_K$. \end{enumerate}  These equivalent conditions imply the condition \begin{enumerate} \setcounter{enumi}{5} \item \label{support6} The map of topological spaces $n : \RR_{>0} \to Y(F)$ underlying the differential realization of the map of KM fans in \eqref{support2} extends to a continuous map $\tilde{n} : \RR_{\geq 0} \to Y(F)$. \end{enumerate} and are equivalent to this condition when the fan $F$ is a classical fan. \end{prop}

\begin{proof} The equivalence of \eqref{support1} and \eqref{support2} is immediate from the definitions of the KM (in fact, classical) fans $\GG_m$ and $\AA^1$ (Example~\ref{example:GmA1}), the definition of a map of KM fans (Definition~\ref{defn:KMfan}), and the definition of the support (Definition~\ref{defn:support}).  Obviously \eqref{support2} implies \eqref{support3} by applying the algebraic realization functor $X$.  Obviously \eqref{support3} implies \eqref{support4} by forgetting log structures.  Obviously \eqref{support4} implies \eqref{support5} by extending scalars.

To prove that \eqref{support5} implies \eqref{support1}, first note that \eqref{support5} is the same as \eqref{support4}, except we work over $K$ instead of $k$, so it is enough to prove that \eqref{support4} implies \eqref{support1} in the case where the base field $k$ is algebraically closed.  To do this, we first treat the case where $N$ is a lattice (i.e.\ $F$ is a lattice KM fan).  Here \eqref{support4} implies that $n \in N \cap \sigma$ for some cone $\sigma \in F$ because if $X(n)$ has such an extension, then so does the corresponding map to the coarse moduli space of $X(F)$, which is $X(\ov{F})$ by Proposition~\ref{prop:coarsemodulispace}, and we know from the classical theory of toric varieties that the proposition holds for the classical fan $\ov{F}$.  It still remains to show that $n \in F_{\sigma}$.  Pick a lifting $L \subseteq N$ of $F_\sigma$, so $L \cap \Span \sigma = F_\sigma$ and it suffices to show that $n \in L$ (for then $n$ will be in $P_{\sigma} = L \cap \sigma$).  We know the lift $\tilde{n} : \AA^1 \to X(F)$ factors through the open substack $X(\sigma,F) \subseteq X(F)$ because this is true on the level of coarse moduli spaces and $X(\sigma,F)$ is the preimage of $X(\sigma,\ov{F}) \subseteq X(\ov{F})$ under $X(F \to \ov{F})$.  We can therefore assume that $X(\sigma,F) = X(F)$.  By construction (\S\ref{section:KMfanrealization}), the realization $X(\sigma,F) = X(F)$ of such a KM fan $F$ is the quotient $[W / S]$, where $W$ is the affine toric variety corresponding to the cone $\sigma \subseteq L_{\RR} = N_{\RR}$ with respect to the lattice $L$ and $S = \GG(\EE(N/L))$ acts on $W$ through the action of the torus $T' := \GG(L^\lor)$ of $W$ and the embedding $S \into T'$ coming from applying $\GG( \slot )$ to the exact sequence on the top row of the diagram below. \bne{NEW} & \xym{ 0 \ar[r] & N^\lor \ar[d]_n \ar[r] & \ar@{.>}[ld] L^\lor  \ar[d]^f \ar[r] & \EE(N/L) \ar@{=}[d] \ar[r] & 0 \\ 0 \ar[r] & \ZZ \ar[r] & A \ar[r] & \EE(N/L) \ar[r] & 0. } \ene (Note $\EE(N)=0$ since $N$ is a lattice.)  The bottom row of \eqref{NEW} is defined to be the pushout of the top row along $n : N^\lor \to \ZZ$.  View an object of $X(F) = [W/S]$ over a scheme $U$ as an $S$ torsor $P \to U$ together with an $S$ equivariant map $P \to W$.  Then the map $n : \GG_m \to X(F)$ corresponds to the $S$ torsor $\GG(A) \to \GG(\ZZ) = \GG_m$ coming from the bottom row of \eqref{NEW} and the $S$-equivariant map $\GG(A) \to W = \AA( L^\lor \cap \sigma^\lor)$ coming from the restriction of the middle vertical arrow in \eqref{NEW}.  By \eqref{support4}, the map $n : \GG_m \to X(F)$ extends to a map $\tilde{n} : \AA^1 \to X(F)=[W/S]$.  In particular, this means that the $S$ torsor $\GG(A) \to \GG_m$ must extend to an $S$ torsor $P \to \AA^1$.  But every $S$-torsor over $\AA^1$ is trivial ($\H^1_{et}(\AA^1,S)=0$), so this means $\GG(A) \to \GG_m$ must be trivial.  (Here we use the fact that the base field $k$ is algebraically closed and of characteristic zero, otherwise there may be nontrivial $S$ torsors over $\AA^1$ coming from nontrivial $S$ torsors over $k$ itself.)  But this means the extension given by the bottom row of \eqref{NEW} must split because $A \mapsto \GG(A)$ defines an isomorphism of groups \be \Ext^1( \EE(N/L) , \ZZ) & \to & \H^1_{et}(\GG_m, S). \ee  (Again we use that $k$ is algebraically closed and of characteristic zero.)  Such a splitting is the same thing as a completion as indicated by the dotted arrow, which is the same thing as saying that $n \in L$, as desired. 

The general case is quite similar to the ``lattice case" treated above, but we have chosen to treat it separately for clarity.  As in the lattice case, we can see that $n \in N \cap \sigma$ for some cone $\sigma \in F$ and we can assume $\sigma$ is the unique maximal cone of $F$.  Let $L$ be a lifting of $F_\sigma$, as before.  We know that $an \in F_{\sigma}$ for some $a \in \ZZ_{>0}$ since $F_\sigma \subseteq N_\sigma$ has finite index (by definition of a KM fan).  This means that the preimage of $L \subseteq N$ under $n : \ZZ \to N$ will be of the form $a \ZZ$ for some $a \in \ZZ_{>0}$.  We want to show that $a=1$, so that $n \in L$.  It is equivalent to show that the finite group $Q := \ZZ / a \ZZ$ is zero, and this in turn is equivalent to showing that $\EE(Q)=0$.  We have a map of short exact sequences \bne{NEW2} & \xym{ 0 \ar[r] & a \ZZ \ar[r] \ar[d] & \ZZ \ar[d]^-n \ar[r] & Q \ar[d] \ar[r] & 0 \\ 0 \ar[r] & L \ar[r] & N \ar[r] & N/L \ar[r] & 0 } \ene where $Q \to N/L$ is injective by the definition of $a \ZZ \subseteq \ZZ$.  From \eqref{NEW2}, we obtain a diagram with exact rows \bne{NEW2dual} & \xym{ 0 \ar[r] & N^\lor \ar[r] \ar[d] & L^\lor \ar[r] \ar[d] & \EE(N/L) \ar[r] \ar@{=}[d] & \EE(N) \ar[r] & 0  \\ 0 \ar[r] & \ZZ \ar[r] \ar[d] & A \ar[r] \ar[d] & \EE(N/L) \ar[r] \ar[d] & 0 \\ 0 \ar[r] & \ZZ \ar[r] & (a \ZZ)^\lor \ar[r] & \EE(Q) \ar[r] & 0.} \ene  The map from the top to the bottom row in \eqref{NEW2dual} comes from dualizing \eqref{NEW2}; the middle row in \eqref{NEW2dual} comes from pulling back the bottom row along $\EE(N/L) \to \EE(Q)$.  As in the lattice case, we have $X(F) = [W/S]$ (same definitions of $S$ and $W$ as above, except now $S \to T'$ may not be an embedding so $S$ may not act effectively on $W$) and the map $n : \GG_m \to X(F) = [W/S]$ corresponds to the $S$ torsor $\GG(A) \to \GG_m$ coming from the middle row of \eqref{NEW2dual} and the $S$ equivariant map $\GG(A) \to W = \AA(L^\lor \cap \sigma^\lor)$ coming from the upper middle vertical arrow in \eqref{NEW2dual}.  For the same reason as in the lattice case, \eqref{support4} ensures that the middle row of \eqref{NEW2dual} splits.  If we denote by $x \in \Ext^1(\EE(Q),\ZZ)$ the extension class of the bottom row of \eqref{NEW2dual}, then this means that $x$ is in the kernel of \bne{emap} \Ext^1( \EE(Q), \ZZ) & \to & \Ext^1( \EE(N/L), \ZZ) . \ene  By basic algebra (cf.\ Lemma~\ref{lem:isotropygroups} in \S\ref{section:isotropygroups}), the map \eqref{emap} is identified with the map $Q \to N/L$, which is injective, so $x=0$, so the bottom row of \eqref{NEW2dual} splits.  But the middle group in that bottom row is isomorphic to $\ZZ$, so the only way that sequence can split is if $\EE(Q)=0$.  

Obviously \eqref{support2} implies \eqref{support6} by taking the differential realization of the map of KM fans $\tilde{n} : \AA^1 \to F$ in \eqref{support2} and looking at the underlying map of topological spaces.  To see that \eqref{support6} implies \eqref{support1} when $F$ is classical, suppose, toward a contradiction, that \eqref{support6} holds but \eqref{support1} does not.  Then obviously $n \neq 0$.  Let $G$ be the (classical) fan obtained from $F$ by adding the ray $\rho$ through $n$ to the set of cones $F$.  (Notice that it would be unclear what this would mean if we were working with KM fans because if $F$ is a KM fan it might be that $\rho$ is already a cone of $F$ even though $n \notin \Supp F$.)  The composition of the continuous map $\tilde{n} : \RR_{\geq 0} \to Y(F)$ assumed to exist in \eqref{support6} and the inclusion $Y(F) \subseteq Y(G)$ (the realization of $F \to G$) must agree with the map of topological spaces $\RR_{\geq 0} \to Y(G)$ underlying the differential realization of the map of KM fans $``n" : \AA^1 \to G$ because $\RR_{>0}$ is dense in $\RR_{\geq 0}$ and $Y(G)$ is Hausdorff (because it is a closed subspace of the analytic topological space underlying $X(G)$).  This is a contradiction because the latter map takes $0 \in \RR_{\geq 0}$ to a point of $Y(G)$ not in $Y(F)$ in light of the stratifications of $Y(F)$, $Y(\AA^1) = \RR_{\geq 0}$, and $Y(G)$ constructed in Proposition~\ref{prop:stratification} and their naturality (Remark~\ref{rem:stratification2}).  Indeed, the strata of $Y(F)$ are the same as those of $Y(G)$, except $Y(G)$ has one extra stratum corresponding to the cone $\rho$, and the latter map takes $0$ into that stratum because $\rho$ is the smallest cone of $G$ (indeed the only cone of $G$) containing the image of the maximal cone of the classical fan $\AA^1$ under the map of classical fans $``n" : \AA^1 \to G$.    \end{proof}

\begin{rem} \label{rem:support} If $F$ is a (non-classical) KM fan, condition \eqref{support6} in Proposition~\ref{prop:support} does not generally imply the other conditions in that proposition.  Indeed, even if $n : \RR_{>0} \to Y(F)$ can be extended to a map of log differentiable spaces $\tilde{n} : \RR_{\geq 0} \to Y(F)$, this does not necessarily mean $n \in \Supp F$.  This is because it is possible to find a map of KM fans $f :F \to F'$ for which $Y(f)$ is an isomorphism of log differentiable spaces but where there is some $n \in N \setminus \Supp F$ with $f(n) \in \Supp F'$.  This occurs in Example~\ref{example:P22coarsening}.  In fact, we will see in Example~\ref{example:inflationrevisited} that this is in fact a rather common phenomenon. \end{rem}

\begin{prop} \label{prop:properness} For a map $f : F \to F'$ of KM fans, the following are equivalent: \begin{enumerate} \item \label{proper1} The algebraic realization $X(f)$ is a proper map of DM stacks. \item \label{proper1prime} The algebraic realization of the corresponding map $\ov{f} : \ov{F} \to \ov{F}'$ of coarse fans (cf.\ Example~\ref{example:coarsefan}) is a proper map of toric varieties. \item \label{proper2} The differential realization $Y(\ov{f})$ of the corresponding map $\ov{f} : \ov{F} \to \ov{F}'$ of coarse fans is a proper map of topological spaces. \item \label{proper2prime} The differential realization $Y(f)$ is a proper map of topological spaces. \item \label{proper3} The preimage of the coarse support of $F'$ under $f$ is equal to the coarse support of $F$. \end{enumerate} \end{prop}

\begin{proof} To see that \eqref{proper1} and \eqref{proper1prime} are equivalent, first note that we have a commutative diagram \bne{Xfdiagram} & \xym{ X(F) \ar[d]_{X(f)} \ar[r] & X(\ov{F}) \ar[d]^{X(\ov{f})} \\ X(F') \ar[r] & X(\ov{F}') } \ene where the horizontal arrows are coarse moduli spaces, hence proper (Proposition~\ref{prop:coarsemodulispace}).  The equivalence then follows from the valuative criterion for properness for maps of such stacks because it is enough to check the valuative criterion in the case where the residue field $K$ of the DVR is algebraically closed, in which case the coarse moduli space maps are ``bijective" on $K$ points (Theorem~\ref{thm:cms}).  Conditions \eqref{proper2} and \eqref{proper2prime} are equivalent by the same reasoning, because in that case the horizontal arrows in the diagram \eqref{Xfdiagram} (with ``$X$" replaced by ``$Y$") are even homeomorphisms (Proposition~\ref{prop:coarsemodulispace}) so the maps $Y(f)$ and $Y(\ov{f})$ are the same at the level of topological spaces.  Condition \eqref{proper1prime} implies, by general GAGA results, that the corresponding map of analytic topological spaces $X^{\rm an}(\ov{F}) \to X^{\rm an}(\ov{F}')$ is proper and this in turn implies \eqref{proper2} because we have a commutative diagram of topological spaces \bne{propdia2} & \xym{ Y(\ov{F}) \ar[r] \ar[d]_{Y(\ov{f})} & X^{\rm an}(\ov{F}) \ar[d]^{X^{\rm an}(\ov{f})} \\ Y(\ov{F}') \ar[r] & X^{\rm an}(\ov{F}') } \ene where the horizontal arrows are closed embeddings.  The equivalence of \eqref{proper1prime} and \eqref{proper3} is a standard fact from the classical theory of toric varieties \cite[\S2.4]{F}.  To see that \eqref{proper2} implies \eqref{proper3} we use the last part of Proposition~\ref{prop:support} to see that if \eqref{proper3} fails, then there is a commutative diagram of topological spaces $$ \xym{ \RR_{>0} \ar[r] \ar[d] & Y(\ov{F}) \ar[d]^{Y(\ov{f})} \\ \RR_{\geq 0} \ar@{.>}[ru] \ar[r] & Y(\ov{F}') } $$ for which there is no lift as indicated---this contradicts \eqref{proper2}.   \end{proof}

\subsection{Products} \label{section:products}  

Let $F = (N,F,\{ F_\sigma \})$, $F' = (N',F',\{ F'_\tau \})$ be KM fans.  We define the \emph{product} KM fan $F \times F'$ as follows:  The ``lattice" of $F \times F'$ is $N \times N'$, the cones of $F \times F'$ are subsets of $(N \times N')_{\RR} = N_{\RR} \times N'_{\RR}$ of the form $\sigma \times \sigma'$, where $\sigma \in F$, $\sigma' \in F'$, and the lattice datum for such a cone is $F_\sigma \times F_{\sigma'}'$.

\begin{prop} \label{prop:products}  The product KM fan $F \times F'$ defined above is the product of $F$ and $F'$ in the category of KM fans.  Furthermore, this product commutes with realization so that $\AA(F \times F') = \AA(F) \times \AA(F')$. \end{prop}

\begin{proof}  Exercise. \end{proof}

\begin{rem} As in the usual theory of toric varieties, one does not expect the category of KM fans to have a good theory of more general finite inverse limits commuting with realization.  This is because, for example, the pushout of a diagram of toric monoids will not generally be toric (nor even integral).  If one wishes to have a satisfactory theory of finite inverse limits, then one should work in the category of stacks over fans in the CZE topology, just as one should work in the category of abstract fans if one wants a satisfactory inverse limit theory for classical toric varieties.  \end{rem}

\begin{example} \label{example:atoroidalsplitting} Any KM fan $F=(N,F,\{ F_\sigma \})$ can be written as a product $F = G \times (B,0,0)$, where $G$ is atoroidal (Definition~\ref{defn:nondegenerate}), $B$ is a lattice, and $(B,0,0)$ is the ``zero fan" associated to $B$ as in Example~\ref{example:grouptofan}.  To see this, let $A$ be the \emph{saturated} subgroup of $N$ generated by all the $F_\sigma$.  Then the image of \eqref{nondegeneratemap} is a finite index subgroup of $A$, so $G :=(A,``F", \{ F_\sigma \})$ is an atoroidal KM fan.  Since $A$ is saturated in $N$, $N/A =: B$ is a lattice and we can choose a splitting $N = A \oplus B$---this yields the desired splitting of $F$.  \end{example}

\subsection{Isotropy groups} \label{section:isotropygroups}  When studying an algebraic stack $X$, one of the most basic issues is to determine the isotropy groups of geometric points of $X$.  We shall now do this when $X=X(F)$ is the algebraic realization of a KM fan $F$.  Among other things, this will provide a stepping stone to the representability results established in \S\ref{section:representability}.

\begin{lem} \label{lem:isotropygroups}  Let $\ov{k}$ be an algebraically closed field of characteristic zero.  There exists an isomorphism of groups $\ov{k}^*_{\rm tor} \cong \QQ / \ZZ$.  For any finite abelian group $A$, there are natural isomorphisms $$ \Hom(\EE(A),\ov{k}^*) =  A = \EE(\EE(A)).$$  (The left isomorphism depends on the choice of an isomorphism $\ov{k}^*_{\rm tor} \cong \QQ / \ZZ$, but is natural in $A$.)  Here $\EE(A) := \Ext^1(A,\ZZ)$, as usual.  In other words, the group $\GG(\EE(A))(\ov{k})$ of $\ov{k}$ points of the finite $\ov{k}$ group scheme $\GG(\EE(A)) = \Spec \ov{k}[\EE(A)]$ is canonically identified with the group $A$. \end{lem}

\begin{proof}  Choose an exact sequence \bne{Ares} & 0 \to F \to G \to A \to 0 \ene with $F$ and $G$ lattices.  Then $\EE(A)$ is (up to unique isomorphism) the cokernel of of the injection $G^\lor \to F^\lor$.  So for the same reason, $\EE(\EE(A))$ is the cokernel of \be (F^{\lor \lor} \to G^{\lor \lor}) & = & (F \to G), \ee which is (up to canonical isomorphism) $A$.  (This is a special case of the evaluation isomorphism between any bounded complex of FGA groups and its double derived dual.)

We have \bne{ONE} \Hom(\EE(A),\ov{k}^*) & = & \Hom(\EE(A),\ov{k}^*_{\rm tor}) \ene since $\EE(A)$ is finite when $A$ is finite.  The hypotheses on $\ov{k}$ ensure that $\ov{k}$ contains an algebraic closure $\ov{\QQ}$ of $\QQ$, so $\ov{k}^*_{\rm tor} \cong \ov{\QQ}_{\rm tor}^*$ (since any root of unity is algebraic over $\QQ$).  If we think of $\ov{\QQ}$ as being contained in $\CC$, then we can define an isomorphism $\QQ / \ZZ \to \ov{\QQ}^*_{\rm tor}$ via the map $q \mapsto \exp( 2 \pi i q )$.  This proves the first statement.  Once we fix an isomorphism $\ov{k}^*_{\rm tor} \cong \QQ / \ZZ$, we have a short exact sequence \be 0 \to \ZZ \to \QQ \to \ov{k}^*_{\rm tor} \to 0. \ee  By applying $\Hom(A, \slot)$ to this exact sequence we obtain a natural isomorphism \bne{ONEA} \Hom(A,\ov{k}^*_{\rm tor}) & = & \EE(A). \ene  By applying $\Hom( \slot , \ov{k}^*_{\rm tor})$ to \eqref{ONEA}, we obtain a natural isomorphism \bne{TWO} \Hom( \Hom(A,\ov{k}^*_{\rm tor}),\ov{k}^*_{\rm tor}) & = & \Hom(\EE(A),\ov{k}^*_{\rm tor}). \ene  Finally, since $A$ is finite, the evaluation map yields a natural isomorphism \bne{THREE} A & = & \Hom( \Hom(A,\ov{k}^*_{\rm tor}), \ov{k}^*_{\rm tor}) . \ene  The natural isomorphism is obtained by combining \eqref{ONE}, \eqref{TWO}, and \eqref{THREE}. \end{proof}

\begin{prop} \label{prop:isotropygroups} Let $\ov{k}$ be an algebraic closure of our fixed characteristic zero base field $k$, equipped with a choice of isomorphism $\ov{k}^*_{\rm tor} \cong \QQ / \ZZ$ as in Lemma~\ref{lem:isotropygroups}.  Let $F$ be a KM fan with algebraic realization $X(F)$ and let $x \in X(F)(\ov{k})$ be a geometric point of $X(F)$.  By Proposition~\ref{prop:stratification}, $x$ lies in the stratum $X(N/F_\sigma,0,0) \subseteq X(F)$ for a unique cone $\sigma \in F$.  The isotropy group of $x$ (the automorphism group of $x$ as an object of the fiber category $X(F)(\ov{k})$) is naturally isomorphic to $(N/F_\sigma)_{\rm tor}$. \end{prop}

\begin{proof} Since the stratum $X(N/F_\sigma,0,0) \subseteq X(F)$ containing $x$ is locally closed in $X(F)$, the isotropy group of $x$ as an $\ov{k}$-point of $X(F)$ is the same as its isotropy group as a point of $X(N/F_\sigma,0,0)$.  We know from Example~\ref{example:atomicopen} that $X(N/F_\sigma,0,0)$ is a gerbe banded by $\GG(\EE(N/F_\sigma))$ over the torus $\GG((N/F_\sigma)^\lor)$, so the isotropy group of $x$ as an $\ov{k}$-point of $X(N/F_\sigma,0,0)$ is given by $\GG(\EE(N/F_\sigma))(\ov{k})$, which is identified with $N/F_\sigma$ by Lemma~\ref{lem:isotropygroups}. \end{proof}

\subsection{Fundamental group} \label{section:fundamentalgroup}  Throughout this section, we shall work over the base field $k = \CC$.  We shall think of the usual algebraic realization functor \be X : \Fans & \to & \Sch_{\CC} \ee simply as a functor \be X : \Fans & \to & \Top \ee by passing to the underlying analytic topological space.  Correspondingly, we shall view the algebraic realization $X(F)$ of a KM fan $F$ as a stack over the category of topological spaces by ``passing to the analytic topology."  With this understanding, we can speak of the fundamental group $\pi_1(X(F))$ of $X(F)$, which we shall now calculate:

\begin{thm} \label{thm:fundamentalgroup} Let $F=(N,F,\{ F_\sigma \})$ be a KM fan.  The fundamental group of the algebraic realization $X(F)$ is given by \be \pi_1(X(F)) & = & N / \sum_{\sigma \in F} F_\sigma. \ee \end{thm}

\begin{proof} The most important thing is simply to check that there at least \emph{exists a natural map} \bne{naturalfundamentalgroupmap} N / \sum_{\sigma \in F} F_\sigma & \to & \pi_1(X(F)). \ene  Indeed, each element $n \in N$ gives rise, by realizing an obvious map of KM fans $n : \GG_m \to F$ (Example~\ref{example:GmA1}), to a map of stacks $n : \GG(\ZZ)=\CC^* \to X(F)$ and hence to an element $n_*(\zeta) \in \pi_1(X(F))$, where $\zeta \in \pi_1(\CC^*)$ is the standard generator of $\pi_1(\CC^*)=\ZZ$.  The element $n_*(\zeta)$ is trivial whenever $n \in F_\sigma$ for some $\sigma \in F$, for then the map of stacks $n : \GG(\ZZ)=\CC^* \to X(F)$ extends to a map of stacks $\ov{n} : \GG(\NN)=\CC \to X(F)$ (again by realizing an obvious map of KM fans).  This completes the construction of the natural map \eqref{naturalfundamentalgroupmap}.  

To see that \eqref{naturalfundamentalgroupmap} is an isomorphism, we can use the same Mayer-Vietoris argument one would use in the classical case of toric varieties (together with induction on the total number of cones), to reduce to the case where $F$ has a unique maximal cone $\sigma$, in which case we are trying to show that the natural map $N/F_\sigma \to \pi_1(X(F))$ is an isomorphism.  Fix a lifting $L$ of $F_\sigma$.  Let $X(\sigma,L)$ be the affine toric variety corresponding to the cone $\sigma$ in the lattice $L$.  By construction, $X(F) = [X(\sigma,L) / G]$, where $G$ is the group given by the analytic topology on the $\CC$ points of the group scheme $\Spec \CC[\EE(N/L)]$.  But this finite $\CC$ group scheme is necessarily discrete and its group $G$ of $\CC$ points is naturally identified with the finite group $N/L$ via Lemma~\ref{lem:isotropygroups}.  Thus we see that $\pi_0(G)=N/L$ and $\pi_1(G)=0$.  The Serre homotopy fibration sequence for the $G$-bundle $X(F) \to X(\sigma,L)$ gives an exact sequence $$ \pi_1(G) \to \pi_1(X(\sigma,L)) \to \pi_1(X(F)) \to \pi_0(G) \to \pi_0(X(\sigma,L)),$$ which we can write as a short exact sequence \bne{Serreseq} 0 \to \pi_1(X(\sigma,L)) \to \pi_1(X(F)) \to N/L \to 0 \ene by using our calculation of $\pi_0(G)$ and $\pi_1(G)$ and the fact that a toric variety is connected.  Since we \emph{have the natural maps} \eqref{naturalfundamentalgroupmap}, we see that \eqref{Serreseq} receives a map from the obvious short exact sequence $$0 \to L/F_\sigma \to N/F_\sigma \to N/L \to 0.$$  By the ``classical" version of the statement we want to prove (just for the toric variety $X(\sigma,L)$), we know the map $L/F_\sigma \to \pi_1(X(\sigma,L))$ is an isomorphism, so we conclude isomorphy for the desired map by the Five Lemma. \end{proof}

\subsection{Quotient theorem} \label{section:quotients}  The purpose of this section is to characterize the maps of KM fans $F \to F'$ whose realizations are torsors under a group action compatible with the stratification defined in the previous section.

\begin{defn} \label{defn:tame} A map of groups $f : N \to N'$ is called \emph{tame} iff $\Cok f$ is finite and $f_{\rm tor} : N_{\rm tor} \to N'_{\rm tor}$ is injective (equivalently $\Ker f$ is torsion-free).  A map of KM fans $f: (N,F,\{ F_\sigma \}) \to (N',F',\{ F'_\tau \})$ is called \emph{semi-tame} (resp.\ \emph{tame}) iff the following conditions are satisfied: \begin{enumerate} \item \label{tame:equidimensional} $\sigma' := f_{\RR}(\sigma) \in F'$ for every $\sigma \in F$ and $f_{\RR} | \sigma : \sigma \to \sigma'$ is bijective, \item \label{tame:conebijection} the map $\sigma \mapsto \sigma'$ defines a bijection from $F$ to $F'$, and \item \label{tame:reducedfibers} $f|F_\sigma : F_\sigma \to F'_{\sigma'}$ is bijective for every $\sigma \in F$ \end{enumerate} (resp.\ and the map of abelian groups $f : N \to N'$ is tame in the previous sense). \end{defn}

\begin{example} The rigidification map $F \to F^{\rm rig}$ of Example~\ref{example:rigidification} is always semi-tame, but it is tame iff it is an isomorphism (iff $F$ is a lattice KM fan). \end{example}

\begin{example} \label{example:inflationistame} The inflation $F \to F'$ with respect to $N \subseteq N'$ discussed in Example~\ref{example:inflation} is tame.  The contraction and canonical resolution of Examples~\ref{example:contraction} and \ref{example:canonicalresolution} are not generally tame, though they are bijective on the underlying sets of cones. \end{example}

\begin{rem} If $f : N \to N'$ is a tame map of groups and $N'$ is a lattice, then $N$ must also be a lattice. \end{rem}

\begin{rem} A composition of tame maps of groups or fans is tame. \end{rem}

When we consider a tame map of KM fans as in the above definition and we refer to something like the ``cone $\sigma \in F$ corresponding to the cone $\sigma' \in F'$" we are of course referring to the bijection in \eqref{tame:conebijection}.  Eventually, in Theorem~\ref{thm:quotients2}, we will prove a result to the effect that the tame maps of KM fans are precisely those maps of KM fans whose realizations are torsors under a group action ``compatible with the orbit stratifications."  We begin by proving a special case of that result, which is important in its own right:

\begin{thm} \label{thm:quotients}  Let $F = (N,F,\{ F_\sigma \})$, $F' = (N',F',\{ F'_\tau \})$ be KM fans, $f : F \to F'$ a tame map of KM fans. \begin{enumerate} \item \label{quotients1} If $f : N \to N'$ is surjective with kernel $K := \Ker f$, then there is a natural action of $\GG(K^\lor)$ on $\AA(F)$ making $\AA(f) : \AA(F) \to \AA(F')$ a $\GG(K^\lor)$ torsor. \item \label{quotients2} If $f : N \to N'$ is injective, then the action of $\GG(\EE(N'/N))=\GG(\EE(\Cok f))$ on $\AA(F)$ makes $\AA(f)$ a $\GG(\EE(N'/N))$ torsor. \item \label{quotients3} If $N'$ (hence also $N$) is a lattice, then the action of $\GG(N^\lor / (N')^\lor)$ on $\AA(F)$ makes $\AA(f)$ a $\GG(N^\lor / (N')^\lor)$ torsor. \end{enumerate} \end{thm}

\begin{proof} The question is local on $\AA(F')$.  By Remark~\ref{rem:preimages} and the tameness assumption we reduce to proving that for any $\sigma \in F$, if we let $\sigma' := f_{\RR}(\sigma)$ be the corresponding cone of $F'$, then \bne{localtorsor} \AA(\sigma,F) & \to & \AA(\sigma',F') \ene is a torsor under the appropriate group action.  Suppose $L \subseteq N$ (resp.\ $L' \subseteq N'$) is a lifting of the lattice datum $F_\sigma$ (resp.\ $F_{\sigma'}$) with $f(L) \subseteq L'$.  (One can always find such $L$, $L'$ by Lemma~\ref{lem:liftings}.)  As discussed in \S\ref{section:KMfanrealization}, we can compute \eqref{localtorsor} as \bne{localtorsor2} \AA(\sigma,F,L) & \to & \AA(\sigma',F',L'). \ene

Assume the hypothesis in \eqref{quotients1}.  Let $L'$ be a lifting of the lattice datum $F_{\sigma'}$.  Since $f_{\rm tor}$ is injective and $f$ takes $F_{\sigma}$ bijectively onto $F_{\sigma'}$ (since $f$ is tame), $L := f^{-1}(L')$ is a lifting of the lattice datum $F_\sigma$ by Lemma~\ref{lem:liftings}.  Since $f : N \to N'$ is surjective, the definition of $L$ ensures that the indicated arrow in the map of exact sequences \bne{SEStorsormap} & \xym{ 0 \ar[r] & L \ar[r] \ar[d]_{f|L} & N \ar[d]_f \ar[r] & N/L \ar[r] \ar[d]_{\cong} & 0 \\ 0 \ar[r] & L' \ar[r] & N' \ar[r] & N'/L' \ar[r] & 0 } \ene is an isomorphism and that the sequence \bne{latticeSES} & 0 \to K \to L \to L' \to 0 \ene is exact (recall that $K := \Ker(f : N \to N')$).  Dualizing \eqref{SEStorsormap} we obtain a map of exact sequences \bne{SEStorsormap2} & \xym{ 0 \ar[r] & M \ar[r] & L^\lor \ar[r] & \EE(N/L) \ar[r] & \EE(N) \ar[r] & 0 \\ 0 \ar[r] & M' \ar[r] \ar[u]^{f^\lor} & (L')^\lor \ar[r] \ar[u]^{(f|L)^\lor} & \EE(N'/L') \ar[u]^{\cong} \ar[r] & \EE(N') \ar[u] \ar[r] & 0 } \ene with the indicated arrow an isomorphism.  Since \eqref{latticeSES} is a short exact sequence of \emph{lattices}, we have $K^\lor = \Cok( (f|L)^\lor)$.  Since $f$ is tame, we see from Lemma~\ref{lem:torsor} (with $f$ there given by $f|L : L \to L'$ here) that \bne{localtorsor3} \AA(S_\sigma(L)) & \to & \AA(S_{\sigma'}(L')) \ene is a $\GG(K^\lor)$ torsor.  The map \eqref{localtorsor3} is $G := \GG(\EE(N/L)) = \GG(\EE(N'/L'))$ equivariant and the map \eqref{localtorsor2} is, by definition, the induced map on the stack quotients by $G$.  The $G$ action and the $\GG(K^\lor)$ action on $\AA(S_\sigma(L))$ commute because both actions are defined by mapping these group objects to the abelian group object $\GG(L^\lor)$, so by general nonsense, the $\GG(K^\lor)$ action on $\AA(S_\sigma(L))$ descends to a $\GG(K^\lor)$ action on the quotient \be \AA(\sigma,F,L) & = & [ \AA(S_\sigma(L)) / G ] \ee making \eqref{localtorsor2} a $\GG(K^\lor)$ torsor, which, incidentally, fits into a cartesian diagram \bne{cartesiantorsordiagram} & \xym{ \AA(S_\sigma(L)) \ar[r] \ar[d] & [ \AA(S_\sigma(L)) / G ] = \AA(\sigma,F,L) \ar[d] \\ \AA(S_{\sigma'}(L')) \ar[r] & [ \AA(S_{\sigma'}(L')) / G ] = \AA(\sigma',F',L').} \ene

Now assume the hypothesis in \eqref{quotients2}.  Pick a lifting $L \subseteq N$ of the lattice datum $F_\sigma$.  Since $f$ is tame and $N \into N'$ is injective with finite cokernel, $L \subseteq N \subseteq N'$ is also a lifting of $F_{\sigma'}$ and we have $S_\sigma(L) = S_{\sigma'}(L)$.  Applying the Snake Lemma to the map of exact sequences $$ \xym{ 0 \ar[r] & L \ar[r] \ar@{=}[d] & N \ar[r] \ar[d] & N/L \ar[r] \ar[d] & 0 \\ 0 \ar[r] & L \ar[r] & N' \ar[r] & N'/L \ar[r] & 0 } $$ we obtain an exact sequence of finite abelian groups \bne{fagSES} & 0 \to N/L \to N'/L \to N'/N \to 0. \ene Dualizing \eqref{fagSES} we obtain a similar exact sequence \bne{fagSESdual} & 0 \to \EE(N'/N) \to \EE(N'/L) \to \EE(N/L) \to 0 \ene and hence an exact sequence of abelian group objects \bne{goSES} & 0 \to \GG(\EE(N/L)) \to \GG(\EE(N'/L)) \to \GG(\EE(N'/N)) \to 0. \ene  In this case, the map \eqref{localtorsor2} is given by \bne{localtorsor4} [ \AA(S_\sigma(L)) / \GG(\EE(N/L)) ] & \to & [ \AA(S_\sigma(L)) / \GG(\EE(N'/L)) ], \ene which is a $\GG(\EE(N'/N))$ torsor by Lemma~\ref{lem:quotients}. 

We will obtain \eqref{quotients3} by combining \eqref{quotients1} and \eqref{quotients2} in much the same way that we bootstrapped up in the proof of Lemma~\ref{lem:torsor}.  Let $N''$ be the sublattice of $N'$ given by the image of $f : N \to N'$.  Since $f$ is tame, we can take images of cones and lattice data under $f : N \to N''$ to define a KM fan $F''$ in $N''$ so that the maps \bne{Nmaps} & N \to N'' \to N' \ene define tame maps of KM fans \bne{tameKMfanmaps} & F \to F'' \to F'. \ene  By the parts we already proved, we know that $\AA(F) \to \AA(F'')$ is a torsor under $\GG(K^\lor)$, where $K := \Ker f = \Ker (N \to N'')$, and $\AA(F'') \to \AA(F')$ is a torsor under $\GG(\EE(N'/N''))$.  We can conclude that the composition $\AA(F) \to \AA(F')$ is a torsor under $\GG(N^\lor / (N')^\lor)$ by showing there is an exact sequence \bne{desiredSES} & 0 \to \EE(N'/N'') \to N^\lor / (N')^\lor \to K^\lor \to 0. \ene  Since the maps in \eqref{Nmaps} have finite cokernels, dualizing gives inclusions \bne{Minclusions} & (N')^\lor \into (N'')^\lor \into N^\lor  \ene and hence an exact sequence \bne{desiredSES2} & 0 \to (N'')^\lor / (N')^\lor \to N^\lor / (N')^\lor \to N^\lor / (N'')^\lor \to 0. \ene  Since $N''$ is free, dualizing \bne{gfg} & 0 \to K \to N \to N'' \to 0 \ene shows that $K^\lor = N^\lor / (N'')^\lor$ and since $N'$ is free, dualizing \bne{gfg2} & 0 \to N'' \to N' \to N'/N'' \to 0 \ene shows that $\EE(N'/N'') = (N'')^\lor / (N')^\lor$.  In light of these identifications, \eqref{desiredSES2} is the desired exact sequence \eqref{desiredSES}. \end{proof}

\begin{example} \label{example:P22revisited}  The map of KM fans $f : \tilde{F} \to F$ described in Example~\ref{example:P22} is a tame map satisfying the first assumption of Theorem~\ref{thm:quotients}.  The underlying map of groups $f : \ZZ^2 \to \ZZ \oplus \ZZ / 2 \ZZ$ fits into an exact sequence \bne{P22SES} & \xym@C+15pt{ 0 \ar[r] & \ZZ \ar[r]_-{ \bp 2 \\ 2 \ep } & \ZZ^2 \ar[r]^-{f}_-{ \bp 1 & -1 \\ 1 & 0 \ep } & \ZZ \oplus \ZZ / 2 \ZZ \ar[r] & 0 } \ene so that $K := \Ker f = \ZZ$.  Dualizing \eqref{P22SES} yields an exact sequence \bne{P22SESdual} & \xym@C+15pt{ 0 \ar[r] & \ZZ \ar[r]^-{f^\lor}_-{ \bp 1 \\ -1 \ep } & \ZZ^2 \ar[r]_-{ \bp 2 & 2 \ep } & \ZZ \ar[r] & \ZZ / 2 \ZZ \ar[r] & 0 } \ene so that the natural map $M=N^\lor \to K$ in Theorem~\ref{thm:quotients} is given by $(2,2) : \ZZ^2 \to \ZZ$.  The torsorial action of $\GG(K^\lor) = \GG_m$ on $X(\tilde{F}) = \AA^2 \setminus \{ 0 \}$ in Theorem~\ref{thm:quotients} in this case is given by $t \cdot (x,y) = (t^2 x, t^2 y)$.  Since this action makes $\AA(f)$ a $\GG_m$ torsor, we see that $$ X(F) = [ X(\tilde{F}) / \GG_m ] = \PP(2,2) $$ is a weighted projective line and $\AA(f)$ is the ``tautological" $\GG_m$ torsor over it.  \end{example}

Theorem~\ref{thm:quotients}---particularly the proof of \eqref{quotients3}---points us in the direction of the most general statement along these lines.  Suppose $f : F \to F'$ is an arbitrary tame map of KM fans.  We want to show that $\AA(f)$ is a torsor under $\GG(G)$ for \emph{some} group $G=\DD(f)$ depending naturally on $f$ \dots but which one?  A good way to figure this out is to write down all of the properties we expect of $\DD(f)$ and \emph{then} try to find the group satisfying those properties.  We can always factor $f : N \to N'$ as a (tame) surjection $N \to N''$ followed by a (tame) injection $N'' \to N'$, then factor $f : F \to F'$ correspondingly as $F \to F''$ followed by $F'' \to F$, just as in the above proof.  We know from Theorem~\ref{thm:quotients} that $\AA(F) \to \AA(F'')$ is a $\GG((\Ker f)^\lor)$ torsor and $\AA(F'') \to \AA(F)$ is a $\GG(\EE(\Cok f))$ torsor, so, as in the proof of Theorem~\ref{thm:quotients}\eqref{quotients3}, we expect $\DD(f)$ to sit in a natural short exact sequence \bne{Gsequence} & 0 \to \EE(\Cok f) \to \DD(f) \to (\Ker f)^\lor \to 0 \ene like the sequence \eqref{desiredSES} in that proof.  We also want the \emph{special formulas} \bne{specialformulas} \DD(f) & = & (\Ker f)^\lor \quad \quad (f {\rm \; surjective}) \\ \nonumber \DD(f) & = & \EE(\Cok f) \quad \quad (f {\rm \; injective}) \\ \nonumber \DD(f) & = & \Cok (f^\lor) \quad \quad (f {\rm \; a \; map \; of \; lattices}) \ene to hold for special types of tame $f$, and we want the action of $\GG(\DD(f))$ on $\AA(f)$ to be the one described in the above theorem in these special cases.  

Now anyone familiar with the idea of the derived category will quickly figure out how to define $\DD(f)$.  

\begin{defn} \label{defn:DDf} For a map of FGA groups $f : N \to N'$, let $[f] := [N \to N']$ denote the \emph{two-term complex associated to} $f$ with $N$ and $N'$ placed in degrees $0$ and $1$, with $f$ as the coboundary map between them, and with all other cochain groups and coboundary maps zero.  (Up to shifting, this is the mapping cone of $f$.)  Viewing $[f]$ as an object of the derived category $\DAb$ of abelian groups, we can consider the group \be \DD(f) & := & \Hom_{\DAb}([f], \ZZ) \\ & = & \H^0(\RHom([f],\ZZ)). \ee \end{defn}

We will eventually see that this $\DD(f)$ is as desired.  Let us first verify that this $\DD(f)$ sits in an exact sequence \eqref{Gsequence} and that the expected ``formulas" \eqref{specialformulas} hold for the special sorts of $f$ mentioned above.  First of all, we have two short exact sequences of chain complexes \bne{SESf1} & 0 \to N'[-1] \to [f] \to N \to 0 \\ \label{SESf2} & 0 \to \Ker f \to [f] \to (\Cok f)[-1] \to 0, \ene where our shifting convention is the usual one, so that, for a group $A$, the complex $A[-1]$ consists of $A$ in degree $1$ and zero elsewhere.  We can view either of these short exact sequences as a distinguished triangle in $\DAb$, then apply $\RHom( \slot, \ZZ)$ to obtain a new distinguished triangle, then look at the long exact cohomology sequence associated to that distinguished triangle.  The non-zero part of the exact sequence thus obtained from \eqref{SESf1} looks like \bne{LESf1} & 0 \to \H^{-1}(\RHom([f],\ZZ)) \to (N')^\lor \to \\ \nonumber & \to  N^\lor \to \DD(f) \to \EE(N') \to \\ \nonumber & \to \EE(N) \to \H^1( \RHom([f],\ZZ)) \to 0. \ene  (The three lines are the $\H^{-1}$ line, the $\H^0$ line, and the $\H^1$ line.)  From \eqref{LESf1} we see that $\H^{-1}(\RHom([f],\ZZ))=0$ iff $f^\lor : (N')^\lor \to N^\lor$ is injective (iff $\Cok f$ is finite).  We also see that   $\H^1(\RHom([f],\ZZ))=0$ iff $\EE(N') \to \EE(N)$ is surjective (iff $\EE(N'_{\rm tor}) \to \EE(N_{\rm tor})$ is surjective iff $f_{\rm tor} : N_{\rm tor} \to N'_{\rm tor}$ is injective).  We have proved:

\begin{prop} \label{prop:tamemaps} A map of FGA groups $f : N \to N'$ is tame iff the augmentation map \be \DD(f) & \to & \RHom([f],\ZZ) \ee is an isomorphism in $\DAb$ (iff $\H^i( \RHom([f],\ZZ) )=0$ for all $i \neq 0$). \end{prop}

From \eqref{SESf2} we obtain the short exact sequence \eqref{Gsequence} (for tame $f$). Indeed, \eqref{Gsequence} is $\H^0(\RHom( \slot, \ZZ))$ of \eqref{SESf2}, which is exact because the cohomology of $\RHom( \slot, \ZZ)$ vanishes in non-zero degrees when $\slot$ is one of the two-term complexes $[f]$, $\Ker f$, $(\Cok f)[-1]$ associated to one of the tame maps $f$, $(\Ker f) \to 0$, $0 \to (\Cok f)$.  The ``special formulas" \eqref{specialformulas} follow easily from the exact sequences \eqref{LESf1} and \eqref{Gsequence}.

\begin{example} \label{example:inflationrevisited} {\bf (Finite Quotients)}  As mentioned in Example~\ref{example:inflationistame}, the inflation $F \to F'$ with respect to a finite index inclusion $N \subseteq N'$ (cf.\ Example~\ref{example:inflation}) is tame, so by Theorem~\ref{thm:quotients}\eqref{quotients2}, $\AA(F) \to \AA(F')$ is a torsor under $\GG(\EE(N'/N))$.  In other words, the map $\AA(F) \to \AA(F')$ is identified with the natural map $\AA(F) \to [ \AA(F) / \GG(\EE(N'/N)) ].$  (In particular, the differential realization $Y(F) \to Y(F')$ will be an isomorphism of log differentiable spaces because the differential realization of $\GG(\EE(N'/N))$ is the trivial group.)  Let us specialize now to the case where $F$ is a lattice KM fan (i.e.\ $N$ is a lattice).  Then $\GG(\EE(N'/N))$ acts on $\AA(F)$ through the $T = \GG(N^\lor)$ action on $\AA(F)$ (cf.\ Remark~\ref{rem:torusaction2}) and the map of group objects $\GG(\EE(N'/N)) \to \GG(N^\lor)$ (note that this may fail to be injective when $N'$ has torsion) induced by the natural map $N^\lor \to \EE(N'/N)$.  Suppose, on the other hand, that we are given a finite abelian group $A$ and a map of groups $g : N^\lor \to A$ (not necessarily surjective).  This gives us a map of groups $\GG(A) \to T$, hence an action of $\GG(A)$ on $\AA(F)$.  We claim that there is a finite index inclusion $N \subseteq N'$ such that the natural map $\AA(F) \to [ \AA(F) / \GG(A) ]$ is the realization of the inflation of $F$ with respect to $N \subseteq N'$.  (In particular, $[ \AA(F) / \GG(A) ]$ is the realization of a KM stack.)  This is a simple application of the ``$\DD$" construction from Definition~\ref{defn:DDf}---we'll see that we can take $N' := \DD(g)$.  Note that $g$ is tame (Definition~\ref{defn:tame}).  The exact sequence \eqref{SESf1} with the $f$ there equal to our $g$ takes the form \bne{SESg} & 0 \to A[-1] \to [g] \to N^\lor \to 0. \ene  Each of the three terms in the sequence \eqref{SESg} is the two-term complex attached to a tame map of groups, so, much as in the discussion just before this example, the exact sequence obtained from \eqref{SESg} by applying $\RHom( \slot , \ZZ)$ and taking cohomology takes the form \bne{SESg2} & 0 \to N \to N' \to \EE(A) \to 0. \ene  This gives us the finite index inclusion $N \subseteq N'$ with respect to which we will inflate, and we obtain the identification $\EE(N'/N)=\EE(\EE(A))=A$ from Lemma~\ref{lem:isotropygroups}.  \end{example}

Once we have the right $\DD(f)$, the rest is not so hard:

\begin{thm} \label{thm:quotients2} If $f : F \to F'$ is a tame map of KM fans, then any realization $\AA(f)$ of $f$ is a torsor under $\GG(\DD(f))$ (Definition~\ref{defn:DDf}).  ``Conversely," suppose that $f : F \to F'$ is a map of KM fans whose algebraic realization $X(f) : X(F) \to X(F')$ satisfies the following properties: \begin{enumerate} \item \label{P1} $X(f)$ is surjective on geometric points \item \label{P2} $X(f)$ is representable. \item \label{P3} $X(f)$ is equidimensional. \item \label{P4} $X(f)$ has reduced fibers. \item \label{P5} The induced map $X(\ov{f}) : X(\ov{F}) \to X(\ov{F}')$ from the coarse moduli space of $X(f)$ to the coarse moduli space of $X(F')$ induces a bijection between the set of torus orbits in the toric variety $X(\ov{F})$ and the set of torus orbits in the toric variety $X(\ov{F}')$. \end{enumerate}  Then $f$ is tame. \end{thm}

\begin{proof} Set $G := \DD(f)$.  For the first statement, suppose $f$ is tame and let us show that $\AA(f)$ is such a torsor.  As in the proof of Theorem~\ref{thm:quotients} we can reduce to the case where $F$ contains a single maximal cone $\sigma$ and hence $F'$ contains a single maximal cone $\sigma'$ taken bijectively onto $\sigma$ by $f_{\RR} : N_{\RR} \to N'_{\RR}$.  Let $L' \subseteq N'$ be a lifting of $F_{\sigma'}'$ and let $L \subseteq N$ be a lifting of $F_\sigma$ taken into $L'$ by $f : N \to N'$ (call this map $f_L : L \to L'$), as in Lemma~\ref{lem:liftings}.  (That lemma asserts that we can take $L = f^{-1}(L')$, but we shall not need to make that particular choice of $L$, as we did in the proof of Theorem~\ref{thm:quotients}.)  According to the construction of the realization $\AA(f)$ (\S\ref{section:KMfanrealization}), the map $\AA(f) : \AA(F) \to \AA(F')$ is given by the natural map \bne{GGtorsor} \AA(F) = [ \AA(S_\sigma(L)) / \GG(\EE(N/L)) ] & \to & [ \AA(S_{\sigma'}(L')) / \GG(\EE(N'/L')) ] = \AA(F'). \ene  We want to show that this is a $\GG(G)$ torsor. 

First notice that the map $\AA(S_\sigma(L)) \to \AA(S_{\sigma'}(L'))$ is the realization of an evident tame map of KM fans (in fact, classical fans) whose underlying map of ``lattices" is the map of lattices $f_L : L \to L'$, so we know from Theorem~\ref{thm:quotients}\eqref{quotients3} that the natural action of $\GG(\Cok(f_L^\lor))$ on $\AA(S_\sigma(L))$ makes this map a $\GG(\Cok(f_L^\lor))$ torsor---that is, we have \bne{Hformula} [ \AA(S_\sigma(L)) / \GG(\Cok(f_L^\lor)) ] & = & \AA(S_{\sigma'}(L')). \ene

Let us assume for a moment that we have constructed a group $G'$, a group homomorphism $L^\lor \to G'$ and two exact sequences \bne{HSES1} & 0 \to G \to G' \to \EE(N/L) \to 0 \\ \label{HSES2} & 0 \to \EE(N'/L') \to G' \to \Cok(f_L^\lor) \to 0. \ene  (We'll also assume that the composition of $L^\lor \to G'$ and the map $G' \to \Cok(f_L^\lor)$ in \eqref{HSES2} is just the usual obvious projection $L^\lor \to \Cok(f_L^\lor)$.)  The action of $\GG(L^\lor)$ on $\AA(S_\sigma(L))$ and the map $L^\lor \to G'$ yield an action of $\GG(H)$ on $\AA(S_\sigma(L))$.  By applying Lemma~\ref{lem:quotients} with $X = \AA(S_\sigma(L))$ and with the exact sequence of groups given by $\GG( \slot )$ of \eqref{HSES1}, we obtain an action of $\GG(G)$ on $\AA(F)$ making the natural map $\AA(F) \to [\AA(S_\sigma(L)) / \GG(G') ]$ a $\GG(G)$ torsor.  On the other hand, by applying that lemma to the same $X$ but with the exact sequence of groups given instead by $\GG( \slot )$ of \eqref{HSES2}, we see that \be [\AA(S_\sigma(L)) / \GG(G') ] & = & [ [ \AA(S_\sigma(L)) /  \GG( \Cok(f_L^\lor) ) ] / \GG(\EE(N'/L')) ] \\ & = & [ \AA(S_{\sigma'})(L') / \GG(\EE(N'/L')) ] \\ & = & \AA(F'), \ee using \eqref{Hformula} for the second equality.  We thus conclude that the natural map \eqref{GGtorsor} is a $\GG(G)$ torsor.

For the proof of the first statement it remains only to construct $G'$, the map $L^\lor \to G'$, and the sequences \eqref{HSES1} and \eqref{HSES2}.  To do this, notice that we have a tame map of groups $f' : L \to N'$, which may be thought of either as the composition of the tame maps $L \into N$ and $f: N \to N'$ or as the composition of the tame maps $f_L : L \to L'$ and $L' \into N'$.  We set $G'  :=  \DD(f')$.  The map $L^\lor \to G'$ is defined to be the natural map appearing in the long exact sequence analogous to \eqref{LESf1}, but for the tame map $f'$, rather than the one for $f$ shown in \eqref{LESf1}.  The exact sequence \eqref{HSES1} is obtained from the evident short exact sequence of chain complexes \bne{HSES3} & 0 \to [f'] \to [f] \to N/L \to 0 \ene by the usual rigmarole (view \eqref{HSES3} as a distinguished triangle in $\DAb$, then apply $\RHom( \slot , \ZZ)$ and take cohomology).  The exact sequence \eqref{HSES2} is obtained from the evident short exact sequence of chain complexes \bne{HSES4} & 0 \to [f_L] \to [f] \to N'/L'[-1] \to 0 \ene in the same fashion.  (It is easy to see that the composition $L^\lor \to \Cok(f_L^\lor)$ described parenthetically above is just the obvious projection.)  The proof of the first statement is complete. 

For the second part, consider a map of KM fans $f : F \to F'$ satisfying the list of properties.  We need to show that $f$ is tame.  First of all, $f : N \to N'$ must certainly have finite cokernel, otherwise \eqref{P1} would not hold.  (One can pass to the underlying coarse fans here, so this is easy.)  Pick a geometric point $x \in X(F)(\ov{k})$ lying in the ``atomic open" subset $X(0) \subseteq X(F)$ (Example~\ref{example:atomicopen}).  By Proposition~\ref{prop:isotropygroups}, the map $f_{\rm tor} : N_{\rm tor} \to N'_{\rm tor}$ can be identified with the map of isotropy groups \be f_* : \Aut_{X(F)(\ov{k})}(x) & \to & \Aut_{X(F')(\ov{k})}(f(x)), \ee which must be injective by the assumption \eqref{P2}.  We have shown that the map of groups $f : N \to N'$ is tame.  By Proposition~\ref{prop:equidimensional}, the properties \eqref{P3} and \eqref{P4} of $f$ ensure that for each cone $\sigma \in F$, $\sigma' := f_{\RR}(\sigma) \in F$ and $f|F_\sigma : F_{\sigma} \to F_{\sigma'}$ is surjective.  Looking at Definition~\ref{defn:tame}, it remains to see that each $f_{\RR} : \sigma \to \sigma'$ is bijective and that each map $f|F_\sigma : F_{\sigma} \to F'_{\sigma'}$ is actually bijective.  If the surjection $f_{\RR} : \sigma \to \sigma'$ were not bijective, then a standard lemma (exercise!) says that there would be some proper face $\tau < \sigma$ such that $\tau' = f_{\RR}(\tau)$ contains an interior point of $\sigma'$, which would imply, by \eqref{P3}, that $\tau' = \sigma'$, in violation of \eqref{P5}.  (Under assumption \eqref{P3}, the map described in \eqref{P5} is just the map of finite sets $F \to F'$ given by $\sigma \mapsto \sigma'$.)  Similarly, if the surjective map of lattices $f|F_\sigma : F_{\sigma} \to F'_{\sigma'}$ were not bijective, then we'd have \be \dim \sigma = \dim F_\sigma > \dim F'_{\sigma'} = \dim \sigma' \ee contradicting the bijectivity of $f_{\RR} : \sigma \to \sigma'$ that we just established. \end{proof}

\subsection{Toric DM stacks} \label{section:toricstacks}  In this section we will prove that the algebraic realization functor $X$ is fully faithful on \emph{lattice} KM fans and we will characterize its essential image.  We remind the reader that, in defining and working with the algebraic realization, we \emph{work over a fixed base field} $k$ \emph{of characteristic zero}.

\begin{defn} \label{defn:toricstack} (cf.\ \cite{GS2}) A \emph{toric DM stack} is a separated, normal Deligne-Mumford stack (of finite type over $k$) equipped with an action of a torus $T$ having an open, dense orbit isomorphic to $T$.  Toric DM stacks form a category where morphisms are required to be torus equivariant for some map of tori.  \end{defn}

\begin{rem} \label{rem:morphismsoftoricstacks}  Since DM stacks form a $2$-category, the ``category" of toric DM stacks is \emph{a priori} a $2$-category (with all $2$-morphisms invertible), but actually one can show that a $1$-morphism of toric DM stacks does not have any non-trivial automorphisms in the $2$-category of stacks (cf.\ \cite[Lemma~4.2.3]{AV}---the assumption there that $F$ be representable is not needed and indeed, is not used in the proof there). \end{rem}

Proposition~\ref{prop:realization} ensures that the algebraic realization $X(F)$ of a lattice KM fan $F$ is a toric DM stack and the realization of a morphism of lattice KM fans is a morphism of toric DM stacks.  Thus $X$ defines a functor from lattice KM fans to toric DM stacks---we will eventually show (Theorem~\ref{thm:main}) that this is an equivalence of categories.

\begin{rem} \label{rem:toricstacks}  In \cite{GS2}, the authors consider a more general theory of \emph{toric stacks} $\mathcal{X}$ where our ``separated and DM" assumptions are replaced by: \begin{enumerate} \item $\mathcal{X}$ is an algebraic stack of finite presentation over $k$ \item $\mathcal{X}$ has affine diagonal \item Geometric points of $\mathcal{X}$ have linearly reductive stabilizers. \end{enumerate} Every toric DM stack is a ``toric stack" in their sense because every DM stack has finite (hence affine) diagonal and every geometric point of a DM stack has finite (hence linearly reductive, because we work over a field of characteristic zero) stabilizers. \end{rem}

The main step in the proof of the ``classical" theorem that every toric variety is the algebraic realization of a classical fan (we shall use this ``classical" result in establishing our analogous ``stacky" results) is Sumihiro's Theorem, which asserts that every toric variety has an open cover by torus invariant \emph{affine} toric varieties.  The analogous result for toric DM stacks, below, is essentially due to A.~Geraschenko and M.~Satriano:

\begin{lem}[Sumihiro's Theorem] \label{lem:Sumihiro} Let $\mathcal{X}$ be a toric DM stack with torus $T$.  Every geometric point $x \in \mathcal{X}$ is contained in a $T$-invariant open substack $\mathcal{Y} \subseteq \mathcal{X}$ of the form $\mathcal{Y} = [ W / S]$, where $W$ is an affine toric variety with torus $T'$, $S$ is a finite subgroup scheme of $T'$, and there is an exact sequence of group schemes \bne{toriexactsequence} & 0 \to S \to T' \to T \to 0 \ene so that the $T$-action on $\mathcal{Y}$ inherited from the $T$-action on $\mathcal{X}$ coincides with the $T$-action on $\mathcal{Y}$ coming from \eqref{toriexactsequence} and the quotient description of $Y$ (cf.\ Lemma~\ref{lem:quotients}). \end{lem}

\begin{proof} By \cite[Theorem~4.5]{GS2} (and Remark~\ref{rem:toricstacks}), applied to the stack $[\mathcal{X}/T]$, we can find an affine toric variety $Y$ with torus $H$ and an open embedding $[Y/H] \into [\mathcal{X}/T]$ taking the image of the distinguished point $e \in Y$ in $[Y/H]$ to the image of $x$ in $[\mathcal{X}/T]$.  Define $\mathcal{Y}$ and $Z$ by the cartesian diagram \bne{Sumihirodiagram} & \xym{ Z \ar[r] \ar[d] & \mathcal{Y} \ar[r] \ar[d] & \mathcal{X} \ar[d] \\ Y \ar[r] & [Y/H] \ar[r] & [\mathcal{X}/T].} \ene Clearly $\mathcal{Y}$ is a $T$-invariant open substack of $\mathcal{X}$ containing $x$, so it remains only to show that $\mathcal{Y}$ is of the desired form.

Since $Z \to Y$ is a base change of the $T$-torsor $\mathcal{X} \to [\mathcal{X}/T]$ it is also a $T$-torsor.  Furthermore, $Y$ is an affine toric variety, so $Z$ is also an affine toric variety, with torus $T \times H$.  The $H$ action on $Z \to Y$ makes $Z \to Y$ an $H$ torsor because the left square in \eqref{Sumihirodiagram} is cartesian and $Y \to [Y/H]$ is an $H$ torsor, hence we have $\mathcal{Y} = [ Z / H ]$.  We aren't done yet though, because $H$ may not be finite, so we must continue to refine our description of $\mathcal{Y}$.  Let $N$ (resp.\ $A$) be the cocharacter lattice of $T$ (resp.\ $H$), so that $N \oplus A$ is the cocharacter lattice of $T \times H$ and $Z$ corresponds (``classically") to a (sharp, rational) cone $\sigma$ in $N_{\RR} \oplus A_{\RR}$.  Let $z$ be the distinguished point of the affine toric variety $Z$.  Since $\mathcal{Y} = [ Z / H ]$ is an open substack of the DM stack $\mathcal{X}$, the stabilizer of $z$ in $H$ must be finite.  By a standard exercise from the classical theory of toric varieties, this finiteness is equivalent to \bne{Sumiconebij} \Span \tau \cap ( \{ 0 \} \times A_{\RR}) & = & \{ 0 \}. \ene  Let $p : N \oplus A \to A$ be the projection.  Then \eqref{Sumiconebij} implies that $p_{\RR}$ takes $\sigma$ bijectively onto its image $\tau := p_{\RR}(\sigma) \subseteq N_{\RR}$ and $p$ takes $B := (N \oplus A) \cap \Span \sigma$ bijectively onto its image $$F_{\tau} := p(B) \subseteq N_{\tau} = N \cap \Span \tau.$$  If $n \in N_{\tau}$, then, by definition of $\tau$, we see that $(n,s) \in \Span \tau$ for some $s \in A_{\RR}$.  In fact, since $n$ is integral and our cones are rational, we can arrange that $s$ can be taken in $A_{\QQ}$ and then, clearing denominators, we see that some positive integer multiple of $n$ is in $F_{\tau}$.  This proves that the inclusion $F_{\tau} \subseteq N_{\tau}$ has finite index, so by letting $F_{\rho} := F_{\tau} \cap \Span \rho$ be the lattice datum for each face $\rho \leq \tau$, we obtain a lattice KM fan $F=(N, [ \tau] , \{ F_\rho \})$.  Furthermore, the map $p$ defines a map of KM fans from the classical fan $(N \oplus A, [ \sigma ])$ (whose realization is $Z$) to $F$ whose realization $X(p) : Z \to X(F)$ is an $H$-torsor (for the obvious action of $H$ on $Z$) by construction of $F$ and the criterion of Theorem~\ref{thm:quotients}\eqref{quotients1}.  

Therefore $X(F) = [ Z / H ] = \mathcal{Y}$, so $\mathcal{Y}$ will have the desired form by construction of $X(F)$.  Indeed, if we pick a lifting $L \subseteq N$ of $F_{\tau}$, then, by construction, $X(F) = [ W / S]$ where $W$ is the affine toric variety corresponding to the cone $\tau \subseteq N_{\RR} = L_{\RR}$ with respect to the lattice $L$, the torus $T'$ of $W$ is $\GG(L^\lor)$ and \eqref{toriexactsequence} is $\GG( \slot )$ of \bne{toriexactsequence2} & 0 \to N^\lor \to L^\lor \to \EE(N/L) \to 0.\ene \end{proof}

\begin{cor} \label{cor:Sumihiro}  Let $\mathcal{X}$ be a toric DM stack with torus $T$ and coarse moduli space $\pi : \mathcal{X} \to \ov{\mathcal{X}}$ (cf.\ \S\ref{section:coarsemodulispace}).  Then: \begin{enumerate} \item \label{Sumi1} $\ov{\mathcal{X}}$ is a toric variety with torus $T$ and $\pi$ is $T$-equivariant. \item \label{Sumi2} For each $T$ invariant affine open subset $Y \subseteq \ov{\mathcal{X}}$, the preimage $\mathcal{Y} := \pi^{-1}(Y)$ of $Y$ in $\mathcal{X}$ is a $T$ invariant open substack of $\mathcal{X}$ which can be expressed ($T$ equivariantly) as a quotient of an affine toric variety by a finite subgroup of its torus. \end{enumerate} \end{cor}

\begin{proof}  First of all, $T$ acts on $\ov{\mathcal{X}}$ making $\pi$ equivariant simply because the formation of coarse moduli spaces is functorial and commutes with products.  By definition, $\mathcal{X}$ has an open, dense $T$ orbit (call it $U$) $T$-equivariantly isomorphic to $T$.  Since $\pi$ is a homeomorphism (cf.\ Theorem~\ref{thm:cms}) and is $T$-equivariant, $\pi(U)$ is an open, dense $T$-invariant subspace of $\ov{\mathcal{X}}$ with preimage $U$ in $\mathcal{X}$.  As a special case of the fact that ``formation of coarse moduli spaces commutes with flat base change" (cf.\ Theorem~\ref{thm:cms}), this means that $\pi(U) = \ov{U}$ must be the coarse moduli space of $U$, but $\ov{U}=U$ since $U \cong T$ is already an algebraic space (even a scheme).  This proves that the $T$ action on $\ov{\mathcal{X}}$ has an open, dense orbit isomorphic to $T$.  By definition, $\mathcal{X}$ is separated, hence $\ov{\mathcal{X}}$ is also separated by general properties of the coarse moduli space map (Theorem~\ref{thm:cms}).  This reduces us to proving \eqref{Sumi1} locally on $\mathcal{X}$ (equivalently, locally on $\ov{\mathcal{X}}$).

Since $\mathcal{X}$ is of finite type, we can cover $\mathcal{X}$ by finitely many $T$ invariant open substacks of the form $\mathcal{Y} = [ W / S ]$, as in Lemma~\ref{lem:Sumihiro}.  Each such $\mathcal{Y}$ is the algebraic realization of a KM fan, so (cf.\ \S\ref{section:coarsemodulispace}) the coarse moduli space $\ov{\mathcal{Y}}$ of such a $\mathcal{Y}$ is the usual ``scheme theoretic" quotient $W/S$ and, furthermore, this $W/S$ is an affine toric variety.  This proves \eqref{Sumi1}.  More precisely, let (following the notation of Lemma~\ref{lem:Sumihiro} and its proof) $L$ (resp.\ $N$) be the cocharacter lattice of the torus $T'$ for $W$ (resp.\ $T$), and let $\sigma$ be the cone in $L_{\RR}$ corresponding to $W$.  As in Lemma~\ref{lem:Sumihiro}, the exact sequence \eqref{toriexactsequence} correponds to a finite index inclusion of lattices $L \into N$ (coming from the dual of \eqref{toriexactsequence2}) and $W/S$ is nothing but the affine toric variety (with torus $T$) corresponding to the cone $\sigma$, viewed as a cone in $N_{\RR}=L_{\RR}$ with respect to the lattice $N$.  The map $[W/S] = \mathcal{Y} \to \ov{\mathcal{Y}} = W/S$ has the property in \eqref{Sumi2} because the cones for the (classical) fan whose realization is $W/S$ (i.e.\ the faces of $\sigma$) are the same as the cones for the KM fan whose realization is $[W/S]$ (only the lattice data differ).  The fan for $\ov{\mathcal{X}}$ has as cones the faces of the various $\sigma$'s as we run over any set of $\mathcal{Y}$'s as above covering $\mathcal{X}$ (this is because the corresponding coarse moduli spaces $\ov{\mathcal{Y}}$ cover $\ov{\mathcal{X}}$).  Therefore \eqref{Sumi2} holds for $\mathcal{X} \to \ov{\mathcal{X}}$ because it holds for each $\mathcal{Y} \to \ov{\mathcal{Y}}$.  \end{proof}

\begin{const} \label{const:main} Let $\mathcal{X}$ be a toric DM stack (Definition~\ref{defn:toricstack}) with torus $T$.  We construct a lattice KM fan $F=F( \mathcal{X} )$ from $\mathcal{X}$ as follows:  Let $N$ be the cocharacter lattice of $T$.  By Corollary~\ref{cor:Sumihiro}, the coarse moduli space $\ov{\mathcal{X}}$ of $\mathcal{X}$ is a toric variety with torus $T$ (this is not \emph{a priori} obvious---at least to us).  By the ``classical" theory of fans and toric varieties, $X$ is the realization of a unique classical fan $\ov{F}$ with lattice $N$.  We will define $F$ by ``lifting" $\ov{F}$ to a KM fan.  We must define lattice data $F_\sigma \subseteq N_{\sigma}$ for the cones $\sigma \in \ov{F}$, or, equivalently, submonoids $P_{\sigma} \subseteq N \cap \sigma$, satisfying certain properties (cf.\ Definition~\ref{defn:KMfan2}).  We define $P_{\sigma}$ to be the submonoid of $N \cap \sigma$ consisting of those $n \in N \cap \sigma$ for which the corresponding cocharacter $n : \GG_m \to T$ extends (necessarily uniquely, up to unique isomorphism) to a map of stacks $\tilde{n} : \AA^1 \to \mathcal{X}$.  \end{const}

Some remarks on Construction~\ref{const:main} are in order.  First:  For any $n \in N \cap \sigma$, the map $n : \GG_m \to T$ does extend to a map $\tilde{n} : \AA^1 \to \ov{\mathcal{X}}$ since $n$ is in the support of $\ov{F}$, but in general this lift won't factor through the coarse moduli space map.  Second: It is not obvious from the construction that the $P_{\sigma}$ satisfy the conditions of Definition~\ref{defn:KMfan2}.  It is possible (though somewhat difficult) to verify this directly (much as in the classical case).  For example, the fact that some positive integer multiple of any $n \in N \cap \sigma$ will be in $P_{\sigma}$ (which ensures that $F_\sigma := P_\sigma^{\rm gp}$ is finite index in $N_{\sigma}$)  follows from the valuative criterion for properness of the coarse moduli space map.  In any case, we shall establish the required conditions during the coarse of proving Theorem~\ref{thm:main} below.

\begin{thm} \label{thm:main} The algebraic realization functor $X$, viewed as a functor from the category of lattice KM fans to the category of toric DM stacks (Definition~\ref{defn:toricstack}), is an equivalence of categories.  Any toric DM stack $\mathcal{X}$ is the algebraic realization of the lattice KM fan $F=F(\mathcal{X})$ constructed in Construction~\ref{const:main}. \end{thm}

\begin{proof} First we check that the functor $X$ is faithful.  Suppose $f : F \to F'$ is the map of lattice KM fans determined by a map of lattices $f : N \to N'$.  Then we can recover $f : N \to N'$ from $X(f)$ as the map on cocharacter groups induced by the map of tori $T \to T'$ which is part of the algebraic realization $X(f)$.  

Next we check that $X$ is full.  Suppose $g : X(F) \to X(F')$ is a map between algebraic realizations equivariant with respect to a map $g_T : T \to T'$ of the corresponding tori.  Let $f : N \to N'$ be the map of cocharacter lattices induced by $g_T$.  We claim that $f$ determines a map of KM fans $f : F \to F'$.  Assuming this claim, we must have $X(f)=g$ since $X(f)$ and $g$ agree on the tori, these tori are dense, and $X(F)$ and $X(F')$ are separated (Proposition~\ref{prop:realization}).  Thus the claim establishes the fullness of our functor.  To prove the claim, we first need to check that $f_{\RR} : N_{\RR} \to N'_{\RR}$ takes each cone $\sigma \in F$ into some cone $\tau \in F'$.  We could check this in essentially the same way one does it for classical fans, or just reduce to the classical case by noting that $g$ induces an equivariant map $\ov{g}$ between coarse moduli spaces, which, by Proposition~\ref{prop:coarsemodulispace}, is a map between realizations of the coarse fans (Example~\ref{example:coarsefan}) underlying $F$ and $F'$.  The map $\ov{g}$ must agree with the map $f : \ov{F} \to \ov{F}'$ determined by $f : N \to N'$ (in particular the former map must be well-defined) because the proposition is known for classical fans, hence we obtain the desired $\tau$ because $F$ and $\ov{F}$ (resp.\ $F'$ and $\ov{F}'$) have the same cones.  We also need to check that $f$ takes $F_\sigma$ into $F'_{\tau}$.  Since $F_\sigma$ (resp.\ $F'_{\tau}$) is the groupification of $P_\sigma := F_\sigma \cap \sigma$ (resp.\ $P_\tau := F'_{\tau} \cap \tau$) and $f$ is a map of groups and we already know that $f$ takes cones into cones, we reduce to showing that $f : N \to N'$ takes $\Supp F$ into $\Supp F'$.  This follows from the interpretation of the fine support in Proposition~\ref{prop:support} since, if $n : \GG_m \to T$ completes to $\tilde{n} : \AA^1 \to X(F)$, then $f(n) : \GG_m \to T'$ completes to $g \tilde{n} : \AA^1 \to X(F')$. 

In remains to prove the final assertion of the theorem.  We first prove this in the special case where $\mathcal{X} = [W /S]$, with $W$ and \bne{tseq} & 0 \to S \to T' \to T \to 0 \ene as in Lemma~\ref{lem:Sumihiro}.  Since $S$ is finite, the cocharacter lattice $L$ of $T'$ has finite index in the cocharacter lattice $N$ of $T$ and \eqref{tseq} is $\GG( \slot )$ of \bne{otherseq} & 0 \to N^\lor \to L^\lor \to \EE(N/L) \to 0 \ene (cf.\ the proof of Lemma~\ref{lem:Sumihiro}, where the same notation is used).  Here the coarse moduli space $\ov{\mathcal{X}} = W/S$ is an affine toric variety with torus $T$, so $\ov{\mathcal{X}} = X(\ov{F})$ where $\ov{F} = (N,[ \sigma], \{ N \cap \Span \tau \})$ for some cone $\sigma \subseteq N_{\RR}$.  Then by construction of the realization, $\mathcal{X} = [W/S]$ is the realization of the KM fan $F=(N, [ \sigma ], \{ F_{\tau} \})$, where the lattice datum $F_{\sigma}$ is defined by $F_{\sigma} = L \cap \sigma$ (so that $L \subseteq N$ is a lift of $F_\sigma$) and the lattice data for faces $\tau \leq \sigma$ are defined by $$ F_\tau := F_{\sigma} \cap \Span \tau = L \cap \Span \tau,$$ as must be the case to satisfy the compatibility condition.  Proposition~\ref{prop:support} ensures that for any face $\tau \leq \sigma$ any $n \in N \cap \tau$, the corresponding map $\GG_m \to \mathcal{X} = X(F)$ extends to a map of stacks $\AA^1 \to \mathcal{X} = X(F)$ iff $n \in F_{\tau}$, so the monoid $P_{\tau}$ defined in Construction~\ref{const:main} is nothing but $F_{\tau} \cap \tau$ and therefore the KM fan $F(\mathcal{X})$ defined in Construction~\ref{const:main} is indeed the KM fan $F$ whose realization $X(F)$ is $\mathcal{X}$.

Now for the general case, let $\ov{F}$ be the unique classical fan in $N$ with $X(\ov{F}) = \ov{\mathcal{X}}$, as in Construction~\ref{const:main}.  Fix some cone $\sigma \in \ov{F}$, so that $Y_{\sigma} := X(\sigma,\ov{F})$ is a $T$ invariant affine open subspace of $X(\ov{F}) = \ov{\mathcal{X}}$.  By Corollary~\ref{cor:Sumihiro}, the preimage $\mathcal{Y}_\sigma$ of $Y_\sigma$ in $\mathcal{X}$ is of the form discussed in the special case above.  For any $n \in N \cap \sigma$, the corresponding cocharacter $n : \GG_m \to T$ completes to a map $\tilde{n} : \AA^1 \to Y_\sigma \subseteq \ov{\mathcal{X}}$, so if that cocharacter $n$ completes to a map $\AA^1 \to \mathcal{X}$, then it must complete to a map $\AA^1 \to \mathcal{Y}_\sigma$.  Therefore the monoid $P_{\sigma}$ defined in Construction~\ref{const:main} for $\mathcal{X}$ is the same as the corresponding monoid constructed for $\mathcal{Y}_\sigma$.  Since this is true for every such $\sigma$, the special case treated above ensures that the $P_{\sigma}$ do indeed define a KM fan $F$ lifting $\ov{F}$.  Furthermore, the special case treated above also ensures that there is an isomorphism of toric DM stacks $X(\sigma,F) \to \mathcal{Y}_\sigma$ (over $X(\ov{F}) = \ov{\mathcal{X}}$).  Since these maps are ``compatible with passing to faces," the definition of $X(F)$ as the colimit of the $X(\sigma,F)$ furnishes a map of toric stacks $X(F) \to \mathcal{X}$ over $X(\ov{F}) = \ov{\mathcal{X}}$, which must be an isomorphism because it is an isomorphism over each $Y_{\sigma}$ and the $Y_{\sigma}$ form an open cover of $X(\ov{F}) = \ov{\mathcal{X}}$. \end{proof}

\subsection{Representability} \label{section:representability}  In this section our goal will be to ``combinatorially" characterize the KM fans (or, more generally, the maps of KM fans) with representable realization.  Of course this will depend on which realization we have in mind.  For example, for the differential realization, \emph{everything} is representable:

\begin{prop} \label{prop:Yrealizationrepresentable} The differential (resp.\ log differential) realization $Y(F)$ of any KM fan $F$ is (representable by) a differentiable (resp.\ log differential) space, hence the differential (resp.\ log differential) realization $Y(f)$ of any morphism of KM fans $f : F \to F'$ ``is" a morphism of differentiable (resp.\ log differential) spaces. \end{prop}

\begin{proof}  It is a basic fact about differential spaces that the differential realization $Y(G)$ of $\GG(G)$, $G$ a finite group, is the trivial group.  (This is because, for a differentiable space $Z$, the positive function $1 \in \O_Z^{>0}(Z)$ has no non-trivial positive roots.)  As for the log realization of any group, it also has trivial log structure.  Since the construction of the realization of $F$ only involves quotients by actions of such groups (\S\ref{section:KMfanrealization}), the result follows. \end{proof}

We next address the question of representability for the algebraic realization \emph{without log structures}.  Here we have a perfect combinatorial characterization of representable morphisms:

\begin{thm} \label{thm:representability}  For a map of KM fans $f : F \to F'$, the following are equivalent: \begin{enumerate} \item \label{representable1} The algebraic realization $X(f)$ of $f$ is representable (by algebraic spaces) as a map of schemes, forgetting log structures. \item \label{representable2} For any $\sigma \in F$, if we let $\tau \in F'$ be the smallest cone of $F'$ containing $f_{\RR}(\sigma)$, then the map of abelian groups $N/F_\sigma \to N'/F'_\tau$ is injective on torsion subgroups.  \item \label{representable3} For any $\sigma \in F$ and \emph{any} cone $\rho \in F'$ with $f_{\RR}(\sigma) \subseteq \rho$, the map of abelian groups $N/F_\sigma \to N'/F'_\rho$ is injective on torsion subgroups. \end{enumerate} \end{thm}

\begin{proof} Suppose \eqref{representable1} holds.  Pick an algebraic closure $\ov{k}$ of our base field $k$.  Given $\sigma \in F$, pick a $\ov{k}$ point $x \in X(N/F_\sigma,0,0)(\ov{k})$ lying in the stratum $X(N/F_\sigma,0,0)$ of $X(F)$ corresponding to $\sigma$ in Proposition~\ref{prop:stratification}.  As discussed in Remark~\ref{rem:stratification2}, the realization $X(f)$ takes this stratum of $X(F)$ into the stratum $X(N'/F'_\tau,0,0)$ of $X(F')$ with $\tau$ as in \eqref{representable2} via the realization of the map of zero fans associated to the map of groups $N/F_\sigma \to N'/F'_\tau$.  Since $X(f)$ is representable, the map induced by $X(f)$ from the isotropy group of $x$ in $X(F)$ to the isotropy group of $X(f)(x)$ in $X(F')$ must be injective.  Since the aforementioned ``strata" are locally closed, this map of isotropy groups is the same as the one induced by $X(N/F_\sigma,0,0) \to X(N/F_\sigma,0,0)$.  From Proposition~\ref{prop:isotropygroups}, we see that our map of isotropy groups \emph{is} the map on torsion subgroups induced by $N/F_\sigma \to N'/F'_\tau$.  This proves that \eqref{representable2} holds. 

Suppose \eqref{representable2} holds.  Let $\sigma$, $\tau$, $\rho$ be as in \eqref{representable2} and \eqref{representable3}.  Then $N/F_\sigma \to N'/F'_\rho$ factors as the composition $$ N/F_\sigma \to N'/F'_\tau \to N'/F'_\rho . $$  The first map is injective on torsion subgroups by \eqref{representable2} and the second map is injective on torsion subgroups by Lemma~\ref{lem:representability}, hence \eqref{representable3} holds. 

Suppose \eqref{representable3} holds and let us establish \eqref{representable1}.  One way to do this would be to argue as in the first paragraph of this proof that \eqref{representable3} ensures that $X(f)$ induces an injective map on isotropy groups for any geometric point of $X(F)$, then we quote some general result to the effect that any map between DM stacks over a field of characteristic zero with this property is representable.  Although one can make this argument work, it has the disadvantage of using many features specific to the \emph{algebraic} realization.  A better argument goes as follows:

Representability of $X(f)$ can be checked Zariski locally on $X(f)$, so it is enough to prove that for $\sigma$ and $\rho$ as in \eqref{representable3}, the map of stacks \bne{repmap} X(\sigma,F) & \to & X(\rho,F') \ene obtained by restricting $X(f)$ to the open substack $X(\sigma,F)$ is representable.  By Lemma~\ref{lem:liftingsexist} we can choose a lifting $L' \subseteq N'$ of $F'_\rho$.  By Lemma~\ref{lem:liftings}, the injectivity assumption \eqref{representable3} ensures that $L := f^{-1}(L') \subseteq N$ is a lifting of $F_\sigma$.  The map $N/L \to N'/L'$ is clearly injective, hence $\EE(N'/L') \to \EE(N/L)$ is surjective, hence $\GG(\EE(N/L)) \to \GG(\EE(N'/L'))$ is a closed embedding (in particular, a monomorphism).  By construction of the realization of $f$ (\S\ref{section:KMfanrealization}), the map \eqref{repmap} is computed as the map on global quotients \bne{repmap2} [ X(S_\sigma(L)) / \GG(\EE(N/L)) ] & \to & [ X(S_\rho(L')) / \GG(\EE(N/L)) ] \ene induced by the map of affine toric varieties $X(S_\sigma(L)) \to X(S_\rho(L'))$ which is equivariant with respect to our \emph{injective} map of group objects.  It is a straightforward exercise to see that this injectivity ensures that \eqref{repmap2} is representable.  \end{proof}

\begin{cor} The algebraic realization of a KM fan $F$ (disregarding log structures) is representable by a sheaf iff $F$ is a classical fan, in which case its algebraic realization is a toric variety. \end{cor}

\begin{proof} By applying the theorem with $F'$ the zero fan, we see that $X(F)$ is representable by a sheaf iff each $N/F_\sigma$ is torsion free, which is equivalent to $F_\sigma = N_\sigma$, which is equivalent to $F$ being classical (cf.\ Example~\ref{example:classicalfan}). \end{proof}

Now we turn to the case of the \emph{log} algebraic and fan realizations.  In the interest of brevity, we will only establish some partial results in this direction, roughly to the effect that these realizations are ``often" representable \emph{by sheaves} (Corollary~\ref{cor:representablebyasheaf}).

\begin{lem} \label{lem:effectiveaction} Let $P$ be an arbitrary monoid, $G$ an arbitrary abelian group, $a : P^{\rm gp} \to G$ a surjective group homomorphism.  For any integral monoid $M$,\footnote{One only needs $M$ to be \emph{quasi-integral} in the sense that $M^*$ acts effectively on $M$.} the action of $G(M) := \Hom_{\Mon}(G,M) = \Hom_{\Ab}(G,M^*)$ on $P(M) := \Hom_{\Mon}(P,M)$ induced by $a$ is effective. \end{lem}

\begin{proof}  Writing everything multiplicatively, the action in question is defined by $(u \cdot f)(p) := u(a(p))f(p)$, where $f \in P(M)$, $u \in G(M)$, $p \in P$, $u \cdot f \in P(M)$.  Suppose $u$ fixes $f$.  Then we have $u(a(p))f(p) = f(p)$ for all $p \in P$, hence $u(a(p)) = 1 \in M^*$ by (quasi-)integrality of $M$ for all $p \in P$.  Since the image of $P \to P^{\rm gp}$ generates $P^{\rm gp}$ and $a$ is surjective, we conclude that $u =1$ as desired. \end{proof}

\begin{prop} \label{prop:sheafquotient} Let $P$ be a fine monoid, $a : P^{\rm gp} \to G$ a map of abelian groups.  Then the following quotient ``stacks" are naturally isomorphic to the corresponding quotient \emph{sheaves}: \begin{enumerate} \item the stack $[ \Spec P / \GG(G) ]$ over the category $\Fans$ of fine fans \item the stack $[ \Spec (P \to k[P]) / \Spec k[G]]$ over the category of fine log schemes $\LogSch$ \item the stack $[ \RR_{\geq 0}(P) / \RR_{>0}(G) ]$ over the category of fine log differentiable spaces. \end{enumerate}  In particular, each such quotient stack is representable by a sheaf. \end{prop}

\begin{rem} It is important in the above proposition that we are considering the various realization functors as functors to \emph{fine} fans, etc. \end{rem}

\begin{proof} In any site $\C$, when a sheaf of groups $G$ acts on a sheaf $X$, the quotient stack $[X/G]$ is defined to be the stackification of the evident groupoid fibration $[X/G]^{\rm pre} \to \C$ whose fiber category over $U \in \C$ is the quotient category $[ X(U) / G(U) ]$.  (When a group $G$ acts on a set $X$, the quotient category $[X/G]$ has $X$ as its set of objects; a morphism from $x$ to $y$ in $[X/G]$ is a $g \in G$ satisfying $g \cdot x = y$.)  If the action of $G(U)$ on $X(U)$ is effective, then the quotient category $[ X(U) / G(U) ]$ is equivalent to the quotient set $X(U) / G(U)$ (thought of as a category with only identity morphisms).  If this is true for all $U \in \C$, then the groupoid fibration $[X/G]^{\rm pre}$ is equivalent to the one associated to the presheaf $U \mapsto X(U)/G(U)$, hence the stackification of the former is equivalent to the sheafification of the latter.

This general discussion applies to any of the settings above, using Lemma~\ref{lem:effectiveaction} to establish the necessary effectivity, to yield the desired result.  The point is that the action of $\GG(G)(U)$ on $\AA(P)(U)$ is identified with the action discussed in that lemma when we take $M := \M_U(U)$, which is integral because we work with fine fans, etc. \end{proof}

\begin{cor} \label{cor:representablebyasheaf}  Suppose $F = (N,F,\{ F_\sigma \})$ is a lattice KM fan.  Then the fan and log algebraic realizations of $F$ are all representable by sheaves (respectively, in the CZE topology and in the strict \'etale topology). \end{cor}

\begin{proof} The question is local, so it suffices to prove that $\AA(\sigma)$ (Definition~\ref{defn:AAsigma}) is represented by a sheaf for each $\sigma \in F$.  If we choose a lift $L$ of $F_\sigma$, then, by construction of $\AA(\sigma)$ (Definition~\ref{defn:AAsigma}) we have $\AA(\sigma) = [ \AA(S_\sigma(L)) / \GG(\EE(N/L)) ]$ where the action in question is induced by the natural group homomorphism $S_\sigma(L)^{\rm gp} = L^\lor \to \EE(N/L)$ appearing in \eqref{fundamentalSES}; the cokernel of this map is $\EE(N) = 0$ since $N$ is a lattice, so we conclude by Proposition~\ref{prop:sheafquotient}. \end{proof}

\begin{rem}  The proof that \eqref{representable3} implies \eqref{representable1} above makes sense for any realization functor $\AA$ (not just the algebraic realization $X$) and shows that \eqref{representable3} implies representability of $\AA(f)$ by $\tau$ sheaves (which are of course ``algebraic spaces" in an appropriate sense).  Of course, one does not have the converse for a general realization functor. \end{rem}

\section{Folding and Unfolding} \label{section:foldingandunfolding}  The goal of this section is to explain the relationship between our theory of KM fans and the ``stacky fans" introduced by Geraschenko and Satriano in \cite{GS1} (which we shall call \emph{GS fans} to avoid confusion), generalizing earlier constructions of Lafforgue, Borisov, Chen, Smith, and many others (see the references in \cite{GS1}).

\subsection{GS fans} \label{section:GSfans}  First we need to recall the basic notions from the Geraschenko-Satriano theory.

\begin{defn} \label{defn:GSfan} (cf.\ \cite[Definition~2.4]{GS1}) A \emph{GS fan} is a pair $(F,\beta)$ where $F$ is a fan in a lattice $L$ and $\beta : L \to N$ is a map of lattices with finite cokernel.  A \emph{morphism} $f=(f_L,f_N)$ from a GS fan $(F,\beta)$ to a GS fan $(F',\beta')$ is a map of lattices $f_L : L \to L'$ defining a map of fans $f_L : F \to F'$, together with a map of lattices $f_N : N \to N'$ making the obvious square commute ($f_N \beta = \beta' f_L$).  GS fans form a category, denoted $\GSFans$. \end{defn}

As in \cite[Definition~2.5]{GS1} we define the \emph{realization} of a GS fan $(F,\beta)$ to be the stack-theoretic quotient \be \AA(F,\beta) & := & [ \AA(F) / \GG(\Cok (\beta^\lor)) ]. \ee  Here $\AA(F)$ is the usual realization of the fan $F$ in the lattice $L$, and the group object $\GG(\Cok (\beta^\lor))$ acts on $\AA(F)$ through the map of group objects $\GG(\Cok (\beta^\lor)) \to \GG(L^\lor)$ appearing in the exact sequence of group objects \bne{GSSES1} & \xym{  0 \ar[r] & \GG(\Cok (\beta^\lor)) \ar[r] & \GG(L^\lor) \ar[r] & \GG(N^\lor) \ar[r] & 0 } \ene obtained by realizing the exact sequence of abelian groups \bne{GSSES} & \xym{  0 \ar[r] & N^\lor \ar[r]^-{\beta^\lor} & L^\lor \ar[r] & \Cok (\beta^\lor) \ar[r] & 0 } \ene and the usual action of the torus $\GG(L^\lor)$ on $\AA(F)$.  From the aforementioned action and the exact sequence \eqref{GSSES1}, we obtain an action of $\GG(N^\lor)$ on $\AA(F,\beta)$ as in Lemma~\ref{lem:quotients}.  Since $\AA(F)$ has an open dense $\GG(L^\lor)$ orbit isomorphic to $\GG(L^\lor)$, we see from exactness of \eqref{GSSES1} that $\AA(F,\beta)$ has an open dense $\GG(N^\lor)$ orbit isomorphic to \be \GG(N^\lor) & = & \GG(L^\lor) / \GG(\Cok (\beta^\lor)).\ee  The realization of a morphism of GS fans $(f_L,f_N) : (F,\beta) \to (F',\beta')$ is defined to be the map on quotients $\AA(F,\beta) \to \AA(F',\beta')$ induced by the $\GG( \Cok(\beta^\lor)) \to \GG( \Cok( (\beta')^\lor))$ equivariant map $\AA(f_L) : \AA(F) \to \AA(F')$.

In analogy with Definition~\ref{defn:KMstack}, we make the following

\begin{defn} \label{defn:GSstack} A \emph{GS stack} is a stack isomorphic to the realization of a GS fan. \end{defn}

\begin{example} \label{example:GSfans} Let $F=(N,F,\{ F_\sigma \})$ be a lattice KM-fan.  Fix a cone $\sigma \in F$ and a lifting $L \subseteq N$ of $F_\sigma$.  Then the stack \be \AA(\sigma) = [ \AA(S_\sigma(L)) / \GG(\EE(N/L)) ] \ee defined in \S\ref{section:KMfanrealization} is a GS stack.  The choice of $L$ gives us a choice of GS fan with realization $\AA(\sigma)$.  Indeed, the assumption that $N$ is a lattice ensures that $\EE(N/L)$ is the cokernel of the dual of the inclusion $L \into N$, so $\AA(\sigma)$ will be the realization of the GS fan $([\sigma],\beta)$, where $[\sigma]$ is the fan consisting of all faces of the cone $\sigma$ in $L_{\RR}$ ($=N_{\RR}$), and $\beta$ is the inclusion $L \into N$. \end{example}

\subsection{Folding} \label{section:folding}  In this section we will describe a recipe (``folding") for constructing a KM fan $\FF(F,\beta)$ from a GS fan $(F,\beta)$ satisfying certain properties.  The GS fans which can be ``folded" in this manner admit a simple combinatorial description which can also be interpreted naturally in terms of the algebraic realization.

\begin{thm} \label{thm:folding} The algebraic realization $X(F,\beta)$ of a GS fan $(F,\beta : L \to N)$ is a separated Deligne-Mumford (DM) stack iff the following two conditions hold: \begin{enumerate} \item \label{DM} $\beta_{\RR} : \sigma \to \beta_{\RR}(\sigma)$ is bijective for every cone $\sigma \in F$. \item \label{separated} The interiors of the cones $\beta_{\RR}(\sigma)$ and $\beta_{\RR}(\tau)$ are disjoint for any two distinct cones $\sigma, \tau \in F$. \end{enumerate}  When these two conditions hold, the set of cones $\FF(F,\beta) := \{ \beta_{\RR}(\sigma) : \sigma \in F \}$ is a fan in $N$, and the groups $\beta(L_\sigma)$ define lattice data for these cones, making $\FF(F,\beta) := (N, \FF(F,\beta), \{ \beta(L_\sigma) \})$ a lattice KM fan, called the \emph{folding} of $(F,\beta)$.  The map $\beta : L \to N$ defines a tame map of KM fans $\beta : F \to \FF(F,\beta)$.  The realization (in any category of spaces) of the KM fan $\FF(F,\beta)$ coincides with the realization of the GS fan $(F,\beta)$. \end{thm}

\begin{proof}  Note that $\beta_{\RR} : L_{\RR} \to N_{\RR}$ is surjective because $\beta$ is assumed to have finite cokernel, so $\beta_{\RR} : \sigma \to \beta_{\RR}(\sigma)$ is bijective iff it is injective.  Set $L_\sigma := L \cap \Span \sigma$, as usual.  Then $\beta_{\RR} : \sigma \to \beta_{\RR}(\sigma)$ is bijective iff $\Ker( \beta | L_\sigma : L_\sigma \to N) =0$ iff the rank of $\Cok( (\beta | L_\sigma)^\lor )$ is zero (i.e.\ that group is finite).  We now have to translate this algebraic condition into a condition about stabilizers.  Let $x_\sigma \in X(F)(k)$ be the distinguished point of $X(\sigma) \subseteq X(F)$ as in \cite[\S2.1]{F}.  Then by a simple exercise in the classical theory of toric varieties, the stabilizer subgroup of $x_\sigma$ for the action of the torus $T = \GG(L^\lor)$ on the toric variety $X(F)$ is the sub-torus $\GG(L_\sigma^\lor)$.  The isotropy group of the image of $x_\sigma$ in $X(F,\beta) = [X(F) / \GG(\Cok(\beta^\lor)) ]$ is hence the intersection of the subgroups $\GG(\Cok (\beta^\lor))$ and $\GG(L_\sigma^\lor)$ of $T$, which is $\GG(A)$, where $A$ is the quotient of $L^\lor$ by the subgroup generated by $N^\lor$ and the kernel $L(\sigma)^\lor$ of $L^\lor \to L_\sigma^\lor$.  Equivalently, thinking of $\GG(A)$ as a subgroup of $\GG(L_\sigma^\lor)$, we see that $A$ can be described as the cokernel of $(\beta | L_\sigma)^\lor$.  Putting this all together, we see that $\beta_{\RR} : \sigma \to \beta_{\RR}(\sigma)$ is bijective iff the isotropy of the image of $x_\sigma$ in $X(F,\beta)$ is finite.  If $X(F,\beta)$ is DM, then the isotropy groups of all its points must be finite, so each $\beta_{\RR} : \sigma \to \beta_{\RR}(\sigma)$ must be bijective.  

Denote the interior of a cone $\sigma$ by $\sigma^\circ$.  Suppose $\beta_{\RR}(\sigma)^\circ \cap \beta_{\RR}(\sigma)^\circ \neq \emptyset$ for distinct cones $\sigma, \tau \in F$.  Then we can find $l_\sigma \in L \cap \sigma^\circ$ and $l_\tau \in L \cap \tau^\circ$ with $\beta(l_\sigma) = \beta( l_\tau)$.  This corresponds to two maps of tori $l_\sigma, l_\tau : \GG_m \rightrightarrows \GG(L^\lor)$ which become equal when composed with $\GG(\beta^\lor) : \GG(L^\lor) \to \GG(N^\lor)$, hence these two maps differ by the action of $\GG(\Cok(\beta^\lor)) = \Ker( \GG(\beta^\lor) )$, hence the two compositions $$ l_\sigma, l_\tau : \GG_m \rightrightarrows \GG(L^\lor) \subseteq X(F) \to X(F,\beta) = [X(F) / \GG(\Cok(\beta^\lor)) ] $$ are ``equal" (more precisely: isomorphic)---call this common map $l : \GG_m \to X(F,\beta)$.  By the classical theory of toric varieties, the maps $l_\sigma, l_\tau : \GG_m \rightrightarrows \GG(L^\lor) \subseteq X(F)$ extend uniquely to $\ov{l}_\sigma, \ov{l}_\tau : \AA^1 \to X(F)$ taking $0 = \AA^1 \setminus \GG_m$ to $x_\sigma, x_\tau$, respectively.  Composing with $X(F) \to X(F,\beta)$ we get two extensions of $l$ which cannot be ``equal" because the images of $x_\sigma$ and $x_\tau$ in $X(F,\beta)$ cannot be ``equal" since $x_\sigma$ and $x_\tau$ are not in the same $T$-orbit, let alone the same orbit under $\GG(\Cok(\beta^\lor)) \subseteq T$.  This shows that $X(F,\beta)$ cannot be separated.

We have proved that if $X(F,\beta)$ is separated DM, then \eqref{DM} and \eqref{separated} hold.  The converse will follow from the final statements of the theorem in light of Proposition~\ref{prop:realization}, so it remains only to prove those final statements.  Let $(F,\beta)$ be a GS fan satisfying \eqref{DM} and \eqref{separated}.  

We first prove that $\FF(F,\beta)$ is a fan.  Each $\beta_{\RR}(\sigma)$ is a sharp cone and $\FF(F,\beta)$ is closed under passing to faces because of \eqref{DM} and the fact that $F$ is a fan.  To see that any two cones of $\FF(F,\beta)$ intersect in a face of each, it suffices to prove that \be \beta_{\RR}(\sigma \cap \tau) & = & \beta_{\RR}(\sigma) \cap \beta_{\RR}(\tau) \ee for any $\sigma, \tau \in F$.  The containment $\subseteq$ is clear.  For the opposite containment, suppose $x \in \beta_{\RR}(\sigma) \cap \beta_{\RR}(\tau)$.  Then there exist $v \in \sigma,w \in \tau$ with $\beta_{\RR}(v)=\beta_{\RR}(w)=x$.  We claim that $v=w$, which will prove that $x \in \beta_{\RR}(\sigma \cap \tau)$.  To prove this claim, let $\sigma'$ (resp.\ $\tau'$) denote the faces of $\sigma$ (resp.\ $\tau$) containing $v$ (resp.\ $w$) in its interior.  We must have $\sigma' = \tau'$, otherwise $x$ will be in the intersection of $\beta_{\RR}(\sigma')^\circ = \beta_{\RR}((\sigma')^\circ)$ and $\beta_{\RR}(\tau')^\circ = \beta_{\RR}((\tau')^\circ)$, in violation of \eqref{separated}.  But now $v,w$ are in $\sigma' = \tau'$, which is mapped bijectively onto its image via $\beta_{\RR}$ (by \eqref{DM}), so we must have $v=w$.

The fact that $\beta(L_\sigma)$ has finite index in $N_{\beta_{\RR}(\sigma)}$ follows from the fact that $\Cok \beta$ is finite.  The compatibility condition \be \beta(L_\tau) & = & \beta(L_\sigma) \cap \Span \beta_{\RR}(\tau) \ee for $\tau \leq \sigma \in F$ is immediate from the definitions.

This proves that $\FF(F,\beta)$ is a KM fan.  It is clear from the definition of $\FF(F,\beta)$ that $\beta : L \to N$ defines a tame map of KM fans $\beta : F \to \FF(F,\beta)$, so by Theorem~\ref{thm:quotients}\eqref{quotients3}, the realization $\AA(\beta)$ of $\beta$ makes $\AA(F)$ a $\GG(\Cok( \beta^\lor))$ torsor over $\FF(F,\beta)$, hence $$ \AA(F,\beta) = [ \AA(F) / \GG(\Cok( \beta^\lor)) ] = \AA(\FF(F,\beta)) $$ as desired. \end{proof}

\subsection{Unfolding} \label{section:unfolding}  In this section, we will describe a construction, called \emph{unfolding}, which will be useful in the study of tame maps of KM fans (Definition~\ref{defn:tame}).  This will ultimately allow us to describe the relationship between GS fans and KM fans.

For a KM fan $F=(N,F,\{ F_\sigma \})$ (Definition~\ref{defn:KMfan}) we set \be \latticedatalimit(F) & := & \dirlim \{ F_\sigma : \sigma \in F \} . \ee  The colimit defining $\latticedatalimit(F)$ is the colimit in the category of abelian groups, taken over the cones $\sigma \in F$ with maps given by the sublattice inclusions $F_\tau \into F_\sigma$ when $\tau \leq \sigma \in F$.  If there is no chance of confusion we write $\latticedatalimit$ instead of $\latticedatalimit(F)$.  The lattice $\ov{\latticedatalimit} = \latticedatalimit / \latticedatalimit_{\rm tor}$ may be viewed as the colimit of the $F_\sigma$ \emph{taken in the category of lattices}.  We write $i_\sigma : F_{\sigma} \to \latticedatalimit$ for the structure map to the direct limit and $\ov{i}_\sigma : F_\sigma \to \ov{\latticedatalimit}$ for its composition with the projection $\latticedatalimit \to \ov{\latticedatalimit}$.  Let $\beta : \latticedatalimit \to N$ be the map induced by the inclusions $F_\sigma \into N$ using the universal property of the direct limit, so that $\beta i_\sigma$ is the inclusion $F_\sigma \into N$ for each $\sigma \in F$.  Since each such inclusion is injective, so are the maps $i_\sigma$.  Since each $F_\sigma$ is also torsion-free, the maps $\ov{i}_\sigma : F_\sigma \to \ov{\latticedatalimit}$ are also injective.

When $N$ is a lattice, $\beta$ factors through the quotient map $\latticedatalimit \to \ov{\latticedatalimit}$ via a map $\ov{\beta} : \ov{\latticedatalimit} \to N = \ov{N}$ such that $\ov{\beta} \ov{i}_\sigma$ is the inclusion $F_\sigma \into N$.

Let us calculate $\latticedatalimit(F)$ in some simple examples.

\begin{example} If $F$ has a unique maximal cone $\sigma$, then $\latticedatalimit(F) = \ov{\latticedatalimit}(F) = F_\sigma$ because $F_\sigma$ is terminal in the direct limit system defining $\latticedatalimit(F)$. \end{example}

\begin{example} \label{example:tildeLnotalattice} In this example, we will see that $\latticedatalimit(F)$ may fail to be a lattice, even when $F$ is a complete classical fan in $N=\ZZ^2$.  Indeed, consider the complete classical fan $F$ in $\RR^2 = (\ZZ^2)_\RR$ with rays $\rho_1 = \RR_{\geq 0}(1, 2)$, $\rho_2 = \RR_{\geq 0}(1 , -2)$, and $\rho_3 = \RR_{\geq 0}(-1,0)$.  The fan $F$ has three two-dimensional cones $\sigma_1 = \RR_{\geq 0} \langle \rho_1,\rho_2 \rangle$,  $\sigma_2 = \RR_{\geq 0} \langle \rho_2,\rho_3 \rangle$, and $\sigma_3 = \RR_{\geq 0} \langle \rho_1, \rho_3 \rangle$.  The group $\latticedatalimit(F)$ is the quotient of $$F_{\sigma_1} \oplus F_{\sigma_2} \oplus F_{\sigma_3} = N_{\sigma_1} \oplus N_{\sigma_2} \oplus N_{\sigma_3} = \ZZ^2 \oplus \ZZ^2 \oplus \ZZ^2$$ be the subgroup $K \subseteq (\ZZ^2)^3$ consisting of elements of the form $$( d \rho_1 + e \rho_2 , - e \rho_2 + f \rho_3 , -d \rho_1 - f \rho_3 ) \quad \quad (d,e,f \in \ZZ).$$  Taking $d=e=f=1$, we see that $$2 ( (1 , 0) , (-1 , 1) , (0 , -1) ) = ( ( 2 , 0) , ( -2 , 2) , (0 , -2) )$$ is in $K$, even though $( (1 , 0) , (-1 , 1) , (0 , -1) )$ is not in $K$ (because its first entry $(1 , 0)$ is not a $\ZZ$-linear combination of $\rho_1$ and $\rho_2$), hence $\latticedatalimit(F) = (\ZZ^2)^3 / K$ has non-trivial $2$-torsion. \end{example}

\begin{example} \label{example:isigmanotsaturated} In this variant of the previous example, we will describe a complete KM fan $F$ in $N = \ZZ^2$ for which $\latticedatalimit = \ov{\latticedatalimit} \cong \ZZ^3$ with a cone $\sigma \in F$ such that the structure map $i_\sigma = \ov{i}_\sigma : F_\sigma \to \latticedatalimit$ is not saturated.  The fan underlying our KM fan $F$ will be the same as the fan in the previous example.  We take the lattice data $F_{\sigma_2}$ and $F_{\sigma_3}$ to be the classical data $F_{\sigma_2} = F_{\sigma_3} = N = \ZZ^2$ as in the previous example, but this time we take $F_{\sigma_1}$ to be the lattice (freely) generated by $\rho_1=(1,2)$ and $\rho_2=(1,-2)$.  The rays of our KM fan $F$ and their lattice data are as in the previous example.  To construct $\latticedatalimit(F)$, we start with the free abelian group $F_{\sigma_1} \oplus F_{\sigma_2} \oplus F_{\sigma_3}$ freely generated by $$i_{\sigma_1}(\rho_1), i_{\sigma_1}(\rho_2), i_{\sigma_2}(e_1), i_{\sigma_2}(e_2), i_{\sigma_3}(e_1), i_{\sigma_3}(e_2)$$ and we impose the three relations \bne{isigmanotsaturatedrelations} i_{\sigma_1}(\rho_1) & = & i_{\sigma_3}(e_1)+2 i_{\sigma_3}(e_2) \\ \nonumber i_{\sigma_1}(\rho_2) & = & i_{\sigma_2}(e_1)-2 i_{\sigma_2}(e_2) \\ \nonumber i_{\sigma_2}(e_1) & = & i_{\sigma_3}(e_1), \ene one for each of the rays $\rho_1$, $\rho_2$, $\rho_3$.  Clearly the resulting group $\latticedatalimit$ is freely generated by $i_{\sigma_2}(e_1)$, $i_{\sigma_2}(e_2)$, and $i_{\sigma_3}(e_2)$.  Using that ordered basis for $\latticedatalimit \cong \ZZ^3$ and the ordered basis $\rho_1$, $\rho_2$ for $F_{\sigma_1} \cong \ZZ^2$, the structure map $i_{\sigma_1} : F_{\sigma_1} \to \latticedatalimit$ is given by the matrix $$ \bp 1 & 1 \\ 0 & -2 \\ 2 & 0 \ep. $$  This map is not saturated because $(0,1,1)$ (transpose) is not in its image, even though $2(0,1,1)$ \emph{is} in its image. \end{example}

\begin{rem} \label{rem:functorialityoflatticedatalimit}  The group $\latticedatalimit(F)$ is covariantly functorial in $F$.  If $U$ and $V$ are sub-KM-fans of $F$ with union $F$, then the pushout diagram of KM fans on the left below is taken by the functor $\latticedatalimit$ to a pushout diagram of groups as on the right below.  \bne{Lpushout} \xym{ U \cap V \ar[r] \ar[d] & U \ar[d] \\ V \ar[r] & F } & \quad \quad & \xym{ \latticedatalimit(U \cap V) \ar[r] \ar[d] & \latticedatalimit(U) \ar[d] \\ \latticedatalimit(V) \ar[r] & \latticedatalimit(F) } \ene  One can use these pushout diagrams to compute $\latticedatalimit(F)$ ``inductively" by taking a maximal cone $\sigma \in F$, then taking $U := F \setminus \{ \sigma \}$ and $V := [ \sigma ]$ (the sub-KM-fan of $F$ whose cones are all the subcones of $\sigma$).  \end{rem}

\begin{rem} \label{rem:latticedatalimitsemitame} If $f : F \to F'$ is a semi-tame (Definition~\ref{defn:tame}) map of KM fans, then the induced map of abelian groups $\latticedatalimit(F) \to \latticedatalimit(F')$ is an isomorphism because $f$ even yields an isomorphism from the limit system defining $\latticedatalimit(F)$ to the one defining $\latticedatalimit(F')$. \end{rem}

Now we return to our construction.  For a cone $\sigma \in F$, we let $i(\sigma) \subseteq \latticedatalimit_{\RR}$ be the image of the cone $\sigma \subseteq F_\sigma \otimes \RR = N_\sigma \otimes \RR$ under the inclusion of $\RR$ vector spaces $i_\sigma \otimes \RR : F_\sigma \otimes \RR \into \latticedatalimit_{\RR}$ (which we will often abusively denote $i$).  Of course $\latticedatalimit_{\RR} = \ov{\latticedatalimit}_{\RR}$, so we can also view $i(\sigma)$ as a cone in $\ov{\latticedatalimit}_{\RR}$.

Our main result about unfolding is that the unfolding of a KM fan is ``the universal torsor" over it (compare Cox's construction \cite{Cox}):

\begin{thm} \label{thm:unfolding}  Let $F=(N,F,\{ F_\sigma \})$ be a KM fan, $\latticedatalimit := \latticedatalimit(F)$.  Then the triple \be \UU(F) & := & ( \latticedatalimit, \{ i(\sigma) : \sigma \in F \}, \{ i_{\sigma}(F_\sigma) : \sigma \in F \} ) \ee is a KM fan, called the \emph{unfolding} of $F$, and the map of groups $\beta : \latticedatalimit \to N$ defines a semi-tame (Definition~\ref{defn:tame}) map of KM fans $\beta : \UU(F) \to F$ which is initial among all semi-tame maps of KM fans with codomain $F$.  That is, any diagram of KM fans $$ \xym{ \UU(F) \ar@{.>}[rr] \ar[rd]_-\beta & & F' \ar[ld]^-{f} \\ & F } $$ with $f$ semi-tame can be uniquely completed as indicated.  If we assume that $F$ is atoroidal (Definition~\ref{defn:nondegenerate}) and that the kernel of $\beta : \latticedatalimit \to N$ is torsion-free, then $\beta : \UU(F) \to F$ is tame.  The construction $F \mapsto \UU(F)$ is functorial in $F$ (as is the map $\beta : \UU(F) \to F$) and takes semi-tame maps to isomorphisms.  The rigidification $\UU^{\rm rig}(F)$ of $\UU(F)$ in the sense of Example~\ref{example:rigidification} is given by the triple \be \UU^{\rm rig}(F) & = & ( \ov{\latticedatalimit}, \{ i(\sigma) : \sigma \in F \}, \{ \ov{i}_{\sigma}(F_\sigma) : \sigma \in F \} ) \ee and called the \emph{rigidified unfolding} of $F$.  Assume now that $N$ is a lattice (i.e.\ $F$ is a lattice KM fan).  Then the map of groups $\ov{\beta} : \ov{\latticedatalimit} \to N$ defines a semi-tame map of KM fans $\ov{\beta} : \UU^{\rm rig}(F) \to F$ (via the universal property of the rigidification) which is initial in the category of semi-tame maps of lattice KM fans $f : F' \to F$.  This map of KM fans is tame whenever $F$ is atoroidal ($\Ker \ov{\beta}$ is always torsion-free since $\ov{\latticedatalimit}$ is a lattice).  If $F$ is a classical fan, then $\UU^{\rm rig}(F)$ is also a classical fan. \end{thm}

\begin{rem} The fan $F$ of Example~\ref{example:tildeLnotalattice} is complete, hence atoroidal, but for that $F$ the kernel of $\beta : \latticedatalimit \to N$ has torsion, since $N = \ZZ^2$ is torsion-free and $\latticedatalimit$ has torsion.  That $F$ is classical, but $\UU(F)$ is not classical. \end{rem}

\begin{proof}  To see that $\UU(F)$ is a KM fan, the main issue is to show that the set of cones $\{ i(\sigma) \}$ is a fan---the other details will be left to the reader.  Since the maps $i_\sigma : F_{\sigma} \to \latticedatalimit$ are injective, so are the maps $i_\sigma \otimes \RR$.  Since the set of cones $F$ is closed under passage to faces, it follows that $\{ i(\sigma) \}$ is closed under passage to faces.  Now suppose $\sigma_1, \sigma_2 \in F$.  We want to show that $i(\sigma_1) \cap i(\sigma_2)$ is a face of, say, $i(\sigma_1)$.  Since $F$ is a fan, we know $\tau := \sigma_1 \cap \sigma_2$ is a face of $\sigma_1$.  We first claim that \be i(\sigma_1) \cap i(\sigma_2) = i(\tau). \ee  The containment $\supseteq$ is clear.  On the other hand, if $x = i(z_1)=i(z_2)$ in the colimit $\latticedatalimit$, with $z_1 \in \sigma_1$, $z_2 \in \sigma_2$, then by definition of $\latticedatalimit$, there must be a cone $\sigma \in F$ containing $\sigma_1,\sigma_2$ as faces, such that $z_1 = z_2$ in $\sigma$. But $z=z_1 = z_2 \in \sigma$ means $z$ belongs to the face $\tau = \sigma_1 \cap \sigma_2 < \sigma$ and $x = i(z)$.  To see that $i(\tau)$ is a face of $i(\sigma_1)$, suppose $i(z) = i(u)+i(v)$ for $z \in \tau$, $u,v \in \sigma_1$. Since $i = i_\sigma \otimes \RR$ is injective, we have $z=u+v$, hence $u,v \in \tau$ since $\tau < \sigma_1$.

Since $i_\sigma \beta$ is the inclusion $F_\sigma \into N$, it is clear that $\beta_{\RR}$ takes $i(\sigma)$ bijectively onto $\sigma$ and $\beta$ takes the lattice datum $i_{\sigma}(F_\sigma)$ for the cone $i(\sigma)$ of $\UU(F)$ bijectively onto the lattice datum $F_\sigma$ for the cone $\sigma$ in the fan $F$.  This proves that $\beta : \latticedatalimit \to N$ defines a semi-tame map of KM fans $\beta : \UU(F) \to F$.

The second assertion amounts to showing that the cokernel of $\beta : \latticedatalimit \to N$ is finite whenever $F$ is atoroidal.  To see this, note that, by construction of $\latticedatalimit$, we have a commutative square of FGA groups as below. $$ \xym{ \bigoplus_{\sigma \in F} F_\sigma \ar[d]_{\sum_{\sigma \in F} i_\sigma} \ar[r] & \bigoplus_{\sigma \in F} N_{\sigma} \ar[d] \\ \latticedatalimit \ar[r]^-{\beta} & N. } $$  The right vertical arrow has finite cokernel since $F$ is atoroidal.   The top horizontal arrow has finite cokernel by definition of a KM fan, so $\beta$ has finite cokernel. 

The functoriality of $F \mapsto \UU(F)$ is clear, and the fact that $\UU$ takes semi-tame maps to isomorphisms is an elementary elaboration on Remark~\ref{rem:latticedatalimitsemitame}.

To see that $\UU^{\rm rig}(F)$ can be described as in the theorem, we need only note that $\latticedatalimit \to \ov{\latticedatalimit}$ takes $i_\sigma(F_\sigma) \subseteq \latticedatalimit$ bijectively onto $\ov{i}_\sigma(F_\sigma) \subseteq \ov{\latticedatalimit}$ for each $\sigma \in F$.

The ``universal property" of $\UU(F)$ is straightforward to establish; the analogous universal property of $\UU^{\rm rig}(F)$ is obtained by combining this universal property with the universal property of rigidification (Example~\ref{example:rigidification}).

To see that $\UU^{\rm rig}(F)$ is classical when $F$ is classical, we need to show that $\ov{i}_\sigma(F_\sigma)$ is saturated in $\ov{\latticedatalimit}$ for each cone $\sigma \in F$ (cf.\ Example~\ref{example:classicalfan}).  Suppose $n \lambda = \ov{i}_\sigma(x)$ in $\ov{\latticedatalimit}$ for some $\lambda \in \ov{\latticedatalimit}$, $x \in F_\sigma$, $n \in \{ 1,2,\dots \}$.  Applying $\ov{\beta} : \ov{\latticedatalimit} \to N$ to this equality and recalling that $\ov{\beta} \ov{i}_\sigma$ is the inclusion $F_\sigma \subseteq N$, we find that $n \beta(\lambda) = x$ in $N$.  Since $F$ is classical, $F_\sigma$ is saturated in $N$, so $\beta(\lambda)=y$ for a unique $y \in F_\sigma \subseteq N$.  This implies that $\ov{i}_\sigma(y) - \lambda \in \ov{\latticedatalimit}$ is killed by multiplication by $n$ (and is in the kernel of $\ov{\beta}$), hence $\ov{i}_\sigma(y) = \lambda$ since $\ov{\latticedatalimit}$ is torsion-free.  \end{proof}

\begin{cor} \label{cor:unfolding} Let $F=(N,F,\{ F_\sigma \})$ be an atoroidal lattice KM fan.  Then the following are equivalent: \begin{enumerate} \item \label{rigidifiedunfoldingsaturated} The rigidified unfolding $\UU^{\rm rig}(F)$ is a classical fan. \item \label{saturatedincs} For every cone $\sigma \in F$, the inclusion $\ov{i}_{\sigma} : F_\sigma \to \ov{\latticedatalimit}(F)$ is saturated.  \item \label{classicalcover} There is a classical fan $F'$ and a tame map of KM fans $f : F' \to F$. \item \label{isaGSstack} The algebraic realization $X(F)$ of $F$ (with its torus action) is a GS stack (Definition~\ref{defn:GSstack}).  \end{enumerate}  If these equivalent conditions hold, then $\ov{\beta} : \ov{\latticedatalimit}(F) \to N$ defines a tame map of KM fans $\UU^{\rm rig}(F) \to F$, and the pair $(\UU^{\rm rig}(F), \ov{\beta} : \ov{\latticedatalimit}(F) \to N)$ is a GS fan whose realization (in any category of spaces) coincides with the realization of $F$. \end{cor}

\begin{proof}  The equivalence of \eqref{rigidifiedunfoldingsaturated} and \eqref{saturatedincs} is just the characterization of classical fans among KM fans discussed in Example~\ref{example:classicalfan}.  If this common condition holds, then \eqref{classicalcover} holds because we can take $f : F' \to F$ to be $\ov{\beta} : \UU^{\rm rig}(F) \to F$ by the theorem.  Conversely, if we have $f:F' \to F$ as in \eqref{classicalcover}, then, by the theorem, $\UU^{\rm rig}(f) : \UU^{\rm rig}(F') \to \UU^{\rm rig}(F)$ is an isomorphism and $\UU^{\rm rig}(F')$ is classical, so $\UU^{\rm rig}(F)$ is classical.

Suppose the first three equivalent conditions hold.  Since $\UU^{\rm rig}(F)$ is a classical fan and $\ov{\latticedatalimit}(F) \to N$ is tame by the theorem, the indicated pair is certainly a GS fan.  By definition, the realization of this GS fan is the stack-theoretic quotient of $\AA( \UU^{\rm rig}(F) )$ by the action of $\GG( \Cok (\ov{\beta}^\lor) )$.  But the theorem says that $\UU^{\rm rig}(F) \to F$ is tame, so Theorem~\ref{thm:quotients}\eqref{quotients3} says that the action of $\GG( \Cok (\ov{\beta}^\lor) )$ on $\AA( \UU^{\rm rig}(F) )$ makes $\AA( \UU^{\rm rig}(F) ) \to \AA(F)$ a $\GG( \Cok (\ov{\beta}^\lor) )$-torsor, so the stack-theoretic quotient in question is just $\AA(F)$, as desired.  In particular, if we take $\AA=X$ to be the algebraic realization, we see that \eqref{isaGSstack} holds.

Now suppose \eqref{isaGSstack} holds.  Then there is a classical fan $F'$ such that the toric variety $X(F')$ admits a map (of algebraic stacks) $g : X(F') \to X(F)$ which is a torsor under the action of some subtorus $S$ of the torus $T'$ for $X(F')$ and such that the torus $T$ for $X(F)$ is the quotient $T'/S$ (and its action on $X(F)$ is the obvious one).  Since the algebraic realization of lattice KM fans is fully faithful (Theorem~\ref{thm:main}), $g=X(f)$ is the algebraic realization of some map of KM fans $f : F \to F'$.  Since $X(f)$ is an $S$-torsor, the ``converse" part of Theorem~\ref{thm:quotients2} implies that $f : F' \to F$ is tame, hence \eqref{classicalcover} holds. \end{proof}

\begin{cor} \label{cor:noteveryKMfanisGS} There is a complete KM fan $F=(N,F,\{ F_\sigma \})$ with $N = \ZZ^2$ whose algebraic realization $X(F)$ is not a GS stack.  This realization $X(F)$ \emph{is}, however, a normal, complete (proper over $\Spec$ of the base field $k$) DM stack with a torus action with a dense orbit. \end{cor}

\begin{proof} The KM fan $F$ described in Example~\ref{example:isigmanotsaturated} will do, in light of the previous corollary. \end{proof}

\begin{thm} \label{thm:KMandGSfans}  Let $\C$ be the full subcategory of the category of GS fans consisting of those GS fans $(F,\beta)$ satisfying \eqref{DM} and \eqref{separated} in Theorem~\ref{thm:folding}.  The folding construction $(F,\beta) \mapsto \FF(F,\beta)$ of Theorem~\ref{thm:folding} defines a functor $\FF$ from $\C$ to the category of KM fans.  The functor $\FF$ is not faithful, but any parallel $\C$-morphisms with the same image under $\FF$ induce the same map on any realizations.  For any atoroidal lattice KM fan $F$ satisfying the conditions of Corollary~\ref{cor:unfolding}, we have \bne{FFequality} \FF( \UU^{\rm rig}(F), \ov{\beta}) & = & F. \ene  The essential image of $\FF$ consists of those lattice KM fans satisfying condition \eqref{classicalcover} in Corollary~\ref{cor:unfolding}. \end{thm}

\begin{proof}  Suppose $f=(f_L,f_N) : (F,\beta) \to (F',\beta')$ is a morphism (Definition~\ref{defn:GSfan}) between GS fans satisfying \eqref{DM} and \eqref{separated} in Theorem~\ref{thm:folding}.  Using the commutativity condition $f_N \beta = \beta f_L$ and the fact that $f_L : L \to L'$ defines a map of fans $f_L : F \to F'$, we see immediately from the definition of folding that $f_N : N \to N'$ defines a map of (lattice) KM fans $\FF(f) : \FF(F,\beta) \to \FF(F',\beta')$, so our functor $\FF$ is given on $\C$-morphisms $(f_L,f_N)$ simply by forgetting $f_L$.  It is easy to make examples of parallel $\C$-morphisms with different $f_L$ (take $N=N'=0$ and $F$, $F'$ the ``zero" fans, for example), so $\FF$ is clearly not faithful.  To see that ``the" (rather: ``any fixed") realization $\AA(f_L,f_N)$ of $(f_L,f_N)$ depends only on $f_N$, first note that we have a commutative diagram of KM fans $$ \xym{ F \ar[r]^-{f_L} \ar[d]_\beta & F' \ar[d]^{\beta'} \\ \FF(F,\beta) \ar[r]^-{f_N} & \FF(F',\beta') } $$ whose realization gives rise to a commutative diagram of stacks $$ \xym{ \AA(F) \ar[r]^-{\AA(f_L)} \ar[d] & \AA(F') \ar[d] \\ \AA(F,\beta') \ar[r]^-{\AA(f_L,f_N)} \ar[d] &  \AA(F',\beta') \ar[d] \\ \AA(\FF(F,\beta)) \ar[r]^-{\AA(f_N)} & \AA(\FF(F',\beta')). } $$  The bottom vertical arrows have nothing to do with $f_L$ (or $f_N$) and are isomorphisms by Theorem~\ref{thm:folding}, so $\AA(f_L,f_N)$ is independent of $f_L$.  (We are merely explaining that the coincidence of realizations in Theorem~\ref{thm:folding} is functorial.)

Given an atoroidal lattice KM fan $F$ satisfying the conditions of Corollary~\ref{cor:unfolding}, we know from that corollary that $\ov{\beta} : \ov{\latticedatalimit}(F) \to N$ defines a tame map of KM fans $\ov{\beta} : \UU^{\rm rig}(F) \to F$, so \eqref{FFequality} is clear from the construction of $\FF$ in Theorem~\ref{thm:folding}.

For every $(F,\beta) \in \C$, we have a tame map of lattice KM fans $\beta : F \to \FF(F,\beta)$ (Theorem~\ref{thm:folding}), so clearly any KM fan in the essential image of $\FF$ is a lattice KM fan satisfying \eqref{classicalcover} in Corollary~\ref{cor:unfolding}.  Conversely, suppose $F=(N,F,\{ F_\sigma \})$ is a lattice KM fan satisfying \eqref{classicalcover} in Corollary~\ref{cor:unfolding}.  Choose a splitting $F=G \times (B,0,0)$ with $G$ atoroidal as in Example~\ref{example:atoroidalsplitting}, with $N = A \oplus B$ the corresponding splitting of lattices.  Note that $G$ also satisfies \eqref{classicalcover} in Corollary~\ref{cor:unfolding} (for example because the projection $N \to A$ defines a tame map of KM fans $F \to G$).  Using the result from the previous paragraph we see easily that \be \FF( \UU^{\rm rig}(G) \times (B,0,0) , \ov{\beta} \times \Id_B ) & = & F, \ee so $F$ is in the essential image of $\FF$.   \end{proof}

\begin{rem} \label{rem:KMandGSfans}  The isomorphism \eqref{FFequality} is functorial in $F$, so for atoroidal lattice KM fans $F$, $F'$ satisfying the conditions of Corollary~\ref{cor:unfolding}, folding and rigidified unfolding yield inverse bijections \be \Hom_{\GSFans}(( \UU^{\rm rig}(F), \ov{\beta}),( \UU^{\rm rig}(F'), \ov{\beta}')) & = & \Hom_{\KMFans}(F,F'). \ee  We suspect that, in general, the folding functor $\FF$ of Theorem~\ref{thm:KMandGSfans} is not full, but we have not sought an example of this behaviour. \end{rem}

\section{Technical Appendix} \label{section:appendix}

The following result is probably obvious but we have provided a detailed proof just to remove any doubt about its validity in the generality in which we use it.

\begin{lem} \label{lem:quotients}  Let $\C$ be a site, $X$ a stack over $\C$ equipped with an action of a sheaf of (not necessarily abelian) groups $B$, \bne{SESsheaves} & 1 \to A \to B \to C \to 1 \ene a short exact sequence of sheaves of groups on $\C$.  There is an action of $C$ on the stack-theoretic quotient $[X/A]$ (natural in both $X$ and the sequence \eqref{SESsheaves}) making the natural map $[X/A] \to [X/B]$ a $C$ torsor.  In particular, $[[X/A]/C] = [X/B]$.  Furthermore, when $X=B$ with the $B$ action given by left multiplication, the $C$ action on $[B/A]=C$ is also given by left multiplication. \end{lem}

\begin{rem} \label{rem:quotientslemma} In the text, this lemma is used only in the following situation: $\C = \Esp$ is a category of spaces with an admissible topology $\tau$ as in \S\ref{section:realizationofmonoids}.  The exact sequence \eqref{SESsheaves} will be the exact sequence \eqref{SequenceC} in Proposition~\ref{prop:tautorsors} obtained by applying the functor $\GG$ of \S\ref{section:realizationofmonoids} to an exact sequence of FGA groups \eqref{SequenceA} as in Proposition~\ref{prop:tautorsors}.  The stack $X$ will be the realization $\AA(F)$ of some KM fan $F$, and the action of $\GG(B)$ on $X$ will always be the realization of an action of the (abstract) group fan $\Spec B$ on $F$. \end{rem}

The rest of this appendix is devoted to explaining and proving this lemma.  Not surprisingly, most of the ``proof" amounts to carefully defining the approproate notions.  However, there will be one or two genuinely nice ideas / constructions.  We will try not to get too bogged down in details, but, on the other hand, you really have to get your hands dirty a little bit to give anything resembling a proof of something like this.

Roughly speaking, the proof proceeds by reducing to the case where $X$ is the terminal object, then reducing to the case where $\C$ is the category of sets, in which case the lemma becomes:

\begin{lem} \label{lem:quotients2} For any exact sequence of groups \bne{SESgroups} & 1 \to A \to B \to C \to 1, \ene there is a natural action of $C$ on the classifying groupoid $\BB A$ (the category with one object $\bullet_A$ with $\Hom( \bullet_A, \bullet_A) = \Aut( \bullet_A, \bullet_A) =A$) making the functor $\BB A \to \BB B$ a $C$ torsor (meaning this functor is ``the" categorical quotient of $\BB A$ by the $C$ action). \end{lem}

The first reduction is easy:  Suppose the result is known when $X$ is the terminal object.  Then we have a $C$ action on the classifying stack $\BB A$ making the map of classifying stacks $\BB A \to \BB B$ a $C$ torsor.  Now, in the general setup of the lemma, the $A$ action on $X$ is induced from the $B$ action, thus one sees readily that the square \bne{2cart} & \xym{ [X/A] \ar[r] \ar[d] & \BB A \ar[d] \\ [X/B] \ar[r] & \BB B } \ene is $2$-cartesian (the horizontal arrows are the $X$-bundles corresponding to the actions), so the general result follows by ``pulling back" the special case where $X$ is the terminal object.

For the ``furthermore," suppose $X=B$ (with $B$ acting on $X$ by left multiplication).  Then $[X/B]=[B/B]$ is the terminal object $\bullet$, so the ``furthermore" is equivalent to showing that there is a $2$-cartesian diagram of stacks \bne{2cartfurthermore} & \xym{ C \ar[r] \ar[d] & \BB A \ar[d] \\ \bullet \ar[r] & \BB B } \ene where $C \to \BB A$ is $C$ equivariant, with $C$ acting on itself by left multiplication and on $\BB A$ in as in the special case where $X$ is the terminal object.

The second reduction is just general nonsense:  Suppose we know Lemma~\ref{lem:quotients2}.  Let $C^{\rm pre}$ be the quotient of $A \to B$, calculated in presheaves on $\C$.  Similarly, let $\BB^{\rm pre}A$ and $\BB^{\rm pre}B$ be the classifying groupoid fibrations, calculated in the $2$-category of groupoid fibrations (``prestacks") over $\C$.  Then $C$ (resp.\ $\BB A$, $\BB B$) is obtained from $C^{\rm pre}$ (resp.\ $\BB^{\rm pre}A$, $\BB^{\rm pre}B$) by sheafification (resp.\ stackification).  For any $Y \in \C$, the sequence of groups \bne{SESsheavesonY} & 1 \to A(Y) \to B(Y) \to C^{\rm pre}(Y) \to 1 \ene is exact (and natural in $Y$) and the fiber categories $(\BB^{\rm pre}A)(Y)$ and $(\BB^{\rm pre}B)(Y)$ are naturally equivalent to the classifying groupoids $\BB( A(Y) )$ and $\BB( B(Y) )$ of the groups $A(Y)$ and $B(Y)$ (respectively), so Lemma~\ref{lem:quotients2} yields an action of $C^{\rm pre}(Y)$ on $\BB( A(Y) )$ making $\BB( A(Y) ) \to \BB( B(Y) )$ a $C^{\rm pre}(Y)$ torsor.  Since this is all natural in $Y$, we conclude that $C^{\rm pre}$ acts on $\BB^{\rm pre} A$ making $\BB^{\rm pre}(A) \to \BB^{\rm pre}(B)$ a $C^{\rm pre}$ torsor.  Since sheafification and stackification ``preserve inverse limits," they take group actions to group actions and torsors to torsors, thus $\BB A \to \BB B$ is a $C$-torsor.

The same general nonsense reduces the ``furthermore" to the case where $\C=\Sets$.

It remains to prove Lemma~\ref{lem:quotients2}.  This is perhaps the only thing that isn't a matter of general nonsense.  We now really need to address an issue which we've been unjustly avoiding thusfar:  We have to explain what it means for a group $G$ to \emph{act} on a category $\C$.  (This also yields a notion of what it means for a sheaf of groups to act on a stack---by passing to fiber categories---and this is the notion of ``action" we have in mind in Lemmas~\ref{lem:quotients} and \ref{lem:quotients2}.)  The ``right" answer is well-known---let's recall it for the reader's convenience.

The naive notion is that of a \emph{strong action} of $G$ on $\C$, which is simply a homomorphism $g \mapsto \phi_g$ from $G$ to the \emph{automorphism} group of $\C$.  This means that the functors $\phi_g : \C \to \C$ are required to satisfy an actual \emph{equality} of functors $\phi_g \phi_h = \phi_{gh}$ for $g,h \in G$, as well as an actual \emph{equality} $\phi_1 = \Id$.  The notion of ``strong action" is generally unsatisfactory because such an action can't be ``transfered" along an \emph{equivalences} of categories:  If we had an equivalence $F : \C \to \D$ with ``inverse" $H : \D \to \C$, then we'd like to be able to ``transfer" a ``$G$ action" $g \mapsto \phi_g$ on $\C$ to a $G$ action $g \mapsto \psi_g$ on $\D$ by setting $\psi_g := F \phi_g H$.  But we can't transfer a strong action $g \mapsto \phi_g$ to a \emph{strong} action in this manner because $\psi_1 = F \phi_1 H = FH$ is no longer the identity and, for $g,h \in G$, the two functors $\psi_{gh} = F \phi_{gh} H = F \phi_g \phi_h H$ and $\psi_g \psi_h = F \phi_g H F \phi_g H$ are not (generally) \emph{equal}.  However, we do have an evident choice of \emph{isomorphisms} (of functors) $\alpha : \psi_1 \to \Id$ and $\alpha(g,h) : \psi_{gh} \to \psi_g \psi_h$ coming from the isomphisms of functors $HF \to \Id$ and $FH \to \Id$ which were part of our equivalence.  The isomorphisms $\alpha$ and $\alpha(g,h)$ satisfy many obvious ``compatibilities" which the interested reader can enumerate.  (For example, for any $g,h,k \in G$, the two evident ways of using the $\alpha( \slot, \slot)$'s to make an isomorphism of functors $\psi_{ghk} \to \psi_g \psi_h \psi_k$ coincide.)  The ``correct" notion of a $G$ action on $\C$ (which \emph{can} be transfered along an equivalence) is that of a \emph{weak action}, which consists of the assignment $g \mapsto \psi_g$ and isomorphisms $\alpha$, $\alpha(g,h)$ as before, satisfying the ``obvious compatibilities."  Luckily, we won't need to worry too much about exactly what the ``obvious compatibilities" are because the weak actions we consider will always arise by transfering a strong action along an equivalence in the manner discussed above.  (It is in general not so clear that one can always ``realize homotopy group actions" in this way.  This is a line of inquiry in topology.)  

\begin{example} \label{example:quotients1} Note that a weak $G$ action on $\C$ induces an action of $G$ on the ``set" of isomorphism classes in $\C$.  If a set $X$ is regarded as a category with only the identity maps as morphisms, then a weak action of $G$ on $X$ is the same as a strong action, which is just the usual notion of $G$ acting on $X$. \end{example}

We also need to know what it means for a functor $Q : \C \to \D$ to be ``the" \emph{(categorical) quotient} of a (weak) $G$-action $g \mapsto \psi_g$ on $\C$.  First of all, it means that $Q$ should be \emph{weakly} $G$ \emph{invariant} in the sense that there are isomorphisms of functors $\epsilon_g : Q \to Q \psi_g$ enjoying ``some obvious compatibilities" with the $\alpha$ and $\alpha(g,h)$ isomorphisms which are part of the weak $G$ action structure.  (These $\epsilon$'s are considered part of the structure of the weakly $G$ invariant functor, so that $Q$ should really refer to the pair $Q=(Q,\epsilon)$.)  We shall only need to consider these ``obvious compatibilities" in the case of a \emph{strong} action, in which case they become:  \bne{1compatibility} \epsilon_1 & = & \Id \\ \label{ghcompatibility} \epsilon_{gh} & = & ( \epsilon_g * \phi_h) \epsilon_h \ene for all $g,h \in G$.  Finally, $Q$ should be $2$-\emph{initial} among such weakly $G$ invariant functors in the sense that for any other weakly $G$ invariant functor $Q'=(Q' : \C \to \D',\epsilon')$, there is a functor $K : \D \to \D'$ and an isomorphism of functors $\zeta : KQ \to Q'$ compatible with $\epsilon$ and $\epsilon'$ in the sense that the diagram \bne{epsiloncompatibility} & \xym{ KQ \ar[r]^-{\zeta} \ar[d]_{K * \epsilon_g} & Q'  \ar[d]^{\epsilon'_g}  \\ KQ \psi_g \ar[r]_-{\zeta * \psi_g} & Q' \psi_g } \ene of isomorphisms (of functors $\C \to \D'$) commutes.  The pair $(K,\zeta)$ is also required to be ``unique" in the sense that if $(K',\zeta')$ is another such pair, then there is a unique isomorphism $\theta : K \to K'$ such that \bne{zetacompatibility} \zeta & = & \zeta'( \theta * Q). \ene

\begin{rem} \label{rem:Gequivariant} If $G$ acts (weakly) on $\C$, $H$ acts (weakly) on $\D$, and we have a group homomorphism $f : G \to H$, then a \emph{(weakly)} $f$ \emph{equivariant} functor $Q : \C \to \D$ is defined in a manner very similar to the way we defined ``weakly $G$ invariant" above (the special case $H=1$).  Again, we don't really need to worry about the exact details, since the only weakly equivariant diagram we consider will be equivalent to an explicit strongly equivariant diagram of strong actions. \end{rem}

Like the notion of a (weak) group action on a category, the notion of a categorical quotient is made to be ``invariant under equivalence" in various senses.  We will make implicit use of the following simple properties of this definition, which are readily verified: \begin{enumerate} \item If $Q=(Q: \C \to \D, \epsilon)$ is ``the" categorical quotient of a weak $G$ action on $\C$ and $Q' : \C \to \D$ is a functor for which there is an isomorphism of functors $\eta : Q \to Q'$, then $Q'$ (with $\epsilon_g' := (\eta * \psi_g)\epsilon_g \eta$) is also ``the" categorical quotient of $\C$ by $G$.  \item Suppose $\C'$ is a category equipped with a (weak) $G$ action and $F : \C' \to \C$ is an equivalence of categories.  Regard $\C$ as a category with (weak) $G$ action by transfering the $G$ action on $\C'$ along $F$, as discussed above.  Then a pair $Q=(Q,\epsilon)$ is ``the" categorical quotient of $\C$ iff $(QF, \epsilon * F)$ is ``the" categorical quotient of $\C'$. \end{enumerate}

\begin{example} \label{example:quotients2}  Suppose a group $G$ acts on a set $X$ and we regard this as a (strong) action of $G$ on the category associated to $X$ as in Example~\ref{example:quotients1}, so that $\phi_g(x)=gx$.  Consider the category $[X/G]$ whose objects are elements of $X$ and where a morphism $g : x \to x'$ is an element $g \in G$ such that $gx=x'$.  Composition is multiplication in $G$: \be (g' : x' \to x'') \circ (g : x \to x') & := & (g'g : x \to x''). \ee  Let $F : X \to [X/G]$ be the unique functor which is the identity on objects.  Define an isomorphism of functors $\epsilon_g : F \to F\phi_g$ by $\epsilon_g(x) := (g : Fx=x \to F\phi_g x=gx)$.  One can check that $(F,\epsilon)$ is ``the" categorical quotient of $X$ by $G$ in the above sense.  In particular, the classifying groupoid $\BB G$ is the categorical quotient of the punctual category by the trivial $G$ action.  This illustrates the philosophy that set-theoretic quotients are obtained by identifying elements while category-theoretic quotients are obtained by adding more isomorphisms---which has the effect of identifying isomorphism classes of objects. \end{example}

We now prove Lemma~\ref{lem:quotients2}.  The key construction is the category (groupoid, in fact) $\BB' A$ which we shall define (naturally) from the exact sequence \eqref{SESgroups} as follows:  Denote the quotient map $B \to C$ by $b \mapsto \ov{b}$.  This quotient map and multiplication in $C$ define an action of $B$ on $C$ and we basically define $\BB' A$ to be the category $[C/B]$ as in Example~\ref{example:quotients2}, except we shall use the \emph{right} action of $B$ on $C$ because we'll also want to introduce a \emph{left} action later.  Explicitly, an object of $\BB' A$ is an element of $C$.  A morphism $b : c \to c'$ in $\BB' A$ is an element $b \in B$ such that $c =c' \ov{b}$ in $C$.  As in Example~\ref{example:quotients2}, composition is multiplication in $B$: \be (b' : c' \to c'') \circ (b : c \to c') & := & (b' b : c \to c''). \ee  

We have a \emph{strong} action $g \mapsto \phi_g$ of $C$ on the category $\BB' A$ (also natural in \eqref{SESgroups}) defined as follows:  On objects $\phi_g$ is given by $\phi_g(c) := gc$ and on morphisms $\phi_g$ is given by $\phi_g(b : c \to c') := (b : gc \to gc')$.  (There are unfortunately two morphisms called ``$b$" here, but they have potentially different domains and codomains.)

Notice that $\BB' A$ is equivalent to the usual classifying groupoid $\BB A$ because any two objects of $\BB' A$ are isomorphic (since $B \to C$ is surjective) and the automorphism group of $1 \in \BB' A$ is clearly $A = \{ b \in B : \ov{b}=1 \}$.  Let us explicitly write down such an equivalence.  We have a functor $H : \BB A \to \BB' A$ given by $\bullet_A \mapsto 1$ on objects and by $a \mapsto (a : 1 \to 1)$ on maps.  Pick a set-theoretic section $s : C \to B$ of the surjection $B \to C$.  Keep in mind that $s$ will not generally be a group homomorphism!  Then we can define a functor $F : \BB' A \to \BB A$ by mapping every object to $\bullet_A$ (as we must) and by taking a morphism $b : c \to c'$ in $\BB' A$ to the morphism $s(c') b s(c)^{-1} \in A$.  The composition $FH : \BB A \to \BB A$ necessarily takes $\bullet_A$ to itself and is given by $a \mapsto s(1) a s(1)^{-1}$ on maps.  We have an isomorphism of functors $\eta : \Id \to FH$ defined by $\eta(\bullet_A) := s(1) \in A$ and an isomorphism of functors $\theta : \Id \to HF$ defined by $\theta(c) := (s(c) : c \to 1)$.  We can transfer the strong action of $C$ on $\BB' A$ along this equivalence in the manner discussed above to get a (weak) $C$ action on $\BB A$.  The reader may wish to explicitly write down the $\alpha$ maps for this weak action, though we shall not need to do so.  We claim this is the desired torsorial action of Lemma~\ref{lem:quotients2}.

We have a functor $Q : \BB' A \to \BB B$ given by mapping all objects to $\bullet_B$ and on morphisms by $Q(b : c \to c') := b$.  This functor is \emph{strongly} $C$ invariant in the sense that we have an actual equality of functors $Q = Q \phi_g$ for each $g \in C$---in particular it can and will be viewed as a weakly $C$ invariant functor with $\epsilon_c = \Id$ for all $c \in C$.  Denote the composition of $F : \BB' A \to \BB A$ and the natural functor $\BB A \to \BB B$ by $\ov{F} : \BB' A \to \BB B$.  Then one checks easily that setting $\eta(c) := s(c)$ defines an isomorphism of functors $\eta : Q \to \ov{F}$.

Since the notion of categorical quotient is ``invariant under equivalences" in the senses made explicit above, we can check that $\BB A \to \BB B$ is ``the" categorical quotient of our $C$ action by instead checking that $Q : \BB'A \to \BB B$ is ``the" categorical quotient of our strong $C$ action.  (We could also equivalently check that $\ov{F} : \BB' A \to \BB B$ is ``the" quotient by $C$.  Just as one can explicitly write down the $\alpha$ maps for the weak $C$ action on $\BB A$, one could also explicitly write down the $\epsilon$ maps for all of our quotient functors.  For example, the $\epsilon$ maps for $\ov{F} : \BB' A \to \BB B$ will be given by $\epsilon_g(c) = s(gc)s(c)^{-1}$.)

It remains only to show that $Q=(Q,\epsilon=\Id)$ is $2$-initial among such weakly $C$-invariant maps, so suppose $Q'=(Q' : \BB' A \to \D, \epsilon')$ is another one.  We define a functor $K : \BB B \to \D$ as follows:  On objects, we take $\bullet_B$ to $Q'(1) \in \D$.  On morphisms, we define $K$ by \bne{Kdefn} K(b) & := & Q'(b : \ov{b} \to 1) \epsilon'_{\ov{b}}(1) \ene for $b \in B = \Hom_{\BB B}(\bullet_B,\bullet_B)$.  To see that $K$ is actually a functor, we first note that $K(1) = \Id$ because $\epsilon'_1 = \Id$ (by \eqref{1compatibility} for the weakly $C$ invariant functor $(F',\epsilon')$) and $F'$ preserves identity maps.  To see that $K$ respects composition we compute: \be K(bb') & = & Q'(bb' : \ov{bb'} \to 1) \epsilon'_{\ov{bb'}}(1) \\ & = & Q'( b : \ov{b} \to 1) Q'( b' : \ov{bb'} \to \ov{b} ) \epsilon'_{\ov{bb'}}(1) \\ & = & Q'( b : \ov{b} \to 1) Q'( b' : \ov{bb'} \to \ov{b} ) \epsilon'_{\ov{b}}(\ov{b}') \epsilon'_{\ov{b'}}(1) \\ & = & Q'(b : \ov{b} \to 1) \epsilon'_{\ov{b}}(1) Q'(b' : \ov{b'} \to 1) \epsilon'_{\ov{b'}}(1) \\ & = & K(b)K(b'). \ee  The first equality is the definition \eqref{Kdefn}.  For the second equality we use the fact that \be (bb' : \ov{bb'} \to 1) & = & ( b : \ov{b} \to 1)( b' : \ov{bb'} \to \ov{b}) \ee in $\BB' A$ and the fact that $Q'$ respects composition.  The third equality uses the formula \be \epsilon'_{\ov{b}\ov{b'}} & = & ( \epsilon'_{\ov{b}} * \phi_{\ov{b'}}) \epsilon'_{\ov{b'}} \ee (an instance of \eqref{ghcompatibility}) evaluated at the object $1 \in \BB' A$.  The forth equality uses the naturality of $\epsilon'_{\ov{b}} : Q' \to Q' \phi_{\ov{b}}$ under the $\BB' A$ morphism $b' : \ov{b'} \to 1$.  

We next define an isomorphism $\zeta : KQ \to Q'$ by setting \bne{zetadefinition} \zeta(c) & := & \epsilon'_c(1). \ene  (Note that both sides are $\D$-morphisms from $Q'(1)$ to $Q'(c)$.)  To check the naturality of $\zeta$ on a $\BB' A$ morphism $b : c \to c'$ we need to check that \bne{TS1} Q'(b : c \to c') \epsilon'_c(1) & = & \epsilon'_{c'}(1)Q'(b : \ov{b} \to 1) \epsilon'_{\ov{b}}(1). \ene  To see this, we first use the naturality of $\epsilon'_{c'}$ with respect to the map $b : \ov{b} \to 1$ (whose image under $\phi_{c'}$ is $b : c \to c'$ because $c=c' \ov{b}$) to find that \bne{TS2} \epsilon'_{c'}(1)Q'(b : \ov{b} \to 1) & = & Q'( b : c \to c') \epsilon'_{c'}(\ov{b}), \ene then we evaluate the formula \eqref{ghcompatibility} (for $\epsilon'$, with $g=c'$, $h=\ov{b}$) at $1$ to find \bne{TS3} \epsilon'_{c'}(1) & = & \epsilon'_{c'}(\ov{b}) \epsilon'_{\ov{b}}(1). \ene  Evidently \eqref{TS1} follows from \eqref{TS2} and \eqref{TS3}.

Next we need to check that our $\zeta$ makes \eqref{epsiloncompatibility} commute.  Since our map $Q$ is strongly $C$ invariant, our $\epsilon_g$ is the identity and the commutativity of \eqref{epsiloncompatibility} on an object $c$ of $\BB' A$ is equivalent to the formula \be \epsilon'_g(c) \zeta(c) & = & \zeta(gc). \ee  Looking at the formula \eqref{zetadefinition}, we see that this is equivalent to the formula \be \epsilon'_g(c) \epsilon'_c(1) & = & \epsilon'_{gc}(1), \ee which is an instance of \eqref{ghcompatibility} for $\epsilon'$ (evaluated at $1$).

Finally, suppose $(K' : \BB B \to \D, \zeta' : K'Q \to Q')$ is another pair satisfying the properties we checked above for our pair $(K,\zeta)$.  We want to show that there is a unique isomorphism of functors $\theta : K \to K'$ satisfying \eqref{zetacompatibility}.  Since $\bullet_B$ is the only object of $\BB B$, $\theta$ is uniquely determined by $\theta( \bullet_B )$ and if we want \eqref{zetacompatibility} to hold even on the object $1$ of $\BB' A$, then we have no choice but to define $\theta(\bullet_B)$ by \bne{thetaformula} \theta(\bullet_B)  & := & \zeta'(1)^{-1} \zeta(1)  \\ \nonumber & = & \zeta'(1)^{-1} \epsilon'_1(1) \\ \nonumber & = & \zeta'(1)^{-1} . \ene It remains only to show that \eqref{thetaformula} is actually natural under $\BB B$ morphisms and that \eqref{zetacompatibility} holds.  To check that $\theta$ is natural under a $\BB B$ morphism $b : \bullet_B \to \bullet_B$ we need to show that \be \theta( \bullet_B) K(b)  & = & K'(b) \theta(\bullet_B). \ee  Using the formulas \eqref{thetaformula} and \eqref{Kdefn}, we see that the issue is to show \bne{SHOW1} Q'(b : \ov{b} \to 1) \epsilon'_{\ov{b}}(1) \zeta'(1) & = & \zeta'(1) K'(b). \ene  Commutativity of \eqref{epsiloncompatibility} for $\zeta'$ with $g=\ov{b}$ gives us \bne{SHOW2} \epsilon'_{\ov{b}}(1) \zeta'(1) & = & \zeta'(\ov{b}). \ene  Naturality of $\zeta'$ on $b : \ov{b} \to 1$ gives \bne{SHOW3} \zeta'(1)K'(b) & = & Q'(b : \ov{b} \to 1) \zeta'( \ov{b} ). \ene  Evidently \eqref{SHOW1} follows from \eqref{SHOW2} and \eqref{SHOW3}.  

This completes the proof of Lemma~\ref{lem:quotients2}.

For the ``furthermore" we still need to construct a $2$-cartesian diagram of categories \eqref{2cartfurthermore} with $C \to \BB A$ weakly $C$ equivariant.  Again, since being $2$-cartesian or weakly $C$ equivariant is appropriately invariant under equivalences, we just need to show that the ``obvious" diagram \bne{2cartfurthermoresub} &  \xym{ C \ar[r] \ar[d] & \BB' A \ar[d]^Q \\ \bullet \ar[r] & \BB B } \ene is $2$-cartesian with $C \to \BB' A$ weakly $C$ equivariant.  Here $C$ is the discrete category with set of objects $C$, acted on \emph{strongly} by $C$ via $g \cdot c = gc$.  The functor $C \to \BB'A$ is the unique functor which is the identity on objects.  This functor is clearly \emph{strongly} $C$ equivariant.  The functor $\bullet \to \BB B$ is the only possible one (since $\BB B$ has only one object).  The diagram \eqref{2cartfurthermoresub} commutes ``on the nose" because both compositions $C \to \BB B$ must take any object $c$ of $C$ to the unique object $\bullet_B$ of $\BB B$, and there are no non-identity morphisms in the discrete category $C$.  We leave it to the reader to check that \eqref{2cartfurthermoresub} is $2$-cartesian. 

It might be worth saying a few words about the (weak) action of $C$ on $X$ in Lemma~\ref{lem:quotients} that we get by following through all these isomorphisms.  Let's stick to the case where $X$ is a sheaf, for simplicity.  Use the notation $U(V) := \Hom_{\C}(V,U)$.  The stack-theoretic quotient $[X/A]$ can be described as follows:  An object of the category $[X/A]$ is an equivalence class $[Y,\{ Y_i \to Y \}, a,f]$ of triples $(Y,\{ Y_i \to Y \}, a,f)$ where: \begin{enumerate} \item $Y$ is an object of $\C$ and $\{ Y_i \to Y \}$ is a cover of $Y$.  Set $Y_{ij} := Y_i \times_Y Y_j$, $Y_{ijk} := Y_i \times_Y Y_j \times_Y Y_k$.  \item \label{transitionfunctions} $a = (a_{ij}) \in \prod_{i,j} A(Y_{ij})$ is a set of \emph{transition functions} satisfying the \emph{cocycle condition} $a_{ij}a_{jk} = a_{ik}$ in $A(Y_{ijk})$ for any triple of indices $i,j,k$ (dropping notation for restriction along the projections from $Y_{ijk}$ to $Y_{ij}$, $Y_{jk}$, $Y_{ik}$). \item \label{localmaps} $f = (f_i) \in \prod_i X(Y_i)$ are \emph{local maps} satisfying the \emph{gluing condition}:  $a_{ij} \cdot f_j  = f_i$ in $X(Y_{ij})$ for any pair of indices $i,j$.  (Again, notation for various restriction maps is dropped.  The $\cdot$ is the action of $A(Y_{ij})$ on $X(Y_{ij})$ induced by the action of $B$ on $X$ and the map $A \to B$.) \end{enumerate}  The equivalence relation on these triples is the smallest one such that: \begin{enumerate}  \item Any triple $(Y,\{ Y_i \to Y \}, a,f)$ is equivalent to the triple obtained from it by passing to a refinement of the cover $\{ Y_i \to Y \}$ and restricting the other data to the refinement. \item  Any triple $(Y,\{ Y_i \to Y \}, a,f)$ is equivalent to the triple obtained from it by changing $a$ and $f$ by a coboundary---i.e.\ replacing $a=(a_{ij})$ with $(a_i a_{ij} a_j^{-1})$ and $(f_i)$ with $(a_i \cdot f_i)$ for some $(a_i) \in \prod_i A(Y_i)$. \end{enumerate}  Morphisms in $[X/A]$ are defined in an evident manner that we leave to the reader to write out in detail.  The structure functor from $[X/A]$ to $\C$ is given by $[Y,\{ Y_i \to Y \}, a,f] \mapsto Y$, so the object $[Y,\{ Y_i \to Y \}, a,f]$ of $[X/A]$ lies in the fiber category $[X/A](Y)$.

Let us explain the meaning of this formal mess to the geometrically-minded reader.  The $a_{ij}$ are supposed to be the transition functions for an $A$-torsor $P \to Y$ equipped with a trivialization $\phi_i : P \times_Y Y_i \cong Y_i \times A$ over each $Y_i$.  (But note that, in our arbitrary category $\C$, there might not actually \emph{be} such a torsor, so we have to think in this more abstract manner!)  The $f_i$ are supposed to be the maps $f_i : X(Y_i)$ corresponding to the $A$-equivariant maps $f \times_Y Y_i : P \times_Y Y_i \to X$ under the trivialization $\phi_i$, where $f : P \to X$ is an $A$-equivariant map.  The equivalence relation on triples is there to remove the artificial choice of a local trivialization of $P \to Y$.

To define the $C$ action on $[X/A]$ suppose we have $[Y,\{ Y_i \to Y \}, a,f] \in [X/A](Y)$ and $c \in C(Y)$.  We define a new object $$c \cdot [Y,\{ Y_i \to Y \}, a,f] = [Y, \{ Y_i \to Y \}, c \cdot a, c \cdot f]  \in [X/A](Y)$$ as follows:  First of all, after possibly passing to a different representative of the equivalence class $[Y,\{ Y_i \to Y \}, a,f]$, we can assume the cover $ \{ Y_i \to Y \}$ is fine enough that there are lifts $b_i \in B(Y_i)$ of $c|Y_i \in C(Y_i)$ (because $B \to C$ is a surjection of sheaves).  We then define $c \cdot f$ by replacing $f_i$ with $b_i \cdot f_i$ and $c \cdot a$ by replacing $a_{ij}$ with $b_i a_{ij} b_j^{-1}$.  Note that $c \cdot a$ and $c \cdot f$ satisfy the cocycle and gluing conditions.  Also note that the equivalence class $[Y, \{ Y_i \to Y \}, c \cdot a, c \cdot f]$ does not actually depend on the \emph{choice} of liftings $b_i$ because if $(b_i) \in \prod_i B(Y_i)$ is another choice of liftings, then $(Y, \{ Y_i \to Y \}, (b_ia_{ij}b_j^{-1}))$ is obtained by changing $(Y, \{ Y_i \to Y \}, (b_i'a_{ij}(b_j')^{-1}))$ by the coboundary $(b_i(b_i')^{-1}) \in \prod_i A(Y_i)$.  

This completes the description of the $C$ action on the level of objects, but to really prove anything about it, one must also explain how it behaves on maps, and then one must explain the $\alpha$ structure maps for this weak action, etc.  Hopefully this explains why our proof goes the way it goes.

\end{document}